\documentclass[12pt,reqno]{amsart}
\usepackage{amsmath}
\usepackage{amsfonts}
\usepackage{amssymb}
\usepackage{amsthm}
\usepackage{amscd}
\usepackage{amsbsy}
\usepackage{fancyhdr}
\usepackage[usenames,dvipsnames,svgnames,x11names,hyperref]{xcolor}
\usepackage{geometry}
\usepackage{graphicx}
\usepackage{hyperref}
\usepackage{showidx}
\hypersetup{backref=true,
pagebackref=true,
hyperindex=true,
colorlinks=true,
breaklinks=true,
urlcolor=NavyBlue,
linkcolor=Fuchsia,
bookmarks=true,
bookmarksopen=false,
filecolor=black,
citecolor=ForestGreen,
linkbordercolor=red
}
\newtheorem{theorem}{Theorem}[section]
\theoremstyle{plain}

\newtheorem{corollary}{Corollary}[section]

\newtheorem{lemma}{Lemma}[section]

\newtheorem{remark}{Remark}[section]

\numberwithin{equation}{section}
\newtheorem{theorema}{Theorem}[section]

\allowdisplaybreaks
\AtBeginDocument{\hypersetup{pdfborder={0 0 0.01}}}
\begin{document}
\title[Sharp Stability of the CKN and Weighted Poincar\'e Inequalities]{Caffarelli-Kohn-Nirenberg and Weighted Gaussian Poincar\'e Inequalities: a complete characterization of sharp $L^2$ stability and $L^p$ extensions}
\author{Anh Xuan Do}
\address{Anh Xuan Do: Department of Mathematics, University of Connecticut, Storrs, CT
06269, USA}
\email{anh.do@uconn.edu}
\author{Nguyen Lam}
\address{Nguyen Lam: School of Science and the Environment, Grenfell Campus, Memorial
University of Newfoundland, Corner Brook, NL A2H5G4, Canada}
\email{nlam@mun.ca}
\author{Guozhen Lu}
\address{Guozhen Lu: Department of Mathematics, University of Connecticut, Storrs, CT
06269, USA}
\email{guozhen.lu@uconn.edu}
\author{Van Hoang Nguyen}
\address{Van Hoang Nguyen: Department of Mathematics, FPT University, Ha Noi, Vietnam}
\email{vanhoang0610@yahoo.com,~hoangnv47@fe.edu.vn}
\date{\today}

\begin{abstract}
We introduce a new family of weighted Gaussian $L^2$-Poincaré-type inequalities with explicit sharp constants, optimizers, and corresponding sharp $L^2$-gradient stability estimates. This family substantially extends the classical Gaussian Poincaré inequality. Owing to the singular nature of the weights involved, standard approaches to classical Gaussian Poincaré inequalities do not apply. To overcome this difficulty, we develop a new method based on generalized Laguerre polynomial expansions, spherical harmonic decompositions, and a Kelvin-type transform.

As an application, we completely characterize the stability of the $L^2$-Caffarelli--Kohn--Nirenberg (CKN) inequalities by establishing sharp stability estimates, together with the stability of the stability inequality results, throughout the entire parameter range. Previous results were available only in a few special cases. We further establish weighted $L^p$
-Poincaré inequalities for all $p>1$, and derive stability estimates for the $L^p$-CKN  inequalities for $p\geq2$ throughout the full parameter regime in which sharp constants and optimizers are known. In contrast, earlier $L^p$
 results were restricted to highly limited parameter ranges.

\end{abstract}
\maketitle
\section{Introduction}

In \cite{BL85}, Brezis and Lieb raised a fundamental question regarding the sharp Sobolev inequality in $W^{1,2}(\mathbb{R}^N)$: can the \emph{deficit}--namely, the difference between the energy and the optimal constant times the critical $L^{2^*}$-norm be bounded below by a positive quantity comparable to the square of the distance from $u$ to the manifold of optimizers? 
This observation gave rise to a broad and influential research program on the stability of functional and geometric inequalities. Understanding such stability is crucial, as it not only refines classical sharp inequalities by revealing how close a function is to achieving equality but also provides deep insights into the underlying geometry, rigidity phenomena, and concentration behavior of extremal sequences. Over the past decades, this perspective has become a central theme in analysis, inspiring numerous quantitative results and new techniques across different settings. For interested readers, we refer to \cite{BWW03, BGKM25, BGKM25b, BDNN20, BFR24, CF13, CFW13, 
CLT23, CFMP09, DoLL26, DN25, DN26, FIL16, FJ1, FJ2, FN19, FZ22, LLR26, LW99, VHN16, 
VHN19}; this list is far from exhaustive.

Brezis and Lieb's question was solved affirmatively by Bianchi and Egnell \cite{BE91}, who established the stability estimate
\begin{equation}\label{StabilityS}
\int_{\mathbb{R}^{N}} |\nabla u|^{2}\,dx
- S_{N}\!\left(
\int_{\mathbb{R}^{N}} |u|^{\frac{2N}{N-2}}\,dx
\right)^{\!\frac{N-2}{N}}
\ge c_{BE}\!
\inf_{U \in E_{Sob}} \!
\int_{\mathbb{R}^{N}} |\nabla(u - U)|^{2}\,dx,
\end{equation}
where $S_{N}$ is the optimal Sobolev constant, $E_{Sob}$ denotes the set of extremal functions, and $c_{BE}>0$ is a universal stability constant. This result shows that the Sobolev deficit controls the distance in the gradient norm to the manifold of optimizers and is optimal in the underlying metrics and exponents.

Despite extensive research, the value and attainability of stability constants have long remained elusive. Recently, Dolbeault, Esteban, Figalli, Frank, and Loss \cite{DEFFL} initiated a systematic analysis of the Bianchi–Egnell constant, deriving sharp lower bounds for $c_{BE}$ as $N \to \infty$ and establishing global stability for the Gaussian log-Sobolev inequality via a gradient flow approach. More recently, Chen, Lu and Tang \cite{CLT24, CLT242, CLT243} obtained explicit lower bounds for the stability constants in Hardy–Littlewood–Sobolev and higher-order fractional Sobolev inequalities, leading to global stability for Beckner’s log-Sobolev inequality on the sphere. In a related development, the authors in \cite{CLTW} established optimal stability of the Sobolev inequality on the Heisenberg group. Since classical rearrangement and flow arguments fail in that setting, they introduced a novel approach based on the CR Yamabe flow to pass from local to global stability.

The original motivation of our paper is to study sharp stability estimates for the following $L^{2}$-Caffarelli–Kohn–Nirenberg (CKN) inequality:
\begin{equation}
\left( \int_{\mathbb{R}^{N}} \frac{|\nabla u|^{2}}{|x|^{2b}}\,dx \right)^{\frac{1}{2}}
\left( \int_{\mathbb{R}^{N}} \frac{|u|^{2}}{|x|^{2a}}\,dx \right)^{\frac{1}{2}}
\geq C(N,a,b)
\left( \int_{\mathbb{R}^{N}} \frac{|u|^{2}}{|x|^{a+b+1}}\,dx \right),
\quad
u \in C_{0}^{\infty}(\mathbb{R}^{N} \setminus \{0\}).
\label{CKN}
\end{equation}
We note that the power $a+b+1$ in the right-hand integral is essential: it is precisely the exponent dictated by the scaling invariance of \eqref{CKN}, and any other power would break this invariance and prevent the inequality from holding with a uniform constant. 

This family \eqref{CKN} encompasses several celebrated inequalities, such as the Heisenberg uncertainty principle (corresponding to $a=-1$, $b=0$), the hydrogen uncertainty principle ($a=b=0$), and the classical Hardy inequality ($a=1$, $b=0$), each of which plays a fundamental role in quantum mechanics. Moreover, \eqref{CKN} is a special case of the more general Caffarelli–Kohn–Nirenberg inequality, introduced by Caffarelli, Kohn, and Nirenberg in \cite{CKN}. Remarkably, many classical functional inequalities—including the Gagliardo–Nirenberg, Sobolev, Hardy–Sobolev, and Nash inequalities—can be recovered within the broader CKN framework. The literature on this topic is extensive and continually expanding; see, for example, \cite{BCFGG20, CKN, CFL23, DY23, DDDL23, DDLL25, DELT09, Dong18, DN23, DN26, Fly20, FLL21, FLL22, GHS25, HT25, KRS22, LL17, VHN15, VHN21, VHN22}, to name just a few.

The optimal constant $C(N,a,b) > 0$ associated with the $L^{2}$-CKN inequality \eqref{CKN} was first investigated by Costa \cite{Cos08} in a restricted parameter regime via the expanding-the-square method. This analysis was subsequently extended to the full admissible range of parameters by Catrina and Costa \cite{CC09}, using spherical harmonic decomposition and a Kelvin-type transform. More recently, alternative direct proofs of the sharp constant $C(N,a,b)$ valid over the entire parameter range, along with a complete characterization of all optimizers, have been established in \cite{CFL21}.
More precisely, the main findings in \cite{CC09, CFL21, Cos08} can be summarized as follows:

\begin{theorema}
\label{TA}
For all admissible parameters $(a,b)$, one has
\[
C(N,a,b)
= \max \left\{
\frac{|N-(a+b+1)|}{2}, \,
\frac{|N-(3b - a + 3)|}{2}
\right\}.
\]
Moreover, the sharp constant and the corresponding extremal functions are characterized as follows:

\begin{enumerate}
\item If $(a,b) \in \mathcal{A}$, then
\[
C(N,a,b) = \frac{|N-(a+b+1)|}{2},
\]
and equality in \eqref{CKN} is attained precisely by functions of the form
\[
u(x) = D \exp\!\left( \frac{t |x|^{\,b+1-a}}{b+1-a} \right),
\]
where $D \neq 0$, $t < 0$ in $\mathcal{A}_{1}$, and $t > 0$ in $\mathcal{A}_{2}$.

\item If $(a,b) \in \mathcal{B}$, then
\[
C(N,a,b) = \frac{|N-(3b - a + 3)|}{2},
\]
and equality in \eqref{CKN} is achieved by functions of the form
\[
u(x) = D\, |x|^{\,2(b+1)-N} \exp\!\left( \frac{t |x|^{\,b+1-a}}{b+1-a} \right),
\]
where $D \neq 0$, $t > 0$ in $\mathcal{B}_{1}$, and $t < 0$ in $\mathcal{B}_{2}$.

\item The only parameter values for which the best constant is not achieved correspond to the line $a=b+1$, where
\[
C(N,b+1,b) = \frac{|N - 2(b+1)|}{2}.
\]
\end{enumerate}
\end{theorema}

The sets involved above are defined by
\[
\left\{
\begin{array}{ll}
\mathcal{A}_{1} := \{ (a,b) \mid b+1-a > 0,\; b \le (N-2)/2 \}, \\[5pt]
\mathcal{A}_{2} := \{ (a,b) \mid b+1-a < 0,\; b \ge (N-2)/2 \}, \\[5pt]
\mathcal{A} := \mathcal{A}_{1} \cup \mathcal{A}_{2}, \\[5pt]
\mathcal{B}_{1} := \{ (a,b) \mid b+1-a < 0,\; b \le (N-2)/2 \}, \\[5pt]
\mathcal{B}_{2} := \{ (a,b) \mid b+1-a > 0,\; b \ge (N-2)/2 \}, \\[5pt]
\mathcal{B} := \mathcal{B}_{1} \cup \mathcal{B}_{2}.
\end{array}
\right.
\]

The stability problem for the $L^2$-CKN inequality \eqref{CKN} was first addressed in \cite{MV21} for the particular case $a = -1$ and $b = 0$, corresponding to the Heisenberg Uncertainty Principle (HUP). In their study, McCurdy and Venkatraman employed the concentration–compactness framework to establish the existence of universal constants $C_{1} > 0$ and $C_{2}(N) > 0$ satisfying
\[
\delta(u) \ge C_{1}\!\left(\int_{\mathbb{R}^{N}} |u(x)|^{2}\,dx\right)\! d^{2}(u, E_{HUP}) 
+ C_{2}(N)\, d^{4}(u, E_{HUP}),
\]
where the HUP deficit is given by
\[
\delta(u)
:= \left(\int_{\mathbb{R}^{N}} |\nabla u|^{2}\,dx\right)
\left(\int_{\mathbb{R}^{N}} |x|^{2} |u|^{2}\,dx\right)
- \frac{N^{2}}{4}\!\left(\int_{\mathbb{R}^{N}} |u|^{2}\,dx\right)^{2}.
\]
Here,
\[
E_{HUP} := \{\alpha e^{-\beta |x|^{2}} : \alpha \in \mathbb{R},\, \beta > 0\}
\]
denotes the family of Gaussian optimizers of the HUP, and the distance to a set $A$ in $L^{2}(\mathbb{R}^{N})$ is defined by
\[
d(u, A) := \inf_{v \in A} \|u - v\|_{2}.
\]
This result shows that a small deficit $\delta
(u) \approx 0$ implies that $u$ lies close, in the $L^{2}$ sense, to a 
Gaussian of the form $\alpha e^{-\beta |x|^{2}}$. A constructive approach was later developed by Fathi \cite{Fat21}, who obtained the same estimate with explicit constants $C_{1} = \tfrac{1}{4}$ and $C_{2} = \tfrac{1}{16}$, though these are not optimal. Eventually, the authors in \cite{CFLL24} established a sharp and fully quantitative stability estimate for the HUP, summarized below.

\begin{theorema}
\label{TB}
For all $u \in C_{0}^{\infty}(\mathbb{R}^{N})$, one has
\[
\bar{\delta}(u)
:= \left(\int_{\mathbb{R}^{N}} |\nabla u|^{2}\,dx\right)^{1/2}
\left(\int_{\mathbb{R}^{N}} |x|^{2} |u|^{2}\,dx\right)^{1/2}
- \frac{N}{2} \int_{\mathbb{R}^{N}} |u|^{2}\,dx
\ge d^{2}(u, E_{HUP}).
\]
Equality holds for nontrivial functions $u \notin E_{HUP}$, and the bound is sharp. As a consequence, \[
\delta(u) \ge N\!\left(\int_{\mathbb{R}^{N}} |u(x)|^{2}\,dx\right)\! d^{2}(u, E_{HUP}) 
+ d^{4}(u, E_{HUP}),
\] and the estimate is optimal.
\end{theorema}

See also \cite{LLR25} for an improved version, and \cite{LLLS25} for the HUP and its stability on the half-space and orthants.

The approach in \cite{CFLL24} is based on two fundamental ideas. The first is the explicit computation of the remainder term in the Heisenberg uncertainty principle (HUP), and the second is the use of a sharp Gaussian-type Poincaré inequality, which states that for any $\lambda \ne 0$,
\begin{equation}\label{keyPoin}
\int_{\mathbb{R}^{N}} |\nabla u|^{2} e^{-\frac{|x|^{2}}{2|\lambda|^{2}}}\,dx
\ge \frac{1}{|\lambda|^{2}}
\inf_{c \in \mathbb{R}} \int_{\mathbb{R}^{N}} |u - c|^{2} e^{-\frac{|x|^{2}}{2|\lambda|^{2}}}\,dx.
\end{equation}
This Gaussian-type Poincaré inequality constitutes the main analytical tool in \cite{CFLL24}.

Building on this idea, the authors established in  \cite{CFLL24} the stability of the $L^{2}$-CKN inequality \eqref{CKN} under the restricted 

conditions
\[
0 \leq b < \frac{N-2}{2}, \quad a \leq \frac{Nb}{N-2}, \quad \text{and} \quad a + b + 1 = \frac{2bN}{N-2}.
\]
These constraints follow from the necessity of applying a Poincaré inequality for log-concave measures, which is used in place of \eqref{keyPoin}. Using a similar approach,    a stability estimate for the $L^{2}$-CKN inequality \eqref{CKN} was derived in \cite{DFLL23} under the assumptions
\[
\frac{N-2}{2} < b \leq N - 2, \quad \text{and} \quad N(b - a + 3) = 2(3b - a + 3).
\]
Once again, these conditions arise from the requirement to apply a Poincaré inequality for log-concave measures. Recently, in \cite{CT25}, the authors, following the same ideas from \cite{CFLL24} and \cite{DFLL23},  extended some stability estimates for the $L^{2}$-CKN inequality \eqref{CKN} under the assumptions
\[
0 \leq b < \frac{N-2}{2}, \quad \text{and} \quad a = \frac{Nb}{N-2}.
\]
It is worth noting that, in addition to the strict parameter conditions imposed in the aforementioned works, information on the sharp stability constants is also lacking due to the reliance on the Poincaré inequality for log-concave measures.

The first main objective of this paper is to fully establish the stability 
of the $L^{2}$-CKN inequality by deriving stability estimates, with explicit 
sharp stability constants, {\bf for all admissible parameters}, thus removing the 
constraints imposed in \cite{CFLL24, CT25, DFLL23}. More precisely, define the \(L^2\)-CKN deficit by
\begin{equation}\label{L2CKNdeficit}
    \delta_2(u) := \left( \int_{\mathbb{R}^N} \frac{|\nabla u|^2}{|x|^{2b}} \, dx \right)^{\frac{1}{2}} 
    \left( \int_{\mathbb{R}^N} \frac{|u|^2}{|x|^{2a}} \, dx \right)^{\frac{1}{2}} 
    - C(N,a,b) \int_{\mathbb{R}^N} \frac{|u|^2}{|x|^{a + b + 1}} \, dx \geq 0.
\end{equation}
We will then establish the following sharp stability estimates.

\begin{theorem}\label{staL2CKN} Let $N \geq 1$, and $a,b \in \mathbb{R}$. Define
\[
C_{PI}(a,b,N) = \min\left\{ 2(1+b-a),\, \sqrt{(N-2-2b)^2 + 4(N-1)} - (N-2-2b) \right\}.
\] Then the following stability estimates hold, and in each case the given constant is sharp and can be attained:
\begin{itemize}
    \item[(i)] If $1 + b - a > 0$ and $b \leq \frac{N-2}{2}$, then
    $$\delta_2 (u) \geq \frac{C_{PI}(a,b,N)}{2}\inf_{c \in \mathbb{R}, \lambda>0}\displaystyle\int_{\mathbb{R}^N}\dfrac{\left|u-ce^{-\frac{\lambda|x|^{b+1-a}}{(b+1-a)}}\right|^2}{|x|^{1+b+a}}dx.$$
    \item[(ii)] If $1 + b - a < 0$ and $b \geq \frac{N-2}{2}$,
    \[
    \delta_2 (u) \geq \frac{C_{PI}(N-a, N-2-b, N)}{2} \inf_{c \in \mathbb{R}, \lambda>0}\displaystyle\int_{\mathbb{R}^N}\dfrac{\left|u-ce^{\frac{\lambda|x|^{b+1-a}}{(b+1-a)}}\right|^2}{|x|^{1+b+a}}dx.
    \]
    \item[(iii)] If $1 + b - a < 0$ and $b \leq \frac{N-2}{2}$,
    \[
    \delta_2 (u) \geq \frac{C_{PI}(-a+2b+2, b, N)}{2}\inf_{c \in \mathbb{R}, \lambda>0}\displaystyle\int_{\mathbb{R}^N}\dfrac{\left|u-ce^{\frac{\lambda|x|^{b+1-a}}{(b+1-a)}}|x|^{2b+2-N}\right|^2}{|x|^{1+b+a}}dx.
    \]
    \item[(iv)] If $1 + b - a > 0$ and $b \geq \frac{N-2}{2}$,
    \[
    \delta_2 (u) \geq  \frac{ C_{PI}(N+a-2b-2, N-b-2, N)}{2}\inf_{c \in \mathbb{R}, \lambda>0}\displaystyle\int_{\mathbb{R}^N}\dfrac{\left|u-ce^{-\frac{\lambda|x|^{b+1-a}}{(b+1-a)}}|x|^{2b+2-N}\right|^2}{|x|^{1+b+a}}dx.
    \]
 \end{itemize} 
 \end{theorem}
\begin{remark}
 As stated in Theorem~A, the case $1+b-a=0$ corresponds to the weighted Hardy inequality, and it is well known that the sharp constant cannot be attained. 
  
\end{remark}
To prove Theorem~\ref{staL2CKN}, we adopt the strategy developed in~\cite{CFLL24}. For this purpose, we will first establish sharp weighted $L^2$-Poincar\'e inequalities with Gaussian-type measures, which constitute our second main result. These inequalities form  a much broader generalization of the classical Gaussian Poincar\'e inequalities and are of independent interests.

\begin{theorem}\label{TPoincare}
Let \(N \geq 1\), and \(a,b \in \mathbb{R}\). For \(u \in C_0^\infty(\mathbb{R}^N)\),  the following weighted Poincar\'e inequalities hold, and in each case the given constant is sharp and can be attained:
\begin{itemize}
    \item[(i)] If \(1 + b - a > 0\) and \(b \leq \frac{N-2}{2}\),
    \[
    \int_{\mathbb{R}^N} |\nabla u|^2 \frac{e^{-\frac{2|x|^{1+b-a}}{1+b-a}}}{|x|^{2b}} \, dx
    \geq C_{PI}(a,b,N) \inf_{c \in \mathbb{R}} \int_{\mathbb{R}^N} |u - c|^2 \frac{e^{-\frac{2|x|^{1+b-a}}{1+b-a}}}{|x|^{1+b+a}} \, dx.
    \]
    \item[(ii)] If \(1 + b - a < 0\) and \(b \geq \frac{N-2}{2}\),
    \[
    \int_{\mathbb{R}^N} |\nabla u|^2 \frac{e^{\frac{2|x|^{1+b-a}}{1+b-a}}}{|x|^{2b}} \, dx
    \geq C_{PI}(N-a, N-2-b, N) \inf_{c \in \mathbb{R}} \int_{\mathbb{R}^N} |u - c|^2 \frac{e^{\frac{2|x|^{1+b-a}}{1+b-a}}}{|x|^{1+b+a}} \, dx.
    \]
    \item[(iii)] If \(1 + b - a < 0\) and \(b \leq \frac{N-2}{2}\),
    \[
    \int_{\mathbb{R}^N} |\nabla u|^2 \frac{e^{\frac{2|x|^{1+b-a}}{1+b-a}}}{|x|^{2N - 2b - 4}} \, dx
    \geq C_{PI}(-a+2b+2, b, N) \inf_{c \in \mathbb{R}} \int_{\mathbb{R}^N} |u - c|^2 \frac{e^{\frac{2|x|^{1+b-a}}{1+b-a}}}{|x|^{2N + a - 3b - 3}} \, dx.
    \]
    \item[(iv)] If \(1 + b - a > 0\) and \(b \geq \frac{N-2}{2}\),
    \[
    \int_{\mathbb{R}^N} |\nabla u|^2 \frac{e^{-\frac{2|x|^{1+b-a}}{1+b-a}}}{|x|^{2N - 2b - 4}} \, dx
    \geq C_{PI}(N+a-2b-2, N-b-2, N) \inf_{c \in \mathbb{R}} \int_{\mathbb{R}^N} |u - c|^2 \frac{e^{-\frac{2|x|^{1+b-a}}{1+b-a}}}{|x|^{2N + a - 3b - 3}} \, dx.
    \]
\end{itemize}
Here, $C_{PI}$ is defined as in Theorem~\ref{staL2CKN}.
\end{theorem}

\begin{remark}
In Section~2, we completely characterize the explicit optimizers and then establish sharp stability estimates for the inequalities in Theorem~\ref{staL2CKN} and Theorem~\ref{TPoincare}.
\end{remark}

We note that in the special case $a = -1$ and $b = 0$, our weighted Poincaré inequality reduces to  
\begin{equation}\label{classicalPoin1}
\int_{\mathbb{R}^{N}} |\nabla u|^{2} e^{-|x|^{2}}\,dx
\geq 2 \inf_{c \in \mathbb{R}} \int_{\mathbb{R}^{N}} |u - c|^{2} e^{-|x|^{2}}\,dx,
\end{equation}
which is equivalent to the following classical Poincaré inequality with respect to the Gaussian measure:
\begin{equation}\label{classicalPoin}
\int_{\mathbb{R}^{N}} |\nabla u|^{2} e^{-\frac{|x|^{2}}{2}}\,dx
\geq \inf_{c \in \mathbb{R}} \int_{\mathbb{R}^{N}} |u - c|^{2} e^{-\frac{|x|^{2}}{2}}\,dx.
\end{equation}

It is well-known that $1$ is the optimal constant in the Gaussian Poincaré inequality \eqref{classicalPoin} (and therefore $2$ is sharp in \eqref{classicalPoin1}), which perfectly aligns with our general result. It is also worth noting that this inequality \eqref{classicalPoin} has been established through various classical methods, including the spectral analysis of the Ornstein-Uhlenbeck operator via Hermite polynomial decomposition, probabilistic techniques involving concentration of measure, semigroup interpolation methods, and by linearizing the logarithmic Sobolev inequality; see, for example, \cite{BGL14, GRO75, GRO75b}.

However, since our Poincaré inequality involves singular weights of the form $|x|^\alpha$, these traditional methods no longer apply directly. The singular nature of the weights breaks many of the symmetries and functional frameworks that classical proofs rely on. Consequently, substantial new techniques are required.

In this paper, to prove Theorem \ref{TPoincare}, we need to establish the following Poincaré inequalities stated in Theorem \ref{Tkey} and Theorem \ref{Tkey1} below.
These one-dimensional Gaussian type weighted Poincar\'e inequalities, which play a crucial role in our proof of Theorem \ref{TPoincare},  have not appeared in the literature.   The proofs of these inequalities involve innovative  ideas of using  the generalized Laguerre polynomial expansion combining with classical tools such as  spherical harmonics decomposition, a Kelvin-type transform, and a judicious change of variables. These techniques allow us to carefully analyze the weighted structure and establish sharp inequalities in this more complex setting. More precisely, this approach enables us to determine the sharp constants and optimizers, and then derive the associated stability estimates.

\begin{theorem}\label{Tkey}
    Let \(\alpha > 0\). For any function \(u \in W^{1,2}((0, \infty), e^{-\frac{s^2}{2}} s^{\alpha - 1} ds)\), the following sharp inequality holds:
    \[
    \int_0^\infty |u'(s)|^2 e^{-\frac{s^2}{2}} s^{\alpha - 1} \, ds 
    \geq 2 \inf_{c \in \mathbb{R}} \int_0^\infty |u(s) - c|^2 e^{-\frac{s^2}{2}} s^{\alpha - 1} \, ds.
    \]
    The constant \(2\) is optimal and attained precisely by functions of the form
    \[
    u(s) = a_0 + a_1 s^2,
    \]
    where \(a_0, a_1 \in \mathbb{R}\).
\end{theorem}

\begin{theorem}\label{Tkey1}
    Let $A \geq 2$ and $C > 0$. Then
\begin{align*}
    &\int_0^\infty \left[ (u'(s))^2 + C \left( \frac{u(s)}{s} \right)^2 \right] e^{-s^2/2} s^{A-1} ds \nonumber\\
    &\geq \frac{ - (A - 2) + \sqrt{(A - 2)^2 + 4 C} }{2} \int_0^\infty u^2(s) e^{-s^2/2} s^{A-1} ds.
\end{align*}  The equality holds if and only if $u(s)=c s^{\frac{-(A-2)+\sqrt{(A-2)^2+4C}}{2}}$. 
\end{theorem}

A direct consequence of Theorem \ref{Tkey} (and Theorem \ref{lemma 0}) is the following sharp Gaussian Poincaré inequality for radial functions, which is also of independent interest:

\begin{theorem}\label{TGPoinRadial}
    Let \(N \geq 1\). For any radial function \(u \in W^{1,2}(\mathbb{R}^N, e^{-\frac{|x|^2}{2}} dx)\), the inequality
    \[
    \int_{\mathbb{R}^N} |\nabla u|^2 e^{-\frac{|x|^2}{2}} \, dx 
    \geq 2 \inf_{c \in \mathbb{R}} \int_{\mathbb{R}^N} |u(x) - c|^2 e^{-\frac{|x|^2}{2}} \, dx
    \]
    holds with the sharp constant \(2\). This constant is attained exactly by functions of the form
    \[
    u(x) = a_0 + a_1 |x|^2,
    \]
    with \(a_0, a_1 \in \mathbb{R}\). Moreover, the following stability result holds:
    \begin{align*}
        &\int_{\mathbb{R}^N} |\nabla u|^2 e^{-\frac{|x|^2}{2}} \, dx 
    - 2 \inf_{c \in \mathbb{R}} \int_{\mathbb{R}^N} |u(x) - c|^2 e^{-\frac{|x|^2}{2}} \, dx \\
    &\geq \frac{1}{2} \inf_{d \in \mathbb{R}}\int_{\mathbb{R}^N} |\nabla \left(u-d|x|^2 \right)|^2 e^{-\frac{|x|^2}{2}} \, dx. 
    \end{align*}
\end{theorem}

Our result shows that restricting to radial functions improves the sharp constant in the classical Gaussian Poincaré inequality, increasing it from $1$ to $2$. Moreover, we establish a gradient stability estimate for this sharp inequality.

Once the sharp weighted \(L^2\)-Poincar\'e inequalities with Gaussian-type measures are established (cf. Theorem \ref{TPoincare}), we will use them to prove Theorem \ref{staL2CKN} and derive stability estimates for the \(L^2\)-CKN inequality with exact optimal constants, covering the full range of parameters \((a,b)\). 

Next, inspired by the aforementioned stability results for various functional and geometric inequalities, we also investigate the stability of the weighted Poincar\'e inequalities (Theorem~\ref{TPoincare}) in Section~2. These results yield sharp improved $L^2$-stability estimates for the $L^2$-CKN inequality, namely the stability version of the sharp inequalities in Theorem~\ref{staL2CKN}. Our approach refines Theorems~\ref{Tkey} and~\ref{Tkey1}, combining them with spherical harmonics decomposition to derive $L^2$-stability for the weighted Gaussian Poincar\'e inequalities, and establishes a general principle equating $L^2$-stability with gradient stability for such inequalities---which holds intrinsic value.

As a special case of our main results, when $a=b=0$, we obtain the following sharp stability result for the Hydrogen Uncertainty Principle, which appears to be still absent in the literature.

\begin{corollary}[The stability of the Hydrogen Uncertainty Principle] Let \(N \geq 2\). Then the following inequality holds:
    \begin{align*}
       & \left( \int_{\mathbb{R}^N} |\nabla u|^2 \, dx \right)^{\frac{1}{2}} \left( \int_{\mathbb{R}^N} |u|^2 \, dx \right)^{\frac{1}{2}} 
       - \frac{N-1}{2} \int_{\mathbb{R}^N} \frac{|u|^2}{|x|} \, dx \\
       & \quad \geq \inf_{c \in \mathbb{R}, \lambda > 0} \int_{\mathbb{R}^N} \frac{\left| u - c e^{-\lambda |x|} \right|^2}{|x|} \, dx. 
    \end{align*}
    The constant $1$  on the right hand side is sharp and can be attained. Moreover, the following stability of the stability inequality holds:
    \begin{align*}
       & \left( \int_{\mathbb{R}^N} |\nabla u|^2 \, dx \right)^{\frac{1}{2}} \left( \int_{\mathbb{R}^N} |u|^2 \, dx \right)^{\frac{1}{2}} 
       - \frac{N-1}{2} \int_{\mathbb{R}^N} \frac{|u|^2}{|x|} \, dx \\
       &- \inf_{c \in \mathbb{R}, \lambda > 0} \int_{\mathbb{R}^N} \frac{\left| u - c e^{-\lambda |x|} \right|^2}{|x|} \, dx \\
       &\geq \inf_{c,d \in \mathbb{R}, \mathbf{a} \in \mathbb{R}^N, \lambda > 0} \int_{\mathbb{R}^N} \frac{\left| u - (c+d|x|+\mathbf{a}\cdot x) e^{-\lambda |x|} \right|^2}{|x|} \, dx.
    \end{align*}
\end{corollary}

In the same spirit, we will next consider the \(L^p\)-CKN inequality. In \cite{DFLL23}, the authors proved the following result: 

\begin{theorema}\label{TC}
Let \(p > 1\), and \(a,b \in \mathbb{R}\). There exists a constant \(C(p,N,a,b) > 0\) such that for all \(u \in C_0^\infty(\mathbb{R}^N \setminus \{0\})\), the inequality
\begin{equation}\label{LpCKN}
\left( \int_{\mathbb{R}^N} \frac{|\nabla u|^p}{|x|^{p b}} \, dx \right)^{\frac{1}{p}}
\left( \int_{\mathbb{R}^N} \frac{|u|^p}{|x|^{p a}} \, dx \right)^{\frac{p-1}{p}}
\geq C(p,N,a,b) \int_{\mathbb{R}^N} \frac{|u|^p}{|x|^{(p-1)a + b + 1}} \, dx
\end{equation}
holds. Moreover, when \(b + 1 - a > 0\) and \(b \leq \frac{N-p}{p}\), the optimal constant is
\[
C(p,N,a,b) = \frac{N - 1 - (p - 1)a - b}{p},
\]
and 
is attained uniquely by functions of the form
\[
u(x) = D \exp\left( \frac{t |x|^{b + 1 - a}}{b + 1 - a} \right), \quad t < 0.
\]
When \(b + 1 - a < 0\) and \(b \geq \frac{N-p}{p}\), the sharp constant is
\[
C(p,N,a,b) = \frac{1 + (p - 1)a + b - N}{p},
\]
and is 
attained uniquely by functions of the form
\[
u(x) = D \exp\left( \frac{t |x|^{b + 1 - a}}{b + 1 - a} \right), \quad t > 0.
\]
\end{theorema}
\begin{remark}
Determining the optimal constants and extremal functions for the inequality \eqref{LpCKN} remains an open problem in parameter regions \((a,b)\) other than those identified in the theorem above.
\end{remark}

Additionally, the authors in \cite{DFLL23} established the following stability estimate:
\begin{align}\label{stabilityLpCKN}
&\left( \int_{\mathbb{R}^N} \frac{|\nabla u|^p}{|x|^{p b}} \, dx \right)^{\frac{1}{p}} 
\left( \int_{\mathbb{R}^N} \frac{|u|^p}{|x|^{p a}} \, dx \right)^{\frac{p-1}{p}} 
- \frac{N - 1 - (p - 1)a - b}{p} \int_{\mathbb{R}^N} \frac{|u|^p}{|x|^{(p - 1) a + b + 1}} \, dx \nonumber \\
&\quad \gtrsim \inf_{c \in \mathbb{R},~\lambda > 0}\int_{\mathbb{R}^N} 
\frac{\left| u - c \exp\left(-\frac{\lambda |x|^{b + 1 - a}}{b + 1 - a}\right) \right|^p}
{|x|^{(p - 1) a + b + 1}} \, dx,
\end{align}
for parameters satisfying
\[
p \geq 2, \quad 0 \leq b < \frac{N-p}{p}, \quad a \leq \frac{N b}{N - p}, \quad \text{and} \quad (p - 1) a + b + 1 = \frac{p b N}{N - p}.
\]
The first condition ensures control over the remainder term in the \(L^p\)-CKN inequality, while the latter three conditions stem from the applicability of an \(L^p\)-Poincaré inequality with log-concave measures.

Motivated by this discussion, we next investigate weighted $L^{p}$-Poincar\'e inequalities with Gaussian-type measures. More precisely, in the $L^{p}$ case, we establish the following result:

\begin{theorem}\label{TLpPoinUnified}
Let \(p > 1\), and \(a,b \in \mathbb{R}\). There exists a constant \(C_{PI}(a,b,p,N) > 0\) such that for all \(u \in C_0^\infty(\mathbb{R}^N)\), the following inequalities hold:

\begin{enumerate}
    \item[(i)] If \(1 + b - a > 0\) and \(b \leq \frac{N - p}{p}\), then
    \[
        \int_{\mathbb{R}^N} 
        \frac{|\nabla u|^p}{|x|^{p b}} 
        e^{-\frac{p |x|^{1 + b - a}}{1 + b - a}} \, dx
        \geq 
        C_{PI}(a,b,p,N)
        \inf_{c \in \mathbb{R}}
        \int_{\mathbb{R}^N} 
        \frac{|u - c|^p}{|x|^{1 + b + (p - 1)a}} 
        e^{-\frac{p |x|^{1 + b - a}}{1 + b - a}} \, dx.
    \]

    \item[(ii)] If \(1 + b - a < 0\) and \(b \geq \frac{N - p}{p}\), then
    \[
        \int_{\mathbb{R}^N} 
        \frac{|\nabla u|^p}{|x|^{p b}} 
        e^{\frac{p |x|^{1 + b - a}}{1 + b - a}} \, dx
        \geq 
        C_{PI}(a,b,p,N)
        \inf_{c \in \mathbb{R}}
        \int_{\mathbb{R}^N} 
        \frac{|u - c|^p}{|x|^{1 + b + (p - 1)a}} 
        e^{\frac{p |x|^{1 + b - a}}{1 + b - a}} \, dx.
    \]

    \item[(iii)] If \(1 + b - a < 0\) and \(b \leq \frac{N - p}{p}\), then
    \[
        \int_{\mathbb{R}^N} 
        \frac{|\nabla u|^p}{|x|^{2N - p b - 2p}} 
        e^{\frac{p |x|^{1 + b - a}}{1 + b - a}} \, dx
        \geq 
        C_{PI}(a,b,p,N)
        \inf_{c \in \mathbb{R}}
        \int_{\mathbb{R}^N} 
        \frac{|u - c|^p}{|x|^{2N + (p - 1)a + (1 - 2p)b - 2p + 1}} 
        e^{\frac{p |x|^{1 + b - a}}{1 + b - a}} \, dx.
    \]

    \item[(iv)] If \(1 + b - a > 0\) and \(b \geq \frac{N - p}{p}\), then
    \[
        \int_{\mathbb{R}^N} 
        \frac{|\nabla u|^p}{|x|^{2N - p b - 2p}} 
        e^{-\frac{p |x|^{1 + b - a}}{1 + b - a}} \, dx
        \geq 
        C_{PI}(a,b,p,N)
        \inf_{c \in \mathbb{R}}
        \int_{\mathbb{R}^N} 
        \frac{|u - c|^p}{|x|^{2N + (p - 1)a + (1 - 2p)b - 2p + 1}} 
        e^{-\frac{p |x|^{1 + b - a}}{1 + b - a}} \, dx.
    \]
\end{enumerate}
\end{theorem}

Since we work in the \(L^p\) setting, the tools used in the \(L^2\) case are no longer available. To overcome this difficulty, our proof of Theorem \ref{TLpPoinUnified} relies on a careful radial--angular decomposition and a delicate change of variables, which reduces the radial part to a one-dimensional log-concave setting. The desired result then follows by a technically involved argument.

Using Theorem~\ref{TLpPoinUnified}, we next obtain stability estimates for the \(L^p\)-Caffarelli--Kohn--Nirenberg inequalities for \(p \geq 2\) in the full range of parameters for which sharp constants and optimizers are known.

\begin{theorem}\label{TLpCKN Stability}
Let \(p \geq 2\), and \(a,b \in \mathbb{R}\). There exists a constant \(C(a,b,p,N) > 0\) such that for all 
\(u \in C_0^\infty(\mathbb{R}^N \setminus \{0\}) \setminus \{0\}\), the following hold:

\begin{enumerate}
    \item[(i)] If \(1 + b - a > 0\) and \(b \leq \dfrac{N - p}{p}\), then
    \begin{align*}
        & \int_{\mathbb{R}^N} \frac{|\nabla u|^{p}}{|x|^{p b}} \, dx 
        + (p - 1) \int_{\mathbb{R}^N} \frac{|u|^{2}}{|x|^{2 a}} \, dx
        - \left(N - 1 - (p - 1)a - b\right)
        \int_{\mathbb{R}^N} \frac{|u|^{p}}{|x|^{(p-1)a + b + 1}} \, dx \\
        & \geq C(a,b,p,N)
        \inf_{c \in \mathbb{R}}
        \int_{\mathbb{R}^N}
        \frac{\big|u - c e^{-\frac{|x|^{1 + b - a}}{1 + b - a}}\big|^p}
        {|x|^{1 + b + (p - 1)a}} \, dx
    \end{align*}
    and 
    \begin{align*}
        & \left( \int_{\mathbb{R}^N} \frac{|\nabla u|^{p}}{|x|^{p b}} \, dx \right)^{\frac{1}{p}}
        \left( \int_{\mathbb{R}^N} \frac{|u|^{p}}{|x|^{p a}} \, dx \right)^{\frac{p - 1}{p}}
        - \frac{N - (p - 1)a - b - 1}{p}
        \int_{\mathbb{R}^N} \frac{|u|^{p}}{|x|^{(p - 1)a + b + 1}} \, dx \\
        & \geq \frac{C(a,b,p,N)}{p}
        \inf_{c \in \mathbb{R},\, \lambda > 0}
        \int_{\mathbb{R}^N}
        \frac{\big|u - c e^{-\frac{\lambda |x|^{1 + b - a}}{1 + b - a}}\big|^p}
        {|x|^{1 + b + (p - 1)a}} \, dx.
    \end{align*}

    \item[(ii)] If \(1 + b - a < 0\) and \(b \geq \dfrac{N - p}{p}\), then
    \begin{align*}
        & \int_{\mathbb{R}^N} \frac{|\nabla u|^{p}}{|x|^{p b}} \, dx 
        + (p - 1) \int_{\mathbb{R}^N} \frac{|u|^{2}}{|x|^{2 a}} \, dx
        - \left(1 + (p - 1)a + b - N\right)
        \int_{\mathbb{R}^N} \frac{|u|^{p}}{|x|^{(p-1)a + b + 1}} \, dx \\
        & \geq C(a,b,p,N)
        \inf_{c \in \mathbb{R}}
        \int_{\mathbb{R}^N}
        \frac{\big|u - c e^{\frac{|x|^{1 + b - a}}{1 + b - a}}\big|^p}
        {|x|^{1 + b + (p - 1)a}} \, dx
    \end{align*}
and
    \begin{align*}
        & \left( \int_{\mathbb{R}^N} \frac{|\nabla u|^{p}}{|x|^{p b}} \, dx \right)^{\frac{1}{p}}
        \left( \int_{\mathbb{R}^N} \frac{|u|^{p}}{|x|^{p a}} \, dx \right)^{\frac{p - 1}{p}}
        - \frac{1 + (p - 1)a + b - N}{p}
        \int_{\mathbb{R}^N} \frac{|u|^{p}}{|x|^{(p - 1)a + b + 1}} \, dx \\
        & \geq \frac{C(a,b,p,N)}{p}
        \inf_{c \in \mathbb{R},\, \lambda > 0}
        \int_{\mathbb{R}^N}
        \frac{\big|u - c e^{\frac{\lambda |x|^{1 + b - a}}{1 + b - a}}\big|^p}
        {|x|^{1 + b + (p - 1)a}} \, dx.
    \end{align*}
\end{enumerate}
\end{theorem}

The paper is organized as follows. In Section 2, we establish the sharp weighted Gaussian $L^2$-Poincar\'e inequalities and their stability estimates. In Section 3, we use these results to prove the stability estimates for the $L^2$-Caffarelli--Kohn--Nirenberg inequalities. In Section 4, we derive the corresponding weighted $L^p$-Poincar\'e inequalities. Finally, in Section 5, we obtain the stability estimates for the $L^p$-Caffarelli--Kohn--Nirenberg inequalities.
\section{Sharp Weighted $L^2$-Poincar\'{e} Inequalities With Gaussian Type Measures and their stability-Proofs of Theorems \ref{TPoincare}, \ref{Tkey} and \ref{Tkey1}}
The proofs of  Theorem \ref{TPoincare}, Theorem \ref{Tkey}, Theorem \ref{Tkey1} and their stability will be established via the following theorems.
\begin{theorem}\label{lemma 0}
    Let $\alpha > 0$. For any function $v \in W^{1,2}((0, \infty),e^{-\frac{s^2}{2}} s^{\alpha - 1}ds)$, we have
    $$\int_0^\infty |v'(s)|^2 e^{-s^2/2} s^{\alpha-1} ds \geq 2 \inf_{c \in \mathbb{R}} \int_0^\infty |v(s) - c|^2 e^{-s^2/2} s^{\alpha-1} ds.$$ The constant $2$ is sharp and can be attained by the functions $$v(s) = a_0 + a_1 s^2,$$
    with $a_0, a_1 \in \mathbb{R}$. Moreover, 
    \begin{align*}
        \int_0^\infty |v'(s)|^2 e^{-s^2/2} s^{\alpha-1} ds &- 2 \inf_{c \in \mathbb{R}} \int_0^\infty |v(s) - c|^2 e^{-s^2/2} s^{\alpha-1} ds\\
        &\geq\frac{1}{2}\inf_{d \in \mathbb{R}}\int_0^\infty |v'(s)-ds|^2 e^{-s^2/2} s^{\alpha-1} ds\\
        &\geq 2\inf_{c,d \in \mathbb{R}}\int_0^\infty |v(s)-c-ds^2|^2 e^{-s^2/2} s^{\alpha-1} ds.
    \end{align*}
\end{theorem}

We now give the  Proof of Theorem \ref{Tkey}.
\begin{proof}  By a density argument, we can assume that $v \in C_0^{\infty}(0,\infty)$. By setting $t = s^2/2$, then $s = \sqrt{2t}$ and $ds = \frac{1}{\sqrt{2t}} dt$. Hence, defining $w(t) = v(\sqrt{2t})$, we obtain
    \begin{align*}
        \int_0^\infty |v'(s)|^2 e^{-s^2/2} s^{\alpha-1} ds 
        &= \int_0^\infty |v'(\sqrt{2t})|^2 e^{-t} (\sqrt{2t})^{\alpha-1} \frac{dt}{\sqrt{2t}} \\
        &= \int_0^\infty |w'(t)|^2 e^{-t} (\sqrt{2t})^\alpha dt \\
        &= 2^{\alpha/2} \int_0^\infty |w'(t)|^2 e^{-t} t^{\alpha/2} dt.
    \end{align*}
    Moreover,
    \begin{align*}
        \int_0^\infty |v(s)-c|^2 e^{-s^2/2} s^{\alpha-1} ds 
        &= \int_0^\infty |w(t) - c|^2 e^{-t} (\sqrt{2t})^{\alpha-2} dt \\
        &= 2^{\alpha/2 - 1} \int_0^\infty |w(t) - c|^2 e^{-t} t^{\alpha/2 - 1} dt.
    \end{align*}

    It is known that the generalized Laguerre polynomials $\{L^{\alpha/2 - 1}_k(x)\}_{k=0,1,2,\ldots}$ form an orthogonal basis of $L^2((0,\infty), e^{-t} t^{\alpha/2 - 1} dt)$, with
    $$\int_0^\infty L^{\alpha/2 - 1}_k(t) L^{\alpha/2 - 1}_j(t) e^{-t} t^{\alpha/2 - 1} dt = \frac{\Gamma(k+\alpha/2)}{k!} \delta_{kj},$$
    and $\partial_t L^{\alpha/2 - 1}_k(t) = - L^{\alpha/2}_{k-1}(t)$.

    Assuming the expansion $w(t) = \sum_{k=0}^\infty a_k L^{\alpha/2 - 1}_k(t)$, we have
    $$w'(t) = - \sum_{k=1}^\infty a_k L^{\alpha/2}_{k-1}(t),$$
    since $L^{\alpha/2 - 1}_0(x) = 1$. Hence,
    \begin{align*}
        \int_0^\infty |w'(t)|^2 e^{-t} t^{\alpha/2} dt 
        &= \sum_{k=1}^\infty |a_k|^2 \int_0^\infty (L^{\alpha/2}_{k-1}(t))^2 e^{-t} t^{\alpha/2} dt \\
        &= \sum_{k=1}^\infty k |a_k|^2 \frac{\Gamma(k + \alpha/2)}{k!}.
    \end{align*}
    Similarly,
    \begin{align*}
        \inf_{c \in \mathbb{R}} \int_0^\infty |w(t) - c|^2 e^{-t} t^{\alpha/2 - 1} dt 
        &= \int_0^\infty |w(t) - a_0|^2 e^{-t} t^{\alpha/2 - 1} dt \\
        &= \sum_{k=1}^\infty |a_k|^2 \int_0^\infty (L^{\alpha/2 - 1}_k(t))^2 e^{-t} t^{\alpha/2 - 1} dt \\
        &= \sum_{k=1}^\infty |a_k|^2 \frac{\Gamma(\alpha/2 + k)}{k!}.
    \end{align*}
    Therefore,
    \begin{align*}
        \int_0^\infty |v'(s)|^2 e^{-s^2/2} s^{\alpha-1} ds 
        &= 2^{\alpha/2} \int_0^\infty |w'(t)|^2 e^{-t} t^{\alpha/2} dt \\
        &= 2^{\alpha/2} \sum_{k=1}^\infty k |a_k|^2 \frac{\Gamma(k + \alpha/2)}{k!} \\
        &\geq 2^{\alpha/2} \sum_{k=1}^\infty |a_k|^2 \frac{\Gamma(k + \alpha/2)}{k!} \\
        &= 2^{\alpha/2} \inf_{c \in \mathbb{R}} \int_0^\infty |w(t) - c|^2 e^{-t} t^{\alpha/2 - 1} dt \\
        &= 2^{\alpha/2} \inf_{c \in \mathbb{R}} \int_0^\infty \left|w\left(\frac{s^2}{2}\right) - c\right|^2 e^{-s^2/2} \left(\frac{s^2}{2}\right)^{\alpha/2 - 1} s ds \\
        &= 2 \inf_{c \in \mathbb{R}} \int_0^\infty |v(s) - c|^2 e^{-s^2/2} s^{\alpha - 1} ds,
    \end{align*}
    as desired. Equality holds when $w(t) = a_0 + a_1 L^{\alpha/2 - 1}_1(t)$, or equivalently,
    $$v(s) = a_0 + a_1 L^{\alpha/2 - 1}_1\left(\frac{s^2}{2}\right),$$
    with $a_0, a_1 \in \mathbb{R}$. 
    
    Moreover, 
    \begin{align*}
        &\int_0^\infty |v'(s)|^2 e^{-s^2/2} s^{\alpha-1} ds -2 \inf_{c \in \mathbb{R}} \int_0^\infty |v(s) - c|^2 e^{-s^2/2} s^{\alpha - 1} ds \\
        &= 2^{\alpha/2} \sum_{k=2}^\infty \left(k-1\right) |a_k|^2 \frac{\Gamma(k + \alpha/2)}{k!} \\
        &\geq 2^{\alpha/2-1} \sum_{k=2}^\infty k |a_k|^2 \frac{\Gamma(k + \alpha/2)}{k!} \\
        &= 2^{\alpha/2-1} \int_0^\infty |w'(t)+a_1 |^2 e^{-t} t^{\alpha/2} dt \\
        &=\frac{1}{2}\int_0^\infty |v'(s)+a_1s|^2 e^{-s^2/2} s^{\alpha-1} ds.
    \end{align*}
    Similarly,
    \begin{align*}
        &\int_0^\infty |v'(s)|^2 e^{-s^2/2} s^{\alpha-1} ds -2 \inf_{c \in \mathbb{R}} \int_0^\infty |v(s) - c|^2 e^{-s^2/2} s^{\alpha - 1} ds \\
        &= 2^{\alpha/2} \sum_{k=2}^\infty \left(k-1\right) |a_k|^2 \frac{\Gamma(k + \alpha/2)}{k!} \\
        &\geq 2^{\alpha/2} \sum_{k=2}^\infty |a_k|^2 \frac{\Gamma(k + \alpha/2)}{k!} \\
        &= 2^{\alpha/2} \int_0^\infty |w(t)-a_0-a_1L_1^{\alpha/2-1}(t) |^2 e^{-t} t^{\alpha/2} dt \\
        &\geq 2\inf_{c,d \in \mathbb{R}}\int_0^\infty |v(s)-c-ds^2|^2 e^{-s^2/2} s^{\alpha-1} ds.
    \end{align*}
\end{proof}

\begin{theorem}\label{lemma 0.1}
    Let $A \geq 2$ and $C > 0$. Then
\begin{align}\label{auxest}
    &\int_0^\infty \left[ (w'(s))^2 + C \left( \frac{w(s)}{s} \right)^2 \right] e^{-s^2/2} s^{A-1} ds \nonumber\\
    &\geq \frac{ - (A - 2) + \sqrt{(A - 2)^2 + 4 C} }{2} \int_0^\infty w^2(s) e^{-s^2/2} s^{A-1} ds.
\end{align}  The equality holds if and only if $w(s)=c s^{\frac{-(A-2)+\sqrt{(A-2)^2+4C}}{2}}$.  Moreover
\begin{align*}
    &\int_0^\infty \left[(w'(s))^2+C\left(\dfrac{w(s)}{s}\right)^2\right]e^{-s^2/2}s^{A-1}ds \\
    &- \dfrac{-(A-2)+\sqrt{(A-2)^2+4C}}{2}\int_0^\infty w^2(s)s^{A-1}e^{-s^2/2}ds \\
    &\geq 2 \inf_{c \in \mathbb{R}} \int_0^\infty \left\vert w(s) - cs^{\frac{-(A-2)+\sqrt{(A-2)^2+4C}}{2}}\right\vert ^2 e^{-s^2/2} s^{A-1} ds.
\end{align*}
\end{theorem}
\begin{proof} (Proof of Theorem \ref{Tkey1}.) Let $w(s)=s^\alpha\eta(s)$ where $\alpha>0$ will be chosen later. Then,
\[
w'(s) = \eta'(s)s^\alpha + \alpha s^{\alpha-1} \eta(s),
\]
which implies that
\begin{align*}
    &\int_0^\infty (w'(s))^2 e^{-s^2/2}s^{A-1}ds\\
    &=\int_0^\infty (\eta'(s))^2s^{A+2\alpha-1}e^{-s^2/2}ds+\alpha^2\int_0^\infty (w(s))^2s^{A-3}e^{-s^2/2}ds+\alpha\int_0^\infty[\eta^2(s)]'s^{A+2\alpha-2}e^{-s^2/2}ds\\
    &=\int_0^\infty (\eta'(s))^2s^{A+2\alpha-1}e^{-s^2/2}ds+(\alpha^2-\alpha(A+2\alpha-2))\int_0^\infty w^2(s)s^{A-3}e^{-s^2/2}ds\\
    &+\alpha \int_0^\infty w^2(s)s^{A-1}e^{-s^2/2}ds.
\end{align*}
Thus,
\begin{align*}
    &\int_0^\infty \left[(w'(s))^2+C\left(\dfrac{w(s)}{s}\right)^2\right]e^{-s^2/2}s^{A-1}ds\\
    &=\int_0^\infty (\eta'(s))^2s^{A+2\alpha-1}e^{-s^2/2}ds+(\alpha^2-\alpha(A+2\alpha-2)+C)\int_0^\infty w^2(s)s^{A-3}e^{-s^2/2}ds\\
    &+\alpha \int_0^\infty w^2(s)s^{A-1}e^{-s^2/2}ds.
\end{align*}
Now, we will choose $\alpha>0$ such that $\alpha^2-\alpha(A+2\alpha-2)+C = 0$. Then, $\alpha$ is determined by
$$\alpha=\dfrac{-(A-2)+\sqrt{(A-2)^2+4C}}{2}.$$
With such $\alpha$, we obtain
\begin{align*}
    &\int_0^\infty \left[(w'(s))^2+C\left(\dfrac{w(s)}{s}\right)^2\right]e^{-s^2/2}s^{A-1}ds \\
    &\geq \dfrac{-(A-2)+\sqrt{(A-2)^2+4C}}{2}\int_0^\infty w^2(s)s^{A-1}e^{-s^2/2}ds,
\end{align*}
as desired. In particular, the equality holds iff $w(s)=ks^{\alpha}=ks^{\frac{-(A-2)+\sqrt{(A-2)^2+4C}}{2}}$.

Moreover, by Theorem \ref{lemma 0}

\begin{align*}
    &\int_0^\infty \left[(w'(s))^2+C\left(\dfrac{w(s)}{s}\right)^2\right]e^{-s^2/2}s^{A-1}ds \\
    &- \dfrac{-(A-2)+\sqrt{(A-2)^2+4C}}{2}\int_0^\infty w^2(s)s^{A-1}e^{-s^2/2}ds \\
    &= \int_0^\infty (\eta'(s))^2s^{A+2\alpha-1}e^{-s^2/2}ds \\
    &\geq 2 \inf_{c \in \mathbb{R}} \int_0^\infty |\eta(s) - c|^2 e^{-s^2/2} s^{A+2\alpha-1} ds\\
    &\geq 2 \inf_{c \in \mathbb{R}} \int_0^\infty \left\vert\frac{w(s)}{s^\alpha} - c\right\vert ^2 e^{-s^2/2} s^{A+2\alpha-1} ds\\
    &\geq 2 \inf_{c \in \mathbb{R}} \int_0^\infty \left\vert w(s) - cs^{\frac{-(A-2)+\sqrt{(A-2)^2+4C}}{2}}\right\vert ^2 e^{-s^2/2} s^{A-1} ds.
\end{align*}
\end{proof}

We also prove that, in the $L^2$-Poincaré type inequality, the $L^2$ stability is equivalent to the gradient stability.

\begin{theorem}\label{lemma 0.2}
    Let $\mu$ and $\nu$ be two finite measures on $\mathbb{R}^N$. We define $L_0$ as the space of constant functions, $L_1$ the space of all functions such that
\begin{equation}\label{GPoin}
    \int_{\mathbb{R}^N} |\nabla u|^2\,d\mu 
    = \lambda 
    \int_{\mathbb{R}^N} |u|^2\,d\nu,
\end{equation}
for some $\lambda>0$. Suppose that $W$ is another space such that $L_0, L_1,$ and $W$ are orthogonal to each other with respect to $W^{1,2}(d\mu)$ and $L^2(d\nu)$. Then, for some $\alpha>0$, the followings are equivalent:
\begin{enumerate}
    \item For all $u=u_0+u_1+w$ with $u_0 \in L_0, u_1 \in L_1$, $w\in W$,
    \[
        \int_{\mathbb{R}^N} |\nabla u|^2\,d\mu 
        - \lambda \inf_{c \in \mathbb{R}} \int_{\mathbb{R}^N} |u - c|^2\,d\nu
        \geq \alpha \inf_{c \in \mathbb{R}, \phi \in L_1} 
        \int_{\mathbb{R}^N} |u- c - \phi|^2\,d\nu.
    \]
    
    \item For all $u=u_0+u_1+w$ with $u_0 \in L_0, u_1 \in L_1$, $w\in W$,
    \[
        \int_{\mathbb{R}^N} |\nabla u|^2\,d\mu 
        - \lambda \inf_{c \in \mathbb{R}} \int_{\mathbb{R}^N} |u - c|^2\,d\nu
        \geq \frac{\alpha}{\lambda + \alpha} 
        \inf_{\phi \in L_1} 
        \int_{\mathbb{R}^N} |\nabla (u - \phi)|^2\,d\mu.
    \]
\end{enumerate}
\end{theorem}

\begin{proof}
Assume (1) holds. For any $w \in W$, we have by (1) that 
\begin{align}
\int_{\mathbb{R}^N} \left\vert \nabla w\right\vert^2 d\mu &\geq \inf_{c \in \mathbb{R}}\lambda \int_{\mathbb{R}^N} \left\vert w-c\right\vert^2 d\nu+ \inf_{c \in \mathbb{R}, \phi \in L_1} \alpha \int_{\mathbb{R}^N} \left\vert w-c-\phi\right\vert^2 d\nu \nonumber \\
&= (\lambda + \alpha )\int_{\mathbb{R}^N} \left\vert w\right\vert^2 d\nu. \label{1}
\end{align}
Now, for any $u=u_0+u_1 + w$ with $u_0 \in L_0$, $u_1 \in L_1$ and $w \in W$, we have
\begin{align*}
\int_{\mathbb{R}^N} \left\vert \nabla u\right\vert^2 d\mu &= \int_{\mathbb{R}^N} \left\vert \nabla u_1\right\vert^2 d\mu + \int_{\mathbb{R}^N} \left\vert \nabla w\right\vert^2 d\mu \\
&= \lambda\int_{\mathbb{R}^N} \left\vert u_1\right\vert^2 d\nu + \int_{\mathbb{R}^N} \left\vert \nabla w\right\vert^2 d\mu \\
&= \lambda\int_{\mathbb{R}^N} \left\vert u-u_0\right\vert^2 d\nu -\lambda\int_{\mathbb{R}^N} \left\vert w\right\vert^2 d\nu+ \int_{\mathbb{R}^N} \left\vert \nabla w\right\vert^2 d\mu \\
&\geq \lambda\inf_{c \in \mathbb{R}}\int_{\mathbb{R}^N} \left\vert u-c\right\vert^2 d\nu -\frac{\lambda}{\lambda + \alpha}\int_{\mathbb{R}^N} \left\vert \nabla w\right\vert^2 d\mu+ \int_{\mathbb{R}^N} \left\vert \nabla w\right\vert^2 d\mu \\
&= \lambda\inf_{c \in \mathbb{R}}\int_{\mathbb{R}^N} \left\vert u-c\right\vert^2 d\nu+ \frac{\alpha}{\lambda + \alpha} \inf_{\phi \in L_1} \int_{\mathbb{R}^N} \left\vert \nabla (u-\phi)\right\vert^2 d\mu.
\end{align*}

Now, assume that (2) holds. Then for any $w \in W$, we have by (2) that 
\begin{align*}
    \int_{\mathbb{R}^N} \left\vert \nabla w\right\vert^2 d\mu &\geq \lambda\inf_{c \in \mathbb{R}}  \int_{\mathbb{R}^N} \left\vert w-c\right\vert^2 d\nu+ \frac{\alpha}{\lambda + \alpha} \inf_{\phi \in L_1}\int_{\mathbb{R}^N} \left\vert \nabla (w-\phi)\right\vert^2 d\mu \nonumber\\
    &=\lambda \int_{\mathbb{R}^N} \left\vert w\right\vert^2 d\nu+ \frac{\alpha}{\lambda + \alpha} \int_{\mathbb{R}^N} \left\vert \nabla w\right\vert^2 d\mu.
\end{align*}
Equivalently $$\int_{\mathbb{R}^N} \left\vert \nabla w\right\vert^2 d\mu \geq (\lambda + \alpha)\int_{\mathbb{R}^N} \left\vert w\right\vert^2 d\nu. $$
Now, for any $u=u_0+u_1 + w$ with $u_0 \in L_0, u_1 \in L_1$, and $w \in W$, we get
\begin{align*}
    \inf_{\phi \in L_1}\int_{\mathbb{R}^N} \left\vert \nabla (u-\phi)\right\vert^2 d\mu &= \int_{\mathbb{R}^N} \left\vert \nabla w\right\vert^2 d\mu \\
    &\geq (\lambda + \alpha)\int_{\mathbb{R}^N} \left\vert w\right\vert^2 d\nu \\
     &= (\lambda + \alpha)\inf_{c \in \mathbb{R}, \phi \in L_1}\int_{\mathbb{R}^N} \left\vert u-c-\phi\right\vert^2 d\nu.
\end{align*}
Therefore, by (2),
\begin{align*}
   \int_{\mathbb{R}^N} \left\vert \nabla u\right\vert^2 d\mu &\geq \inf_{c \in \mathbb{R}} \lambda \int_{\mathbb{R}^N} \left\vert u-c\right\vert^2 d\nu+ \inf_{\phi \in L_1}\frac{\alpha}{\lambda + \alpha} \int_{\mathbb{R}^N} \left\vert \nabla (u-\phi)\right\vert^2 d\nu \\
   &\geq \lambda\inf_{c \in \mathbb{R}}  \int_{\mathbb{R}^N} \left\vert u-c\right\vert^2 d\nu+ \alpha\inf_{c \in \mathbb{R}, \phi \in L_1} \int_{\mathbb{R}^N} \left\vert u-c-\phi\right\vert^2 d\nu,
\end{align*}
as desired.
\end{proof}

We are now ready to prove our weighted Poincar\'e inequalities (Theorem \ref{TPoincare}) and their stability estimates.

\begin{theorem}\label{lemma 1}
    Assume $1 + b - a > 0$ and $b \leq \frac{N-2}{2}$, we have
    \begin{equation}\label{Poin1}
        \int_{\mathbb{R}^N} |\nabla v|^2 \frac{e^{-\frac{2|x|^{1+b-a}}{1+b-a}}}{|x|^{2b}} \, dx 
        \geq C_{PI}(a, b,N) \inf_{c \in \mathbb{R}} \int_{\mathbb{R}^N} |v - c|^2 \frac{e^{-\frac{2|x|^{1+b-a}}{1+b-a}}}{|x|^{1+b+a}}  \, dx.
    \end{equation} Here the optimal constant is \[
C_{PI}(a,b,N) = \min\left\{ C_1:=2(1+b-a),\, D_1:= \sqrt{(N-2-2b)^2 + 4(N-1)} - (N-2-2b) \right\}.
\] If $C_{PI}(a,b,N) = C_1$, then the equality happens with $
v(x) = a_0 + a_1 |x|^{1+b-a}.
$
Otherwise, if
$
C_{PI}(a,b,N) = D_1,
$
then the equality occurs at
$$
v(x)=c+|x|^{\frac{D_1}{2}-1} \mathbf{d}\cdot x
$$
where $c \in \mathbb{R}, \mathbf{d} \in \mathbb{R}^N$. 

Moreover, if $C_1<D_1$, we have
\begin{align*}
        &\int_{\mathbb{R}^N} |\nabla v|^2 \frac{e^{-\frac{2|x|^{1+b-a}}{1+b-a}}}{|x|^{2b}} \, dx 
        - C_1 \inf_{c \in \mathbb{R}} \int_{\mathbb{R}^N} |v - c|^2 \frac{e^{-\frac{2|x|^{1+b-a}}{1+b-a}}}{|x|^{1+b+a}}  \, dx \\
        &\geq \min\left\{\frac{1}{2},1-\frac{C_1}{D_1}\right\} \inf_{d \in \mathbb{R}} \int_{\mathbb{R}^N} \left\vert\nabla\left(v -  d|x|^{1+b-a}\right)\right\vert^2 \frac{e^{-\frac{2|x|^{1+b-a}}{1+b-a}}}{|x|^{2b}} dx,
    \end{align*}
and
\begin{align*}
        &\int_{\mathbb{R}^N} |\nabla v|^2 \frac{e^{-\frac{2|x|^{1+b-a}}{1+b-a}}}{|x|^{2b}} \, dx 
        - C_1 \inf_{c \in \mathbb{R}} \int_{\mathbb{R}^N} |v - c|^2 \frac{e^{-\frac{2|x|^{1+b-a}}{1+b-a}}}{|x|^{1+b+a}}  \, dx \\
        &\geq \min\left\{C_1,D_1-C_1\right\} \inf_{c,d \in \mathbb{R}} \int_{\mathbb{R}^N} \left\vert v -c-  d|x|^{1+b-a} \right\vert^2 \frac{e^{-\frac{2|x|^{1+b-a}}{1+b-a}}}{|x|^{a+b+1}} dx,
    \end{align*}
If $C_1>D_1$, then 
\begin{align*}
        &\int_{\mathbb{R}^N} |\nabla v|^2 \frac{e^{-\frac{2|x|^{1+b-a}}{1+b-a}}}{|x|^{2b}} \, dx 
        - D_1 \inf_{c \in \mathbb{R}} \int_{\mathbb{R}^N} |v - c|^2 \frac{e^{-\frac{2|x|^{1+b-a}}{1+b-a}}}{|x|^{1+b+a}}  \, dx \\ 
        &\geq \min\left\{C_1- D_1,D_2-D_1 \right\} \inf_{c \in \mathbb{R},\mathbf{d} \in \mathbb{R}^N}\int_{\mathbb{R}^N} \left\vert v-c - |x|^{\frac{D_1}{2}-1}\mathbf{d} \cdot x\right\vert^2 \frac{e^{-\frac{2|x|^{1+b-a}}{1+b-a}}}{|x|^{a+b+1}} \, dx
    \end{align*}
    and \begin{align*}
        &\int_{\mathbb{R}^N} |\nabla v|^2 \frac{e^{-\frac{2|x|^{1+b-a}}{1+b-a}}}{|x|^{2b}} \, dx 
        - D_1 \inf_{c \in \mathbb{R}} \int_{\mathbb{R}^N} |v - c|^2 \frac{e^{-\frac{2|x|^{1+b-a}}{1+b-a}}}{|x|^{1+b+a}}  \, dx \\
        &\geq \min\left\{1- \frac{D_1}{C_1},1-\frac{D_1}{D_2} \right\} \inf_{c \in \mathbb{R},\mathbf{d} \in \mathbb{R}^N}\int_{\mathbb{R}^N} \left\vert \nabla(v-c - |x|^{\frac{D_1}{2}-1}\mathbf{d} \cdot x)\right\vert^2 \frac{e^{-\frac{2|x|^{1+b-a}}{1+b-a}}}{|x|^{2b}} \, dx.
    \end{align*}
    If $C_1=D_1$, then 
      \begin{align*}
        &\int_{\mathbb{R}^N} |\nabla v|^2 \frac{e^{-\frac{2|x|^{1+b-a}}{1+b-a}}}{|x|^{2b}} \, dx 
        - D_1 \inf_{c \in \mathbb{R}} \int_{\mathbb{R}^N} |v - c|^2 \frac{e^{-\frac{2|x|^{1+b-a}}{1+b-a}}}{|x|^{1+b+a}}  \, dx \\ 
        &\geq \min\left\{D_1,D_2-D_1 \right\}\inf_{c, d \in \mathbb{R},\mathbf{d} \in \mathbb{R}^N}\int_{\mathbb{R}^N} \left\vert v-c -d|x|^{1+b-a}- |x|^{\frac{D_1}{2}-1}\mathbf{d} \cdot x\right\vert^2 \frac{e^{-\frac{2|x|^{1+b-a}}{1+b-a}}}{|x|^{a+b+1}} \, dx
    \end{align*}
    and
 \begin{align*}
        &\int_{\mathbb{R}^N} |\nabla v|^2 \frac{e^{-\frac{2|x|^{1+b-a}}{1+b-a}}}{|x|^{2b}} \, dx 
        - D_1 \inf_{c \in \mathbb{R}} \int_{\mathbb{R}^N} |v - c|^2 \frac{e^{-\frac{2|x|^{1+b-a}}{1+b-a}}}{|x|^{1+b+a}}  \, dx \\
        &\geq \min\left\{\frac{1}{2},1-\frac{D_1}{D_2} \right\}\inf_{c, d \in \mathbb{R},\mathbf{d} \in \mathbb{R}^N}\int_{\mathbb{R}^N} \left\vert \nabla(v-c -d|x|^{1+b-a}- |x|^{\frac{D_1}{2}-1}\mathbf{d} \cdot x)\right\vert^2 \frac{e^{-\frac{2|x|^{1+b-a}}{1+b-a}}}{|x|^{2b}} \, dx.
    \end{align*}   
    Here \[
D_2 = \sqrt{(N - 2 - 2b)^2 + 8N} - (N - 2 - 2b).
\]
\end{theorem}
\begin{proof} The proof is divided into several steps.\\
{\bf Step 1}. First, consider the radial case. If $v$ is radial, then $v(x) = V(r)$ with $r = |x|$. Thus,
\[
\int_{\mathbb{R}^N} |\nabla v|^2 \frac{e^{-\frac{2|x|^{1+b-a}}{1+b-a}}}{|x|^{2b}} \, dx = \omega_{N-1} \int_0^\infty |V'(r)|^2 e^{-\frac{2r^{1+b-a}}{1+b-a}} r^{N-1-2b} dr.
\]

Set $\dfrac{s^2}{2} = \dfrac{2r^{1+b-a}}{1+b-a}$, which implies
$$r = \left(\frac{1+b-a}{4}\right)^{\frac{1}{1+b-a}} s^{\frac{2}{1+b-a}},$$
and 
$$dr = \left(\frac{1+b-a}{4}\right)^{\frac{1}{1+b-a}} \frac{2}{1+b-a} s^{\frac{2}{1+b-a} - 1} ds.$$
Then,
\begin{align*}
    &\int_{\mathbb{R}^N} |\nabla v|^2 \frac{e^{-\frac{2|x|^{1+b-a}}{1+b-a}}}{|x|^{2b}} \, dx \\
    &= \omega_{N-1} \left(\frac{1+b-a}{4}\right)^{\frac{N-2b}{1+b-a}} \frac{2}{1+b-a} 
       \int_0^\infty \left\vert v'\left(\left(\frac{1+b-a}{4}\right)^{\frac{1}{1+b-a}} s^{\frac{2}{1+b-a}}\right)\right\vert^2 e^{-s^2 / 2} s^{\frac{2(N-2b)}{1+b-a} - 1} ds.
\end{align*}
Define 
$$w(s) = v\left(\left(\frac{1+b-a}{4}\right)^{\frac{1}{1+b-a}} s^{\frac{2}{1+b-a}}\right),$$
so
$$v'\left(\left(\frac{1+b-a}{4}\right)^{\frac{1}{1+b-a}} s^{\frac{2}{1+b-a}}\right) 
  = \left(\frac{1+b-a}{4}\right)^{-\frac{1}{1+b-a}} \frac{1+b-a}{2} s^{1 - \frac{2}{1+b-a}} w'(s).$$
 
Since $a < b + 1 \leq \frac{N}{2}$, we have $a + b + 1 < N$. Using Theorem \ref{lemma 0} with 
$$\alpha = \frac{2(N - a - b - 1)}{1 + b - a} > 0,$$
we get
\begin{align*}
    &\int_{\mathbb{R}^N} |\nabla v|^2 \frac{e^{-\frac{2|x|^{1+b-a}}{1+b-a}}}{|x|^{2b}} \, dx \\
    &= 2 \omega_{N-1} \left(\frac{1+b-a}{4}\right)^{\frac{N - a - b - 1}{1 + b - a}} 
       \int_0^\infty |w'(s)|^2 e^{-s^2/2} s^{\frac{2(N - a - b - 1)}{1 + b - a} - 1} ds \\
    &\geq 4 \omega_{N-1} \left(\frac{1+b-a}{4}\right)^{\frac{N - a - b - 1}{1 + b - a}} 
       \inf_{c \in \mathbb{R}} \int_0^\infty |w(s) - c|^2 e^{-s^2/2} s^{\frac{2(N - a - b - 1)}{1 + b - a} - 1} ds \\
    &= 2 (1 + b - a) \inf_{c \in \mathbb{R}} \int_{\mathbb{R}^N} |v(x) - c|^2 \frac{e^{-\frac{2|x|^{1+b-a}}{1+b-a}}}{|x|^{a + b + 1}} dx,
\end{align*}
where we used the substitution $s = \left(\frac{4}{1+b-a}\right)^{1/2} r^{\frac{1+b-a}{2}}$. This establishes the estimate for the radial case. Moreover, by Theorem \ref{lemma 0},
\begin{align}\label{rad1}
    &\int_{\mathbb{R}^N} |\nabla v|^2 \frac{e^{-\frac{2|x|^{1+b-a}}{1+b-a}}}{|x|^{2b}} \, dx - 2 (1 + b - a) \inf_{c \in \mathbb{R}} \int_{\mathbb{R}^N} |v(x) - c|^2 \frac{e^{-\frac{2|x|^{1+b-a}}{1+b-a}}}{|x|^{a + b + 1}} dx \nonumber\\
    &=2 \omega_{N-1} \left(\frac{1+b-a}{4}\right)^{\frac{N - a - b - 1}{1 + b - a}} 
       \int_0^\infty |w'(s)|^2 e^{-s^2/2} s^{\frac{2(N - a - b - 1)}{1 + b - a} - 1} ds \nonumber\\
    &- 4 \omega_{N-1} \left(\frac{1+b-a}{4}\right)^{\frac{N - a - b - 1}{1 + b - a}} 
       \inf_{c \in \mathbb{R}} \int_0^\infty |w(s) - c|^2 e^{-s^2/2} s^{\frac{2(N - a - b - 1)}{1 + b - a} - 1} ds \nonumber\\
       &\geq \omega_{N-1} \left(\frac{1+b-a}{4}\right)^{\frac{N - a - b - 1}{1 + b - a}} 
       \inf_{d \in \mathbb{R}}\int_0^\infty |w'(s)-ds|^2 e^{-s^2/2} s^{\frac{2(N - a - b - 1)}{1 + b - a} - 1} ds \nonumber\\
       &=\frac{1}{2}\inf_{d \in \mathbb{R}}\int_{\mathbb{R}^N} |\nabla (v-d|x|^{1+b-a})|^2 \frac{e^{-\frac{2|x|^{1+b-a}}{1+b-a}}}{|x|^{2b}} \, dx.
\end{align}
Also, 
\begin{align}\label{rad2}
    &\int_{\mathbb{R}^N} |\nabla v|^2 \frac{e^{-\frac{2|x|^{1+b-a}}{1+b-a}}}{|x|^{2b}} \, dx - 2 (1 + b - a) \inf_{c \in \mathbb{R}} \int_{\mathbb{R}^N} |v(x) - c|^2 \frac{e^{-\frac{2|x|^{1+b-a}}{1+b-a}}}{|x|^{a + b + 1}} dx\nonumber\\
    &=2 \omega_{N-1} \left(\frac{1+b-a}{4}\right)^{\frac{N - a - b - 1}{1 + b - a}} 
       \int_0^\infty |w'(s)|^2 e^{-s^2/2} s^{\frac{2(N - a - b - 1)}{1 + b - a} - 1} ds \nonumber\\
    &- 4 \omega_{N-1} \left(\frac{1+b-a}{4}\right)^{\frac{N - a - b - 1}{1 + b - a}} 
       \inf_{c \in \mathbb{R}} \int_0^\infty |w(s) - c|^2 e^{-s^2/2} s^{\frac{2(N - a - b - 1)}{1 + b - a} - 1} ds \nonumber\\
       &\geq 4\omega_{N-1} \left(\frac{1+b-a}{4}\right)^{\frac{N - a - b - 1}{1 + b - a}} 
       \inf_{c, d \in \mathbb{R}}\int_0^\infty |w(s)-c-ds^2|^2 e^{-s^2/2} s^{\frac{2(N - a - b - 1)}{1 + b - a} - 1} ds\nonumber\\
       &=2(1+b-a)\inf_{c, d \in \mathbb{R}}\int_{\mathbb{R}^N} |v(x)-c-d|x|^{1+b-a}|^2 \frac{e^{-\frac{2|x|^{1+b-a}}{1+b-a}}}{|x|^{a+b+1}} \, dx.
\end{align}
{\bf Step 2.}  For a general function $v$, we use the spherical harmonic expansion
\[
v(r,\sigma) = \sum_{k=0}^{\infty}\sum_{i=0}^{N_k} v_{k,i}(r)\,\phi_{k,i}(\sigma),
\]
where the spherical harmonics $\phi_{k,i}$ satisfy
\[
-\Delta_{\mathbb{S}^{N-1}}\phi_{k,i} = c_k\,\phi_{k,i}, 
\quad c_k = k(N+k-2), 
\quad \int_{\mathbb{S}^{N-1}}\phi_{k,i}^2\,d\sigma = 1,
\]
and the radial coefficients obey $v_{k,i}(r) = O(r^k)$ as $r \to 0^+$. 
When no ambiguity arises, we suppress the second index and write $v_k$ and $\phi_k$ in place of $v_{k,i}$ and $\phi_{k,i}$.

Then,
\[
\int_{\mathbb{R}^N} |\nabla v|^2 \frac{e^{-\frac{2|x|^{1+b-a}}{1+b-a}}}{|x|^{2b}} dx = \sum_{k=0}^\infty \int_{\mathbb{R}^N} \left( |\nabla v_k|^2 + c_k \frac{|v_k|^2}{|x|^2} \right) \frac{e^{-\frac{2|x|^{1+b-a}}{1+b-a}}}{|x|^{2b}} dx.
\]
Also, 
\[
v - c = (v_0 - c) + \sum_{k=1}^\infty v_k \phi_k,
\]
and
\[
\int_{\mathbb{R}^N} |v - c|^2 \frac{e^{-\frac{2|x|^{1+b-a}}{1+b-a}}}{|x|^{a + b + 1}} dx = \int_{\mathbb{R}^N} |v_0 - c|^2 \frac{e^{-\frac{2|x|^{1+b-a}}{1+b-a}}}{|x|^{a + b + 1}} dx + \sum_{k=1}^\infty \int_{\mathbb{R}^N} |v_k|^2 \frac{e^{-\frac{2|x|^{1+b-a}}{1+b-a}}}{|x|^{a + b + 1}} dx.
\]
Consequently,
\[
\inf_{c \in \mathbb{R}} \int_{\mathbb{R}^N} |v - c|^2 \frac{e^{-\frac{2|x|^{1+b-a}}{1+b-a}}}{|x|^{a + b + 1}} dx = \inf_{c \in \mathbb{R}} \int_{\mathbb{R}^N} |v_0 - c|^2 \frac{e^{-\frac{2|x|^{1+b-a}}{1+b-a}}}{|x|^{a + b + 1}} dx + \sum_{k=1}^\infty \int_{\mathbb{R}^N} |v_k|^2 \frac{e^{-\frac{2|x|^{1+b-a}}{1+b-a}}}{|x|^{a + b + 1}} dx.
\]
Define 
\[
D_k = \inf \frac{\int_{\mathbb{R}^N} \left( |\nabla u|^2 + c_k \frac{|u|^2}{|x|^2} \right) \frac{e^{-\frac{2|x|^{1+b-a}}{1+b-a}}}{|x|^{2b}} dx}{\int_{\mathbb{R}^N} |u|^2 \frac{e^{-\frac{2|x|^{1+b-a}}{1+b-a}}}{|x|^{a + b + 1}} dx},
\]
where the infimum is taken over all radial functions $u \not\equiv 0$ satisfying $u(r) = O(r)$ as $r \to 0^+$. Since $c_k \geq c_1$ for all $k \geq 1$, we have $D_k \geq D_1$ for all $k \geq 1$. From the previous radial case,
\[
\int_{\mathbb{R}^N} |\nabla v_0|^2 \frac{e^{-\frac{2|x|^{1+b-a}}{1+b-a}}}{|x|^{2b}} dx \geq 2(1 + b - a) \inf_{c \in \mathbb{R}} \int_{\mathbb{R}^N} |v_0 - c|^2 \frac{e^{-\frac{2|x|^{1+b-a}}{1+b-a}}}{|x|^{a + b + 1}} dx,
\]
and thus
\[
C_{PI} \geq \min \{ 2(1 + b - a), D_1 \}.
\]
{\bf Step 3.} Next, compute $D_1$ explicitly. Since $c_1 = N - 1$, we have for radial functions $u$ that
\[
\int_{\mathbb{R}^N} \left( |\nabla u|^2 + c_1 \frac{|u|^2}{|x|^2} \right) \frac{e^{-\frac{2|x|^{1+b-a}}{1+b-a}}}{|x|^{2b}} dx = \omega_{N-1} \int_0^\infty \left[ |u'(r)|^2 + (N - 1) \frac{|u(r)|^2}{r^2} \right] e^{-\frac{2 r^{1+b-a}}{1+b-a}} r^{N - 1 - 2b} dr.
\]
By the previous substitution for $w(s)$, it follows that
\begin{align*}
    &\int_{\mathbb{R}^N} \left( |\nabla u|^2 + c_1 \frac{|u|^2}{|x|^2} \right) \frac{e^{-\frac{2|x|^{1+b-a}}{1+b-a}}}{|x|^{2b}} dx  \\
    &= 2 \omega_{N-1} \left( \frac{1 + b - a}{4} \right)^{\frac{N - a - b - 1}{1 + b - a}} \int_0^\infty \left[ (w'(s))^2 + \frac{4 (N-1)}{(1 + b - a)^2} \frac{w^2(s)}{s^2} \right] e^{-s^2/2} s^{\frac{2 (N - a - b - 1)}{1 + b - a} - 1} ds.
\end{align*}
To estimate the integral on the right hand side, we apply Theorem \ref{lemma 0.1} with
$$C = \frac{4 (N - 1)}{(1 + b - a)^2}, \quad A = \frac{2 (N - a - b - 1)}{1 + b - a} \geq 2.$$
In this case, we get
\begin{align*}
    &\int_{\mathbb{R}^N} \left( |\nabla u|^2 + c_1 \frac{|u|^2}{|x|^2} \right) \frac{e^{-\frac{2|x|^{1+b-a}}{1+b-a}}}{|x|^{2b}} dx \\
    &\geq 2 \omega_{N-1} \left(\frac{1 + b - a}{4}\right)^{\frac{N - a - b -1}{1 + b - a}} \frac{ \sqrt{(N - 2 - 2b)^2 + 4 (N - 1)} - (N - 2 - 2b) }{1 + b - a} \\
    &\quad \times \int_0^\infty w^2(s) e^{-s^2/2} s^{\frac{2 (N - a - b - 1)}{1 + b - a} - 1} ds \\
    &= \left( \sqrt{(N - 2 - 2b)^2 + 4 (N - 1)} - (N - 2 - 2b) \right) \int_{\mathbb{R}^N} |u|^2 \frac{e^{-\frac{2|x|^{1+b -a}}{1 + b - a}}}{|x|^{a + b + 1}} dx.
\end{align*}

Thus,
\[
D_1 = \sqrt{(N - 2 - 2b)^2 + 4 (N - 1)} - (N - 2 - 2b).
\]
Moreover, by Theorem \ref{lemma 0.1},  we have for radial functions $u$ that
\begin{align*}
    &\int_{\mathbb{R}^N} \left( |\nabla u|^2 + c_1 \frac{|u|^2}{|x|^2} \right) \frac{e^{-\frac{2|x|^{1+b-a}}{1+b-a}}}{|x|^{2b}} dx - D_1 \int_{\mathbb{R}^N} |u|^2 \frac{e^{-\frac{2|x|^{1+b -a}}{1 + b - a}}}{|x|^{a + b + 1}} dx\\
    &=C'(N,a,b)\Bigg( \int_0^\infty \left[ (w'(s))^2 + \frac{4 (N-1)}{(1 + b - a)^2} \frac{w^2(s)}{s^2} \right] e^{-s^2/2} s^{\frac{2 (N - a - b - 1)}{1 + b - a} - 1} ds\\
    &- \frac{ \sqrt{(N - 2 - 2b)^2 + 4 (N - 1)} - (N - 2 - 2b) }{1 + b - a} 
    \int_0^\infty w^2(s) e^{-s^2/2} s^{\frac{2 (N - a - b - 1)}{1 + b - a} - 1} ds\Bigg)\\
    &\geq 2C'(N,a,b) \inf_{d \in \mathbb{R}} \int_0^\infty \left\vert w(s) - ds^{\frac{-(A-2)+\sqrt{(A-2)^2+4C}}{2}}\right\vert ^2 e^{-s^2/2} s^{A-1} ds\\
    &=2\inf_{d \in \mathbb{R}}\int_{\mathbb{R}^N} |u- d|x|^{\frac{\sqrt{(N-2-2b)^2+4(N-1)}-(N-2-2b)}{2}}|^2 \frac{e^{-\frac{2|x|^{1+b -a}}{1 + b - a}}}{|x|^{a + b + 1}} dx,
\end{align*}
where $A = \frac{2 (N - a - b - 1)}{1 + b - a}$, $C = \frac{4 (N - 1)}{(1 + b - a)^2}$ and $C'(N,a,b)=2 \omega_{N-1} \left( \frac{1 + b - a}{4} \right)^{\frac{N - a - b - 1}{1 + b - a}}$.
We also note that by the same approach, we also obtain that with \[
D_k = \sqrt{(N - 2 - 2b)^2 + 4c_k} - (N - 2 - 2b),
\]
we have 
\begin{align}\label{k1}
    &\int_{\mathbb{R}^N} \left( |\nabla u|^2 + c_k \frac{|u|^2}{|x|^2} \right) \frac{e^{-\frac{2|x|^{1+b-a}}{1+b-a}}}{|x|^{2b}} dx \geq D_k \int_{\mathbb{R}^N} |u|^2 \frac{e^{-\frac{2|x|^{1+b -a}}{1 + b - a}}}{|x|^{a + b + 1}} dx,
\end{align}
and
\begin{align}\label{k2}
    &\int_{\mathbb{R}^N} \left( |\nabla u|^2 + c_k \frac{|u|^2}{|x|^2} \right) \frac{e^{-\frac{2|x|^{1+b-a}}{1+b-a}}}{|x|^{2b}} dx - D_k \int_{\mathbb{R}^N} |u|^2 \frac{e^{-\frac{2|x|^{1+b -a}}{1 + b - a}}}{|x|^{a + b + 1}} dx \nonumber\\
    &\geq 2(1+b-a)\inf_{d \in \mathbb{R}}\int_{\mathbb{R}^N} |u- d|x|^{\frac{D_k}{2}}|^2 \frac{e^{-\frac{2|x|^{1+b -a}}{1 + b - a}}}{|x|^{a + b + 1}} dx.
\end{align}
{\bf Step 4.} Therefore, the sharp constant in \eqref{Poin1} is
\[
C_{PI}(a,b,N) = \min \left\{ 2(1 + b - a), \sqrt{(N - 2 - 2b)^2 + 4 (N - 1)} - (N - 2 - 2b) \right\}.
\]

Moreover, this constant is sharp: if $C_{PI}(a,b,N) = 2 (1 + b - a)$, choose 
$$
v(x) = a_0 + a_1 L_1^{\alpha/2 - 1} \left( |x|^{1+b-a} \right).
$$
Otherwise, if
$$
C_{PI}(a,b,N) = \sqrt{(N - 2 - 2b)^2 + 4 (N - 1)} - (N - 2 - 2b),
$$
choose
$$
v(x) = |x|^{\frac{\sqrt{(N-2-2b)^2+4(N-1)}-(N-2-2b)}{2}-1}\mathbf{d}\cdot x.
$$
{\bf Step 5.} If $C_1:=2 (1 + b - a) < D_1=\sqrt{(N - 2 - 2b)^2 + 4 (N - 1)} - (N - 2 - 2b)$, then we have from \eqref{rad1} and \eqref{k1} that
\begin{align*}
        &\int_{\mathbb{R}^N} |\nabla v|^2 \frac{e^{-\frac{2|x|^{1+b-a}}{1+b-a}}}{|x|^{2b}} \, dx 
        - C_1 \inf_{c \in \mathbb{R}} \int_{\mathbb{R}^N} |v - c|^2 \frac{e^{-\frac{2|x|^{1+b-a}}{1+b-a}}}{|x|^{1+b+a}}  \, dx \\
        &= \int_{\mathbb{R}^N} |\nabla v_0|^2 \frac{e^{-\frac{2|x|^{1+b-a}}{1+b-a}}}{|x|^{2b}} dx - C_1 \inf_{c \in \mathbb{R}} \int_{\mathbb{R}^N} |v_0 - c|^2 \frac{e^{-\frac{2|x|^{1+b-a}}{1+b-a}}}{|x|^{a + b + 1}} dx \\
        &+ \sum_{k=1}^\infty \int_{\mathbb{R}^N} \left( |\nabla v_k|^2 + c_k \frac{|v_k|^2}{|x|^2} \right) \frac{e^{-\frac{2|x|^{1+b-a}}{1+b-a}}}{|x|^{2b}} dx -C_1 \int_{\mathbb{R}^N} |v_k|^2 \frac{e^{-\frac{2|x|^{1+b-a}}{1+b-a}}}{|x|^{a + b + 1}} dx \\
        &\geq \frac{1}{2} \inf_{d \in \mathbb{R}} \int_{\mathbb{R}^N} \left\vert\nabla\left(v_0 - d|x|^{1+b-a}\right)\right\vert^2 \frac{e^{-\frac{2|x|^{1+b-a}}{1+b-a}}}{|x|^{2b}} dx \\
        &+ (1-\frac{C_1}{D_1}) \sum_{k=1}^\infty \int_{\mathbb{R}^N} \left( |\nabla v_k|^2 + c_k \frac{|v_k|^2}{|x|^2} \right) \frac{e^{-\frac{2|x|^{1+b-a}}{1+b-a}}}{|x|^{2b}} dx \\
        &\geq \min\left\{\frac{1}{2},1-\frac{C_1}{D_1}\right\} \inf_{d \in \mathbb{R}} \int_{\mathbb{R}^N} \left\vert\nabla\left(v -  d|x|^{1+b-a}\right)\right\vert^2 \frac{e^{-\frac{2|x|^{1+b-a}}{1+b-a}}}{|x|^{2b}} dx.
    \end{align*}
    Note that we also have 
    \begin{align*}
        &\int_{\mathbb{R}^N} |\nabla v|^2 \frac{e^{-\frac{2|x|^{1+b-a}}{1+b-a}}}{|x|^{2b}} \, dx 
        - C_1 \inf_{c \in \mathbb{R}} \int_{\mathbb{R}^N} |v - c|^2 \frac{e^{-\frac{2|x|^{1+b-a}}{1+b-a}}}{|x|^{1+b+a}}  \, dx \\
        &= \int_{\mathbb{R}^N} |\nabla v_0|^2 \frac{e^{-\frac{2|x|^{1+b-a}}{1+b-a}}}{|x|^{2b}} dx - C_1 \inf_{c \in \mathbb{R}} \int_{\mathbb{R}^N} |v_0 - c|^2 \frac{e^{-\frac{2|x|^{1+b-a}}{1+b-a}}}{|x|^{a + b + 1}} dx \\
        &+ \sum_{k=1}^\infty \int_{\mathbb{R}^N} \left( |\nabla v_k|^2 + c_k \frac{|v_k|^2}{|x|^2} \right) \frac{e^{-\frac{2|x|^{1+b-a}}{1+b-a}}}{|x|^{2b}} dx -C_1 \int_{\mathbb{R}^N} |v_k|^2 \frac{e^{-\frac{2|x|^{1+b-a}}{1+b-a}}}{|x|^{a + b + 1}} dx \\
        &\geq C_1 \inf_{c,d \in \mathbb{R}} \int_{\mathbb{R}^N} \left\vert v_0 - c - d|x|^{1+b-a}\right\vert^2 \frac{e^{-\frac{2|x|^{1+b-a}}{1+b-a}}}{|x|^{a+b+1}} dx \\
        &+ (D_1-C_1) \sum_{k=1}^\infty \int_{\mathbb{R}^N} |v_k|^2 \frac{e^{-\frac{2|x|^{1+b-a}}{1+b-a}}}{|x|^{a + b + 1}} dx \\
        &\geq \min\left\{C_1,D_1-C_1\right\} \inf_{c,d \in \mathbb{R}} \int_{\mathbb{R}^N} \left\vert v -c-  d|x|^{1+b-a} \right\vert^2 \frac{e^{-\frac{2|x|^{1+b-a}}{1+b-a}}}{|x|^{a+b+1}} dx.
    \end{align*}
   {\bf Step 6.} Now, assume that $C_1=2 (1 + b - a) > D_1=\sqrt{(N - 2 - 2b)^2 + 4 (N - 1)} - (N - 2 - 2b)$. In this case, we have
\begin{align*}
        &\int_{\mathbb{R}^N} |\nabla v|^2 \frac{e^{-\frac{2|x|^{1+b-a}}{1+b-a}}}{|x|^{2b}} \, dx 
        - D_1 \inf_{c \in \mathbb{R}} \int_{\mathbb{R}^N} |v - c|^2 \frac{e^{-\frac{2|x|^{1+b-a}}{1+b-a}}}{|x|^{1+b+a}}  \, dx \\
        &= \int_{\mathbb{R}^N} |\nabla v_0|^2 \frac{e^{-\frac{2|x|^{1+b-a}}{1+b-a}}}{|x|^{2b}} dx - D_1 \inf_{c \in \mathbb{R}} \int_{\mathbb{R}^N} |v_0 - c|^2 \frac{e^{-\frac{2|x|^{1+b-a}}{1+b-a}}}{|x|^{a + b + 1}} dx \\
        &+\sum_{i=1}^{N} \int_{\mathbb{R}^N} \left( |\nabla v_{1,i}|^2 + c_1 \frac{|v_{1,i}|^2}{|x|^2} \right) \frac{e^{-\frac{2|x|^{1+b-a}}{1+b-a}}}{|x|^{2b}} dx -D_1 \int_{\mathbb{R}^N} |v_{1,i}|^2 \frac{e^{-\frac{2|x|^{1+b-a}}{1+b-a}}}{|x|^{a + b + 1}} dx \\
        &+ \sum_{k=2}^\infty \int_{\mathbb{R}^N} \left( |\nabla v_k|^2 + c_k \frac{|v_k|^2}{|x|^2} \right) \frac{e^{-\frac{2|x|^{1+b-a}}{1+b-a}}}{|x|^{2b}} dx -D_1 \int_{\mathbb{R}^N} |v_k|^2 \frac{e^{-\frac{2|x|^{1+b-a}}{1+b-a}}}{|x|^{a + b + 1}} dx \\
        &\geq \left(C_1- D_1\right) \inf_{c \in \mathbb{R}} \int_{\mathbb{R}^N} |v_0 - c|^2 \frac{e^{-\frac{2|x|^{1+b-a}}{1+b-a}}}{|x|^{a + b + 1}} dx \\ \\
        &+ 2(1+b-a)\sum_{i=1}^{N}\inf_{d_i \in \mathbb{R}}\int_{\mathbb{R}^N} |v_{1,i}\phi_{1,i}- d_i|x|^{\frac{D_1}{2}}\phi_{1,i}|^2 \frac{e^{-\frac{2|x|^{1+b -a}}{1 + b - a}}}{|x|^{a + b + 1}} dx\\
        &+ \left(D_2 -D_1\right)\sum_{k=2}^\infty \int_{\mathbb{R}^N} |v_k|^2 \frac{e^{-\frac{2|x|^{1+b-a}}{1+b-a}}}{|x|^{a + b + 1}} dx\\ 
        &\geq \min\left\{C_1- D_1,D_2-D_1 \right\} \inf_{c \in \mathbb{R},\mathbf{d} \in \mathbb{R}^N}\int_{\mathbb{R}^N} \left\vert v-c - |x|^{\frac{D_1}{2}-1}\mathbf{d} \cdot x\right\vert^2 \frac{e^{-\frac{2|x|^{1+b-a}}{1+b-a}}}{|x|^{a+b+1}} \, dx.
    \end{align*}
    Now, by applying Theorem \ref{lemma 0.2}, we also get that
\begin{align*}
        &\int_{\mathbb{R}^N} |\nabla v|^2 \frac{e^{-\frac{2|x|^{1+b-a}}{1+b-a}}}{|x|^{2b}} \, dx 
        - D_1 \inf_{c \in \mathbb{R}} \int_{\mathbb{R}^N} |v - c|^2 \frac{e^{-\frac{2|x|^{1+b-a}}{1+b-a}}}{|x|^{1+b+a}}  \, dx \\
        &\geq \min\left\{1- \frac{D_1}{C_1},1-\frac{D_1}{D_2} \right\} \inf_{c \in \mathbb{R},\mathbf{d} \in \mathbb{R}^N}\int_{\mathbb{R}^N} \left\vert \nabla(v-c - |x|^{\frac{D_1}{2}-1}\mathbf{d} \cdot x)\right\vert^2 \frac{e^{-\frac{2|x|^{1+b-a}}{1+b-a}}}{|x|^{2b}} \, dx.
    \end{align*}
Indeed, in this case, we have

\begin{equation*}
        \int_{\mathbb{R}^N} |\nabla v|^2 \frac{e^{-\frac{2|x|^{1+b-a}}{1+b-a}}}{|x|^{2b}} \, dx 
        \geq D_1 \inf_{c \in \mathbb{R}} \int_{\mathbb{R}^N} |v - c|^2 \frac{e^{-\frac{2|x|^{1+b-a}}{1+b-a}}}{|x|^{1+b+a}}  \, dx,
    \end{equation*}
where the equality holds iff $v(x)=c+|x|^{\frac{D_1}{2}-1}\mathbf{d}\cdot x.$ 
Then,
$$\int_{\mathbb{R}^N}|\nabla v|^2d\mu=D_1 \int_{\mathbb{R}^N}|v|^2d\nu$$ 
iff $v \in \text{span}\left\{|x|^{\frac{D_2}{1}-1}\mathbf{d}\cdot x\right\}=\text{span}\left\{\phi_{1,j}(\sigma)r^{\frac{D_1}{2}}:j=\overline{1,N}\right\}$. We claim that Theorem \ref{lemma 0.2} can be applied with 
$$L_1:=\text{span}\left\{\phi_{1,j}(\sigma)r^{\frac{D_1}{2}}: j=\overline{1,N}\right\}.$$
Indeed, by using the spherical harmonics decomposition, we can represent
\begin{align*}
    v&=v_0(r)+\sum_{k=1}^\infty \sum_{j=1}^{N_k} v_{k,j}(r)\phi_{k,j}(\sigma)\\
    &=v_{0}(r)+\sum_{j=1}^N v_{1,j}(r)\phi_{1,j}(\sigma)+ \sum_{k=2}^\infty \sum_{j=1}^{N_k} v_{k,j}(r)\phi_{k,j}(\sigma)\\
    &=\beta+\sum_{j=1}^N \alpha_jr^{\frac{D_1}{2}}\phi_{1,j}(\sigma)\\
    &+(v_{0}(r)-\beta)+\sum_{j=1}^N \left(v_{1,j}(r)-\alpha_j r^{\frac{D_1}{2}}\right)\phi_{1,j}(\sigma)+ \sum_{k=2}^\infty \sum_{j=1}^{N_k} v_{k,j}(r)\phi_{k,j}(\sigma).
\end{align*}
It suffices to determine $(\alpha_j)_j$, $\beta$ such that 
$(v_{0}(r)-\beta)$ and $ \left(v_{1,j}(r)-\alpha_j r^{\frac{D_1}{2}}\right)\phi_{1,j}(\sigma)$
are orthogonal to $L_0$ and $L_1$ in $L^2(\mathbb{R}^N, d\nu)$ and $W^{1,2}(\mathbb{R}^N, d\mu)$ for each $j$. Here, we denote
$$d\nu=\frac{e^{-\frac{2|x|^{1+b-a}}{1+b-a}}}{|x|^{1+b+a}}  \, dx\;\;\text{and}\;\; d\mu=\frac{e^{-\frac{2|x|^{1+b-a}}{1+b-a}}}{|x|^{2b}}dx.$$
Notice that, with $\beta=\frac{\int_{\mathbb{R}^N}v_{0}(r)d\nu}{\int _{\mathbb{R}^N}d\nu}$, the orthogonality of $(v_{0}-\beta)$ is automatically satisfied. Next, by choosing 
$$\alpha_j=\dfrac{D_1\int_{\mathbb{R}^N} r^{\frac{D_1}{2}}v_{1,j}(r)(\phi_{1,j}(\sigma))^2d\nu}{\int_{\mathbb{R}^N}\left|\nabla\left(r^{\frac{D_1}{2}}\phi_{1,j}(\sigma)\right)\right|^2d\mu},$$
we get the orthogonality of $ \left(v_{1,j}(r)-\alpha_j r^{\frac{D_1}{2}}\right)\phi_{1,j}(\sigma)$ in  $L^2(\mathbb{R}^N, d\nu)$. With this $\alpha_j$, it is enough to check if
$$\int_{\mathbb{R}^N}\nabla\left(\alpha_j r^{\frac{D_1}{2}}\phi_{1,j}(\sigma)\right)\nabla\left(\left(v_{1,j}(r)-\alpha_j r^{\frac{D_1}{2}}\right)\phi_{1,j}(\sigma)\right)d\mu=0.$$
It is true since
\begin{align*}
    &\int_{\mathbb{R}^N}\nabla\left(\alpha_j r^{\frac{D_1}{2}}\phi_{1,j}(\sigma)\right)\nabla\left(\left(v_{1,j}(r)-\alpha_j r^{\frac{D_1}{2}}\right)\phi_{1,j}(\sigma)\right)d\mu\\
    &=\alpha_j\left(\int_{\mathbb{R}^N}\nabla\left(r^{\frac{D_1}{2}}\phi_{1,j}(\sigma)\right) \nabla\left( v_{1,j}(r)\phi_{1,j}(\sigma)\right)d\mu-D_1\int_{\mathbb{R}^N} r^{\frac{D_1}{2}}v_{1,j}(r) (\phi_{1,j}(\sigma))^2 d\nu \right),
\end{align*}
and
\begin{align*}
    &\int_{\mathbb{R}^N}\nabla\left(r^{\frac{D_1}{2}}\phi_{1,j}(\sigma)\right) \nabla\left( v_{1,j}(r)\phi_{1,j}(\sigma)\right)d\mu\\
    &=\left(N-1-\frac{D_1^2}{4}-\frac{D_1(N-2b-2)}{2}\right)\omega_{N-1}\int_0^\infty v_{1,j}(r)r^{\frac{D_1}{2}+N-3-2b}e^{-\frac{2r^{1+b-a}}{1+b-a}}dr\\
    &+D_1\omega_{N-1}\int_0^\infty v_{1,j}(r)r^{\frac{D_1}{2}+N-a-b-2}e^{-\frac{2r^{1+b-a}}{1+b-a}}dr\\
    &=D_1\omega_{N-1} \int_0^\infty v_{1,j}(r)r^{\frac{D_1}{2}+N-a-b-2}e^{-\frac{2r^{1+b-a}}{1+b-a}}dr\\
    &=D_1\int_{\mathbb{R}^N} r^{\frac{D_1}{2}}v_{1,j}(r) (\phi_{1,j}(\sigma))^2 d\nu,
\end{align*}
which follows from $N-1-\frac{D_1^2}{4}-\frac{D_1(N-2b-2)}{2}=0$ and $\int_{\mathbb{S}^{N-1}}(\phi_{1,j}(\sigma))^2d\sigma=\omega_{N-1}.$ Therefore, in this case, we choose
$$L_1=\text{span} \left\{|x|^{\frac{D_1}{2}-1}\mathbf{d}\cdot x\right\}.$$
{\bf Step 7.} If $C_1=2 (1 + b - a) = D_1=\sqrt{(N - 2 - 2b)^2 + 4 (N - 1)} - (N - 2 - 2b)$, then
    \begin{align*}
        &\int_{\mathbb{R}^N} |\nabla v|^2 \frac{e^{-\frac{2|x|^{1+b-a}}{1+b-a}}}{|x|^{2b}} \, dx 
        - D_1 \inf_{c \in \mathbb{R}} \int_{\mathbb{R}^N} |v - c|^2 \frac{e^{-\frac{2|x|^{1+b-a}}{1+b-a}}}{|x|^{1+b+a}}  \, dx \\
        &= \int_{\mathbb{R}^N} |\nabla v_0|^2 \frac{e^{-\frac{2|x|^{1+b-a}}{1+b-a}}}{|x|^{2b}} dx - C_1 \inf_{c \in \mathbb{R}} \int_{\mathbb{R}^N} |v_0 - c|^2 \frac{e^{-\frac{2|x|^{1+b-a}}{1+b-a}}}{|x|^{a + b + 1}} dx \\
        &+\sum_{i=1}^{N} \int_{\mathbb{R}^N} \left( |\nabla v_{1,i}|^2 + c_1 \frac{|v_{1,i}|^2}{|x|^2} \right) \frac{e^{-\frac{2|x|^{1+b-a}}{1+b-a}}}{|x|^{2b}} dx -D_1 \int_{\mathbb{R}^N} |v_{1,i}|^2 \frac{e^{-\frac{2|x|^{1+b-a}}{1+b-a}}}{|x|^{a + b + 1}} dx \\
        &+ \sum_{k=2}^\infty \int_{\mathbb{R}^N} \left( |\nabla v_k|^2 + c_k \frac{|v_k|^2}{|x|^2} \right) \frac{e^{-\frac{2|x|^{1+b-a}}{1+b-a}}}{|x|^{2b}} dx -D_1 \int_{\mathbb{R}^N} |v_k|^2 \frac{e^{-\frac{2|x|^{1+b-a}}{1+b-a}}}{|x|^{a + b + 1}} dx \\
        &\geq C_1 \inf_{c,d \in \mathbb{R}} \int_{\mathbb{R}^N} \left\vert v_0 -c- d|x|^{1+b-a}\right\vert^2 \frac{e^{-\frac{2|x|^{1+b-a}}{1+b-a}}}{|x|^{a+b+1}} dx \\
        &+ 2(1+b-a)\sum_{i=1}^{N}\inf_{d_i \in \mathbb{R}}\int_{\mathbb{R}^N} |v_{1,i}\phi_{1,i}- d_i|x|^{\frac{D_1}{2}}\phi_{1,i}|^2 \frac{e^{-\frac{2|x|^{1+b -a}}{1 + b - a}}}{|x|^{a + b + 1}} dx\\
        &+ \left(D_2 -D_1\right)\sum_{k=2}^\infty \int_{\mathbb{R}^N} |v_k|^2 \frac{e^{-\frac{2|x|^{1+b-a}}{1+b-a}}}{|x|^{a + b + 1}} dx\\ 
        &\geq \min\left\{D_1,D_2-D_1 \right\}\\
        &\times\inf_{c, d \in \mathbb{R},\mathbf{d} \in \mathbb{R}^N}\int_{\mathbb{R}^N} \left\vert v-c -d|x|^{1+b-a}- |x|^{\frac{D_1}{2}-1}\mathbf{d} \cdot x\right\vert^2 \frac{e^{-\frac{2|x|^{1+b-a}}{1+b-a}}}{|x|^{a+b+1}} \, dx.
    \end{align*}
    Here, we used \eqref{rad2}, \eqref{k1}, and \eqref{k2}.
    Once again, by using Theorem \ref{lemma 0.2}, we obtain 
    \begin{align*}
        &\int_{\mathbb{R}^N} |\nabla v|^2 \frac{e^{-\frac{2|x|^{1+b-a}}{1+b-a}}}{|x|^{2b}} \, dx 
        - D_1 \inf_{c \in \mathbb{R}} \int_{\mathbb{R}^N} |v - c|^2 \frac{e^{-\frac{2|x|^{1+b-a}}{1+b-a}}}{|x|^{1+b+a}}  \, dx \\
        &\geq \min\left\{\frac{1}{2},1-\frac{D_1}{D_2} \right\}\\
        &\times\inf_{c, d \in \mathbb{R},\mathbf{d} \in \mathbb{R}^N}\int_{\mathbb{R}^N} \left\vert \nabla(v-c -d|x|^{1+b-a}- |x|^{\frac{D_1}{2}-1}\mathbf{d} \cdot x)\right\vert^2 \frac{e^{-\frac{2|x|^{1+b-a}}{1+b-a}}}{|x|^{2b}} \, dx.
    \end{align*}
    To use Theorem \ref{lemma 0.2}, we choose $$L_1=\text{span} \left\{|x|^{\frac{D_1}{2}-1}\mathbf{d}\cdot x, L_1^{(\alpha/2-1)}\left(\dfrac{2}{1+b-a}|x|^{1+b-a}\right) \right\}.$$
Indeed, in this case, the function $v$ is represented as follows
\begin{align*}
    v=v_0(r)+\sum_{k=1}^\infty \sum_{j=1}^{N_k} v_{k,j}(r)\phi_{k,j}(\sigma)&=v_{0}(r)+\sum_{j=1}^N v_{1,j}(r)\phi_{1,j}(\sigma)+ \sum_{k=2}^\infty \sum_{j=1}^{N_k} v_{k,j}(r)\phi_{k,j}(\sigma)\\
    &=\beta\\
    &+\sum_{j=1}^N \alpha_jr^{\frac{D_1}{2}}\phi_{1,j}(\sigma)+\gamma L_1^{(\alpha/2-1)}\left(\dfrac{2}{1+b-a}|x|^{1+b-a}\right)\\
    &+\left(v_{0}(r)-\beta-\gamma L_1^{(\alpha/2-1)}\left(\dfrac{2}{1+b-a}|x|^{1+b-a}\right)\right)\\
    &+\sum_{j=1}^N \left(v_{1,j}(r)-\alpha_j r^{\frac{D_1}{2}}\right)\phi_{1,j}(\sigma)+ \sum_{k=2}^\infty \sum_{j=1}^{N_k} v_{k,j}(r)\phi_{k,j}(\sigma),
\end{align*}
where $\beta$ and $\alpha_j$ are determined as in the case $C_1>D_1$. Besides, by defining  $$\gamma=\dfrac{\int_{\mathbb{R}^N} L_1^{(\alpha/2-1)}\left(\dfrac{2}{1+b-a}|x|^{1+b-a}\right) v_{0}(r)d\nu}{\int_{\mathbb{R}^N} \left(L_1^{(\alpha/2-1)}\left(\dfrac{2}{1+b-a}|x|^{1+b-a}\right)\right)^2d\nu},$$
the orthogonality in $L^2(\mathbb{R}^N, d\nu)$ is simply verified. With such constants $\beta, \alpha_j$, and $\gamma$, to check the orthogonality in $W^{1,2}(\mathbb{R}^N, d\mu)$, it is enough to prove that 
$$\int_{\mathbb{R}^N}\nabla \left( L_1^{(\alpha/2-1)}\left(\dfrac{2}{1+b-a}|x|^{1+b-a}\right)\right) \nabla (v_{0}(r))d\mu=C_1 \int_{\mathbb{R}^N} L_1^{(\alpha/2-1)}\left(\dfrac{2}{1+b-a}|x|^{1+b-a}\right) v_0(r) d\nu.$$
However, it is obvious since we have proved that $C_1$ is the sharp constant of the inequality
$$\int_{\mathbb{R}^N}|\nabla u|^2d\mu \geq C \int_{\mathbb{R}^N} |u|^2d\nu,$$
with $ L_1^{(\alpha/2-1)}\left(\dfrac{2}{1+b-a}|x|^{1+b-a}\right)$ is an optimizer. 
\end{proof}

Using the Kelvin transform and suitable changes of variables, we also derive sharp weighted Poincar\'e inequalities with Gaussian-type measures on other domains. For brevity, we state here only the sharp inequalities and their constants.

\begin{theorem}\label{lemma 2}
    Assume $1 + b - a < 0$ and $b \geq \frac{N-2}{2}$. Then
    \begin{equation}
        \int_{\mathbb{R}^N} |\nabla v|^2 \frac{e^{\frac{2|x|^{1+b-a}}{1+b-a}}}{|x|^{2b}} \, dx 
        \geq C_{2,PI}(a,b,N) \inf_{c \in \mathbb{R}} \int_{\mathbb{R}^N} |v - c|^2 \frac{e^{\frac{2|x|^{1+b-a}}{1+b-a}}}{|x|^{1+b+a}} \, dx,
    \end{equation}
    where $C_{2, PI}(a,b,N) = C_{PI}(N - a, N - 2 - b, N).$
\end{theorem}
\begin{proof}
    By Theorem \ref{lemma 1}, for $1 + n - m > 0$ and $n \leq \frac{N-2}{2}$,
    \[
    \int_{\mathbb{R}^N} |\nabla v(x)|^2 \frac{e^{-\frac{2 |x|^{1+n-m}}{1+n-m}}}{|x|^{2n}} \, dx \geq C_{PI}(m,n,N) \inf_{c \in \mathbb{R}} \int_{\mathbb{R}^N} |v(x) - c|^2 \frac{e^{-\frac{2 |x|^{1+n-m}}{1+n-m}}}{|x|^{1+n+m}} \, dx.
    \]
    Using the Kelvin transform $x \mapsto y = \frac{x}{|x|^2}$ with Jacobian $dx = \frac{dy}{|y|^{2N}}$ and the chain rule, one obtains
    \[
    \int_{\mathbb{R}^N} |\nabla_x v(x)|^2 \frac{e^{-\frac{2|x|^{1+n-m}}{1+n-m}}}{|x|^{2n}} \, dx =
    \int_{\mathbb{R}^N} |\nabla v(y/|y|^2)|^2 \frac{e^{-\frac{2 |y|^{-(1+n-m)}}{1+n-m}}}{|y|^{2N - 2n}} \, dy.
    \]
    Setting $w(y) = v(y/|y|^2)$ yields
    \[
    |\nabla_y w(y)|^2 = \frac{1}{|y|^4} |\nabla v(y/|y|^2)|^2,
    \]
    and thus
    \[
    \int_{\mathbb{R}^N} |\nabla_y w(y)|^2 \frac{e^{-\frac{2|y|^{-(1+n-m)}}{1+n-m}}}{|y|^{2N - 2n - 4}} \, dy \geq 
    C_{PI}(m,n,N) \inf_{c \in \mathbb{R}} \int_{\mathbb{R}^N} |w(y) - c|^2 \frac{e^{-\frac{2|y|^{-(1+n-m)}}{1+n-m}}}{|y|^{2N - (1 + n + m)}} \, dy.
    \]

    Setting $2b := 2N - 2n - 4$ and $a := N - m$, we observe that $b = N - n - 2 \geq \frac{N-2}{2}$ and $1 + b - a = m - n - 1 < 0$.
    Hence, the lemma follows:
    \[
    \int_{\mathbb{R}^N} |\nabla_y w(y)|^2 \frac{e^{\frac{2|y|^{1+b-a}}{1+b-a}}}{|y|^{2b}} \, dy \geq C_{PI}(N - a, N - 2 - b, N) \inf_{c \in \mathbb{R}} \int_{\mathbb{R}^N} |w(y) - c|^2 \frac{e^{\frac{2|y|^{1+b-a}}{1+b-a}}}{|y|^{a+b+1}} \, dy.
    \]
\end{proof}

\begin{theorem}\label{lemma 1.3}
    Assume $1 + b - a < 0$ and $b \leq \frac{N-2}{2}$. Then
    \begin{equation}
        \int_{\mathbb{R}^N} |\nabla v|^2 \frac{e^{\frac{2|x|^{1+b-a}}{1+b-a}}}{|x|^{2N - 2b - 4}} \, dx \geq C_{3, PI}(a,b,N) \inf_{c \in \mathbb{R}} \int_{\mathbb{R}^N} |v - c|^2 \frac{e^{\frac{2|x|^{1+b-a}}{1+b-a}}}{|x|^{2N + a - 3b - 3}} \, dx,
    \end{equation}
    where $C_{3,PI}(a,b,N) = C_{PI}(-a + 2b + 2, b, N).$
\end{theorem}
\begin{proof}
    From Theorem \ref{lemma 2}, for $1 + n - m < 0$ and $n \geq \frac{N-2}{2}$,
    \[
    \int_{\mathbb{R}^N} |\nabla v|^2 \frac{e^{\frac{2 |x|^{1+n-m}}{1+n-m}}}{|x|^{2n}} \, dx \geq C_{2, PI}(m,n,N) \inf_{c \in \mathbb{R}} \int_{\mathbb{R}^N} |v - c|^2 \frac{e^{\frac{2 |x|^{1+n-m}}{1+n-m}}}{|x|^{1+n+m}} \, dx.
    \]
    Setting $2 N - 2 b - 4 = 2 n$ and $1 + b - a = 1 + n - m$, we have $b = N - 2 - n \leq \frac{N-2}{2}$ and $1 + b - a < 0$. Therefore,
    \[
    \int_{\mathbb{R}^N} |\nabla v|^2 \frac{e^{\frac{2 |x|^{1+b-a}}{1+b-a}}}{|x|^{2N - 2b - 4}} \, dx \geq C_{2, PI}(N + a - 2 b - 2, N - b - 2, N) \inf_{c \in \mathbb{R}} \int_{\mathbb{R}^N} |v - c|^2 \frac{e^{\frac{2 |x|^{1+b-a}}{1+b-a}}}{|x|^{2N + a - 3 b - 3}} \, dx,
    \]
    which implies $C_{3,PI}(a,b,N) = C_{PI}(-a + 2 b + 2, b, N)$.
\end{proof}

\begin{theorem}\label{lemma 1.4}
    Assume $1 + b - a > 0$ and $b \geq \frac{N-2}{2}$. Then
    \begin{equation}
        \int_{\mathbb{R}^N} |\nabla v|^2 \frac{e^{-\frac{2|x|^{1+b-a}}{1+b-a}}}{|x|^{2N - 2b - 4}} \, dx \geq C_{4, PI}(a,b,N) \inf_{c \in \mathbb{R}} \int_{\mathbb{R}^N} |v - c|^2 \frac{e^{-\frac{2|x|^{1+b-a}}{1+b-a}}}{|x|^{2N + a - 3b - 3}} \, dx,
    \end{equation}
    where $C_{4,PI}(a,b,N) = C_{PI}(N + a - 2 b - 2, N - b - 2, N).$
\end{theorem}
\begin{proof}
    By Theorem \ref{lemma 1}, for $1 + n - m > 0$ and $n \leq \frac{N-2}{2}$,
    \[
    \int_{\mathbb{R}^N} |\nabla v|^2 \frac{e^{-\frac{2|x|^{1+n-m}}{1+n-m}}}{|x|^{2n}} \, dx \geq C_{PI}(m,n,N) \inf_{c \in \mathbb{R}} \int_{\mathbb{R}^N} |v-c|^2 \frac{e^{-\frac{2|x|^{1+n-m}}{1+n-m}}}{|x|^{1+n+m}} \, dx.
    \]
    Choosing $a,b$ such that $2 N - 2 b - 4 = 2 n$ and $1 + b - a = 1 + n - m$, we have $b = N - 2 - n \geq \frac{N-2}{2}$ and $1 + b - a > 0$. Thus,
    \[
    \int_{\mathbb{R}^N} |\nabla v|^2 \frac{e^{-\frac{2|x|^{1+b-a}}{1+b-a}}}{|x|^{2N - 2b - 4}} \, dx \geq C_{PI}(N + a - 2 b - 2, N - b - 2, N) \inf_{c \in \mathbb{R}} \int_{\mathbb{R}^N} |v-c|^2 \frac{e^{-\frac{2|x|^{1+b-a}}{1+b-a}}}{|x|^{2N + a - 3 b - 3}} \, dx,
    \]
    so that
    \[
    C_{4, PI}(a,b,N) = C_{PI}(N + a - 2 b - 2, N - b - 2, N).
    \]
\end{proof}

\section{A complete characterization of the Sharp Stability estimates of the $L^2$-Caffarelli-Kohn-Nirenberg inequalities--Proof of Theorem \ref{staL2CKN}}
We first recall the following result in \cite{DFLL23}:
\begin{theorem}\label{L2CKNI}
    Let $\Omega$ be an open set in $\mathbb{R}^N$, $N \geq 1$, $\alpha > 0$, $A \in C^1(\Omega)$ and $\mathbf{X} \in C^1(\Omega, \mathbb{R}^N)$. Then for any $u \in C_0^1(\Omega)$, we have
\begin{align*}
    &|\alpha|^2 \int_\Omega A |\nabla u|^2 \, dx + \frac{1}{|\alpha|^2} \int_\Omega A |\mathbf{X}|^2 |u|^2 \, dx  \\
    &= -\int_\Omega \mathrm{div}(A\mathbf{X}) |u|^2 \, dx + \int_\Omega A \left| \alpha \nabla u - \frac{1}{\alpha} u \mathbf{X} \right|^2 dx.
\end{align*}
As a consequence,
\begin{align*}
    \int_\Omega A |\nabla u|^2 dx 
    &= \int_\Omega \left( -\mathrm{div}(A\mathbf{X}) - A |\mathbf{X}|^2 \right) |u|^2 dx + \int_\Omega A |\nabla u + \mathbf{X} u|^2 dx.
\end{align*}
Also, if $A \geq 0$, then for any $u \in C_0^1(\Omega)\setminus \{0\}$:
\begin{align*}
    &\left( \int_\Omega A |\nabla u|^2 dx \right)^{\frac{1}{2}} \left( \int_\Omega A |\mathbf{X}|^2 |u|^2 dx \right)^{\frac{1}{2}} + \frac{1}{2} \int_\Omega \mathrm{div}(A\mathbf{X}) |u|^2 dx \\
    &= \frac{1}{2} \int_\Omega A \left| \left( \frac{\int_\Omega A |\mathbf{X}|^2 |u|^2 dx}{\int_\Omega A |\nabla u|^2 dx} \right)^{\frac{1}{4}} \nabla u 
    - \left( \frac{\int_\Omega A |\nabla u|^2 dx}{\int_\Omega A |\mathbf{X}|^2 |u|^2 dx} \right)^{\frac{1}{4}} u \mathbf{X} \right|^2 dx.
\end{align*}
\end{theorem}
We are now ready to establish several sharp stability estimates for the $L^2$-CKN inequalities and present a proof of Theorem \ref{staL2CKN}. More precisely, Theorem \ref{staL2CKN} will be proved using the following results.

Recall that \[
C_{PI}(a,b,N) = \min\left\{ C_1:=2(1+b-a),\, D_1:= \sqrt{(N-2-2b)^2 + 4(N-1)} - (N-2-2b) \right\}.
\]
Denote
$$\delta_1 (u)=\int_{\mathbb{R}^{N}}\dfrac{|\nabla u|^{2}}{|x|^{2b}}dx+\int_{\mathbb{R}%
^{N}}\dfrac{|u|^{2}}{|x|^{2a}}dx-2C\left(N,a,b\right)  \int_{\mathbb{R}^{N}%
}\dfrac{|u|^{2}}{|x|^{a+b+1}}dx$$ and 
$$\delta_2 (u)=\left(  \int_{\mathbb{R}^{N}}\dfrac{|\nabla u|^{2}}{|x|^{2b}}dx\right)
^{\frac{1}{2}}\left(  \int_{\mathbb{R}^{N}}\dfrac{|u|^{2}}{|x|^{2a}}dx\right)
^{\frac{1}{2}}-C\left( N,a,b\right) \left(  \int
_{\mathbb{R}^{N}}\dfrac{|u|^{2}}{|x|^{a+b+1}}dx\right)$$ with \[
C(N,a,b)
= \max \left\{
\frac{|N-(a+b+1)|}{2}, \,
\frac{|N-(3b - a + 3)|}{2}
\right\}.
\]
Then we have
\begin{theorem}
Let $b+1-a>0$ and $b\leq\dfrac{N-2}{2}$. Then for $u\in C_{0}^{\infty}\left(
\mathbb{R}^{N} \right)  :$
\begin{equation*}
\delta_1(u)  \geq C_{PI}(a,b,N)\inf_{c \in \mathbb{R}}\int_{\mathbb{R}^N}\left|u-ce^{-\frac{|x|^{1+b-a}}{1+b-a}}\right|^2\dfrac{dx}{|x|^{1+b+a}}
\end{equation*}
and
\begin{equation*}
\delta_2(u)\geq \frac{C_{PI}(a,b,N)}{2}\inf_{c \in \mathbb{R}, \lambda>0}\displaystyle\int_{\mathbb{R}^N}\dfrac{\left|u-ce^{-\frac{\lambda|x|^{b+1-a}}{(b+1-a)}}\right|^2}{|x|^{1+b+a}}dx.
\end{equation*}
Moreover, if $C_1 < D_1$, then 
\begin{align*}
&\delta_1(u)-C_{PI}(a, b, N)\inf_{c \in \mathbb{R}}\int_{\mathbb{R}^N}\left|u-ce^{-\frac{|x|^{1+b-a}}{1+b-a}}\right|^2\dfrac{dx}{|x|^{1+b+a}}\\
&\geq \min\left\{C_1,D_1-C_1\right\} \inf_{c,d \in \mathbb{R}} \int_{\mathbb{R}^N} \left\vert u -(c+  d|x|^{1+b-a})e^{-\frac{|x|^{1+b-a}}{1+b-a}} \right\vert^2 \frac{dx}{|x|^{a+b+1}}
\end{align*}
and 
\begin{align*}
&\delta_2(u)-\frac{C_{PI}(a,b,N)}{2}\inf_{c \in \mathbb{R}, \lambda>0}\displaystyle\int_{\mathbb{R}^N}\dfrac{\left|u-ce^{-\frac{\lambda|x|^{b+1-a}}{(b+1-a)}}\right|^2}{|x|^{1+b+a}}dx\\
&\geq \frac{\min\left\{C_1,D_1-C_1\right\}}{2} \inf_{c,d \in \mathbb{R},\lambda > 0} \int_{\mathbb{R}^N} \left\vert u -(c+  d|x|^{1+b-a})e^{-\frac{\lambda|x|^{1+b-a}}{1+b-a}} \right\vert^2 \frac{dx}{|x|^{a+b+1}}.
\end{align*}
If $C_1 > D_1$, then
\begin{align*}
&\delta_1(u)-C_{PI}(a, b, N)\inf_{c \in \mathbb{R}}\int_{\mathbb{R}^N}\left|u-ce^{-\frac{|x|^{1+b-a}}{1+b-a}}\right|^2\dfrac{dx}{|x|^{1+b+a}}\\
&\geq \min\left\{C_1- D_1,D_2-D_1 \right\} \inf_{c \in \mathbb{R},\mathbf{d} \in \mathbb{R}^N}\int_{\mathbb{R}^N} \left\vert u-(c + |x|^{\frac{D_1}{2}-1}\mathbf{d} \cdot x)e^{-\frac{|x|^{1+b-a}}{1+b-a}}\right\vert^2 \frac{dx}{|x|^{a+b+1}}
\end{align*} 
and 
\begin{align*}
&\delta_2(u)-\frac{C_{PI}(a,b,N)}{2}\inf_{c \in \mathbb{R}, \lambda>0}\displaystyle\int_{\mathbb{R}^N}\dfrac{\left|u-ce^{-\frac{\lambda|x|^{b+1-a}}{(b+1-a)}}\right|^2}{|x|^{1+b+a}}dx\\
&\geq \frac{\min\left\{C_1- D_1,D_2-D_1 \right\}}{2} \inf_{c \in \mathbb{R},\lambda>0,\mathbf{d} \in \mathbb{R}^N}\int_{\mathbb{R}^N} \left\vert u-(c + |x|^{\frac{D_1}{2}-1}\mathbf{d} \cdot x)e^{-\frac{\lambda|x|^{1+b-a}}{1+b-a}}\right\vert^2 \frac{dx}{|x|^{a+b+1}}.
\end{align*}
If $C_1=D_1$, then 
\begin{align*}
&\delta_1(u)-C_{PI}(a, b, N)\inf_{c \in \mathbb{R}}\int_{\mathbb{R}^N}\left|u-ce^{-\frac{|x|^{1+b-a}}{1+b-a}}\right|^2\dfrac{dx}{|x|^{1+b+a}}\\
&\geq \min\left\{D_1,D_2-D_1 \right\} \inf_{c, d \in \mathbb{R},\mathbf{d} \in \mathbb{R}^N}\int_{\mathbb{R}^N} \left\vert u-(c +d|x|^{1+b-a}+ |x|^{\frac{D_1}{2}-1}\mathbf{d} \cdot x)e^{-\frac{|x|^{1+b-a}}{1+b-a}}\right\vert^2 \frac{dx}{|x|^{a+b+1}}
\end{align*}
and 
\begin{align*}
&\delta_2(u)-\frac{C_{PI}(a,b,N)}{2}\inf_{c \in \mathbb{R}, \lambda>0}\displaystyle\int_{\mathbb{R}^N}\dfrac{\left|u-ce^{-\frac{\lambda|x|^{b+1-a}}{(b+1-a)}}\right|^2}{|x|^{1+b+a}}dx\\
&\geq \frac{\min\left\{D_1,D_2-D_1 \right\}}{2} \inf_{c, d \in \mathbb{R},\lambda >0,\mathbf{d} \in \mathbb{R}^N}\int_{\mathbb{R}^N} \left\vert u-(c +d|x|^{1+b-a}+ |x|^{\frac{D_1}{2}-1}\mathbf{d} \cdot x)e^{-\frac{\lambda|x|^{1+b-a}}{1+b-a}}\right\vert^2 \frac{dx}{|x|^{a+b+1}}.
\end{align*}
Here \[
D_2 = \sqrt{(N - 2 - 2b)^2 + 8N} - (N - 2 - 2b).
\]
\end{theorem}

\begin{proof}
By choosing $A=\frac{1}{|x|^{2b}}$ and $\mathbf{X}=-|x|^{b-a}\frac{x}{|x|}$ in Theorem \ref{L2CKNI}, and using Theorem \ref{lemma 1}, we obtain
\begin{align*}
\delta_1(u)&=\int_{\mathbb{R}^{N}}\dfrac{1}{|x|^{2b}}\left\vert \nabla\left(
u.e^{\frac{|x|^{b+1-a}}{b+1-a}}\right)  \right\vert ^{2}e^{-\frac
{2|x|^{b+1-a}}{b+1-a}}dx\\
& \geq C_{PI}(a, b, N)\inf_{c \in \mathbb{R}}\int_{\mathbb{R}^N}\left|u-ce^{-\frac{|x|^{1+b-a}}{1+b-a}}\right|^2\dfrac{dx}{|x|^{1+b+a}}.
\end{align*}
Similarly, we also get that if $C_1 < D_1$, then 
\begin{align*}
&\delta_1(u)-C_{PI}(a, b, N)\inf_{c \in \mathbb{R}}\int_{\mathbb{R}^N}\left|u-ce^{-\frac{|x|^{1+b-a}}{1+b-a}}\right|^2\dfrac{dx}{|x|^{1+b+a}}\\
&\geq \min\left\{C_1,D_1-C_1\right\} \inf_{c,d \in \mathbb{R}} \int_{\mathbb{R}^N} \left\vert u -(c+  d|x|^{1+b-a})e^{-\frac{|x|^{1+b-a}}{1+b-a}} \right\vert^2 \frac{dx}{|x|^{a+b+1}}.
\end{align*}
If $C_1 > D_1$, then
\begin{align*}
&\delta_1(u)-C_{PI}(a, b, N)\inf_{c \in \mathbb{R}}\int_{\mathbb{R}^N}\left|u-ce^{-\frac{|x|^{1+b-a}}{1+b-a}}\right|^2\dfrac{dx}{|x|^{1+b+a}}\\
&\geq \min\left\{C_1- D_1,D_2-D_1 \right\} \inf_{c \in \mathbb{R},\mathbf{d} \in \mathbb{R}^N}\int_{\mathbb{R}^N} \left\vert u-(c + |x|^{\frac{D_1}{2}-1}\mathbf{d} \cdot x)e^{-\frac{|x|^{1+b-a}}{1+b-a}}\right\vert^2 \frac{dx}{|x|^{a+b+1}}.
\end{align*}
If $C_1=D_1$, then 
\begin{align*}
&\delta_1(u)-C_{PI}(a, b, N)\inf_{c \in \mathbb{R}}\int_{\mathbb{R}^N}\left|u-ce^{-\frac{|x|^{1+b-a}}{1+b-a}}\right|^2\dfrac{dx}{|x|^{1+b+a}}\\
&\geq \min\left\{D_1,D_2-D_1 \right\} \inf_{c, d \in \mathbb{R},\mathbf{d} \in \mathbb{R}^N}\int_{\mathbb{R}^N} \left\vert v-(c +d|x|^{1+b-a}+ |x|^{\frac{D_1}{2}-1}\mathbf{d} \cdot x)e^{-\frac{|x|^{1+b-a}}{1+b-a}}\right\vert^2 \frac{dx}{|x|^{a+b+1}}.
\end{align*}
For the deficit $\delta_2$, by applying a standard scaling argument to Theorem \ref{lemma 1}, for any $\lambda>0$, we have
$$\lambda^{1+b-a}\int_{\mathbb{R}^N} |\nabla v|^2 \frac{e^{-\frac{2|x|^{1+b-a}}{(1+b-a)\lambda^{1+b-a}}}}{|x|^{2b}} \, dx \geq C_{PI}(a,b,N) \inf_{c \in \mathbb{R}} \int_{\mathbb{R}^N} |v-c|^2 \frac{e^{-\frac{2|x|^{1+b-a}}{(1+b-a)\lambda^{1+b-a}}}}{|x|^{1+b+a}}  \, dx.$$
Then, once again, by choosing $A=\frac{1}{|x|^{2b}}$ and $\mathbf{X}=-|x|^{b-a}\frac{x}{|x|}$ in Theorem \ref{L2CKNI}, we get
\begin{align*}
\delta_2(u) &=\frac{1}{2}\lambda^{b-a+1}\int_{\mathbb{R}^{N}}\frac{1}{|x|^{2b}%
}\left\vert \nabla\left(  ue^{\frac{|x|^{b+1-a}}{\left(  b+1-a\right)
\lambda^{b-a+1}}}\right)  \right\vert ^{2}e^{-\frac{2|x|^{b+1-a}}{\left(
b+1-a\right)  \lambda^{b-a+1}}}dx\\
&\geq \frac{C_{PI}(a,b,N)}{2}\inf_{c \in \mathbb{R}}\displaystyle\int_{\mathbb{R}^N}\dfrac{\left|u-ce^{-\frac{|x|^{b+1-a}}{(b+1-a)\lambda^{b-a+1}}}\right|^2}{|x|^{1+b+a}}dx\\
&\geq \frac{C_{PI}(a,b,N)}{2}\inf_{c \in \mathbb{R}, \lambda>0}\displaystyle\int_{\mathbb{R}^N}\dfrac{\left|u-ce^{-\frac{\lambda|x|^{b+1-a}}{(b+1-a)}}\right|^2}{|x|^{1+b+a}}dx.
\end{align*}

Similarly, we can also deduce other estimates for $\delta_2$.
\end{proof}

Using a similar method as in the theorem above, we derive sharp stability 
estimates for the $L^2$-CKN inequalities and their refinements on other domains. 
For conciseness, we present only the sharp inequalities and their constants here.

\begin{theorem}
Let $b+1-a<0$ and $b\geq\dfrac{N-2}{2}$. Then for $u\in C_{0}^{\infty}\left(
\mathbb{R}^{N}\setminus\left\{  0\right\}  \right)  :$
\begin{equation*}
 \delta_1(u)\geq C_{2,PI}(a,b,N)\inf_{c \in \mathbb{R}}\int_{\mathbb{R}^N}\left|u-ce^{\frac{|x|^{b+1-a}}{b+1-a}}\right|^2\dfrac{dx}{|x|^{1+b+a}}.
\end{equation*}
Also, for $u\in C_{0}^{\infty}\left(  \mathbb{R}^{N}\setminus\left\{
0\right\}  \right)  \setminus\left\{  0\right\} $:
\begin{equation*}
\delta_2(u)\geq \frac{C_{2,PI}(a,b,N)}{2}\inf_{c \in \mathbb{R}, \lambda>0}\displaystyle\int_{\mathbb{R}^N}\dfrac{\left|u-ce^{\frac{\lambda|x|^{b+1-a}}{(b+1-a)}}\right|^2}{|x|^{1+b+a}}dx.
\end{equation*}
\end{theorem}

\begin{proof}
By choosing $A=\frac{1}{|x|^{2b}}$ and $\mathbf{X}=|x|^{b-a}\frac{x}{|x|}$ in Theorem \ref{L2CKNI}, and using Theorem \ref{lemma 2}, we get
\begin{align*}
&  \int_{\mathbb{R}^{N}}\dfrac{|\nabla u|^{2}}{|x|^{2b}}dx+\int_{\mathbb{R}%
^{N}}\dfrac{|u|^{2}}{|x|^{2a}}dx-\left(  a+b+1-N\right)  \int_{\mathbb{R}^{N}%
}\dfrac{|u|^{2}}{|x|^{a+b+1}}dx\\
&  =\int_{\mathbb{R}^{N}}\dfrac{1}{|x|^{2b}}\left\vert \nabla\left(
u.e^{-\frac{|x|^{b+1-a}}{b+1-a}}\right)  \right\vert ^{2}e^{\frac
{2|x|^{b+1-a}}{b+1-a}}dx\\
&\geq C_{2,PI}(a,b,N)\inf_{c \in \mathbb{R}}\int_{\mathbb{R}^N}\left|u-ce^{\frac{|x|^{b+1-a}}{b+1-a}}\right|^2\dfrac{dx}{|x|^{1+b+a}}.
\end{align*}
For the second one, by applying a standard scaling argument to Theorem \ref{lemma 2}, for any $\lambda>0$, we have
$$\lambda^{1+b-a}\int_{\mathbb{R}^N} |\nabla v|^2 \frac{e^{\frac{2|x|^{1+b-a}}{(1+b-a)\lambda^{1+b-a}}}}{|x|^{2b}} \, dx \geq C_{2,PI}(a,b,N) \inf_{c \in \mathbb{R}} \int_{\mathbb{R}^N} |v-c|^2 \frac{e^{\frac{2|x|^{1+b-a}}{(1+b-a)\lambda^{1+b-a}}}}{|x|^{1+b+a}}  \, dx.$$
Then, by choosing $A=\frac{1}{|x|^{2b}}$ and $\mathbf{X}=-|x|^{b-a}\frac{x}{|x|}$ in Theorem \ref{L2CKNI}, we obtain
\begin{align*}
&  \left(  \int_{\mathbb{R}^{N}}\dfrac{|\nabla u|^{2}}{|x|^{2b}}dx\right)
^{\frac{1}{2}}\left(  \int_{\mathbb{R}^{N}}\dfrac{|u|^{2}}{|x|^{2a}}dx\right)
^{\frac{1}{2}}-\left( \dfrac{a+b+1-N}{2}\right) \left(  \int
_{\mathbb{R}^{N}}\dfrac{|u|^{2}}{|x|^{a+b+1}}dx\right) \\
&  =\frac{1}{2}\lambda^{b-a+1}\int_{\mathbb{R}^{N}}\frac{1}{|x|^{2b}%
}\left\vert \nabla\left(  ue^{-\frac{|x|^{b+1-a}}{\left(  b+1-a\right)
\lambda^{b-a+1}}}\right)  \right\vert ^{2}e^{\frac{2|x|^{b+1-a}}{\left(
b+1-a\right)  \lambda^{b-a+1}}}dx\\
&\geq \frac{C_{2,PI}(a,b,N)}{2}\inf_{c \in \mathbb{R}}\displaystyle\int_{\mathbb{R}^N}\dfrac{\left|u-ce^{\frac{|x|^{b+1-a}}{(b+1-a)\lambda^{b-a+1}}}\right|^2}{|x|^{1+b+a}}dx\\
&\geq \frac{C_{2,PI}(a,b,N)}{2}\inf_{c \in \mathbb{R}, \lambda>0}\displaystyle\int_{\mathbb{R}^N}\dfrac{\left|u-ce^{\frac{\lambda|x|^{b+1-a}}{(b+1-a)}}\right|^2}{|x|^{1+b+a}}dx.
\end{align*}

\end{proof}

\begin{theorem}
Let $b+1-a<0$ and $b\leq\dfrac{N-2}{2}$. Then for $u\in C_{0}^{\infty}\left(
\mathbb{R}^{N}\setminus\left\{  0\right\}  \right)  :$
\begin{equation*}
 \delta_1(u)\geq C_{3,PI}(a,b,N)\inf_{c \in \mathbb{R}}\int_{\mathbb{R}^N}\left|u-ce^{\frac{|x|^{b+1-a}}{b+1-a}}|x|^{2b+2-N}\right|^2\dfrac{dx}{|x|^{1+b+a}}.
\end{equation*}

Also, for $u\in C_{0}^{\infty}\left(  \mathbb{R}^{N}\setminus\left\{
0\right\}  \right)  \setminus\left\{  0\right\} $:
\begin{equation*}
 \delta_2(u)\geq \frac{C_{3,PI}(a,b,N)}{2}\inf_{c \in \mathbb{R}, \lambda>0}\displaystyle\int_{\mathbb{R}^N}\dfrac{\left|u-ce^{\frac{\lambda|x|^{b+1-a}}{(b+1-a)}}|x|^{2b+2-N}\right|^2}{|x|^{1+b+a}}dx.
\end{equation*}
\end{theorem}

\begin{proof}
Similarly as before, this time we choose $A = \frac{1}{|x|^{2b}}$ and $\mathbf{X} = \left( |x|^{b-a} - (N-2b-2)\frac{1}{|x|} \right) \frac{x}{|x|}$ in Theorem~\ref{L2CKNI}, and use Theorem~\ref{lemma 1.3} to obtain
\begin{align*}
&  \int_{\mathbb{R}^{N}}\dfrac{|\nabla u|^{2}}{|x|^{2b}}dx+\int_{\mathbb{R}%
^{N}}\dfrac{|u|^{2}}{|x|^{2a}}dx-\left(  N-3b+a-3\right)  \int_{\mathbb{R}%
^{N}}\dfrac{|u|^{2}}{|x|^{a+b+1}}dx\\
&  =\int_{\mathbb{R}^{N}}\dfrac{1}{|x|^{2N-2b-4}}\left\vert \nabla\left(
u|x|^{N-2b-2}e^{-\frac{|x|^{b+1-a}}{b+1-a}}\right)  \right\vert ^{2}%
e^{\frac{2|x|^{b+1-a}}{b+1-a}}dx\\
&\geq C_{3,PI}(a,b,N)\inf_{c \in \mathbb{R}}\int_{\mathbb{R}^N}\left|u-ce^{\frac{|x|^{b+1-a}}{b+1-a}}|x|^{2b+2-N}\right|^2\dfrac{dx}{|x|^{1+b+a}}.
\end{align*}
For the second one, by applying a standard scaling argument to Theorem \ref{lemma 1.3}, for any $\lambda>0$, we have
$$\lambda^{1+b-a}\int_{\mathbb{R}^N} |\nabla v|^2 \frac{e^{\frac{2|x|^{1+b-a}}{(1+b-a)\lambda^{1+b-a}}}}{|x|^{2N-2b-4}} \, dx \geq C_{3,PI}(a,b,N) \inf_{c \in \mathbb{R}} \int_{\mathbb{R}^N} |v-c|^2 \frac{e^{\frac{2|x|^{1+b-a}}{(1+b-a)\lambda^{1+b-a}}}}{|x|^{2N+a-3b-3}}  \, dx.$$
Then, again by choosing $A=\frac{1}{|x|^{2b}}$ and $\mathbf{X}=\left(|x|^{b-a}-(N-2b-2)\frac{1}{|x|}\right)\frac{x}{|x|}$ in Theorem \ref{L2CKNI}, we get
\begin{align*}
&  \left(  \int_{\mathbb{R}^{N}}\dfrac{|\nabla u|^{2}}{|x|^{2b}}dx\right)
^{\frac{1}{2}}\left(  \int_{\mathbb{R}^{N}}\dfrac{|u|^{2}}{|x|^{2a}}dx\right)
^{\frac{1}{2}}-\left( \dfrac{N-3b+a-3}{2}\right) \left(
\int_{\mathbb{R}^{N}}\dfrac{|u|^{2}}{|x|^{a+b+1}}dx\right) \\
&  =\frac{1}{2}\lambda^{b-a+1}\int_{\mathbb{R}^{N}}\dfrac{1}{|x|^{2N-2b-4}%
}\left\vert \nabla\left(  u\left\vert x\right\vert ^{N-2b-2}e^{-\frac
{|x|^{b+1-a}}{\left(  b+1-a\right)  \lambda^{b-a+1}}}\right)  \right\vert
^{2}e^{\frac{2|x|^{b+1-a}}{\left(  b+1-a\right)  \lambda^{b-a+1}}}dx\\
&\geq \frac{C_{3,PI}(a,b,N)}{2}\inf_{c \in \mathbb{R}}\displaystyle\int_{\mathbb{R}^N}\dfrac{\left|u-ce^{\frac{|x|^{b+1-a}}{(b+1-a)\lambda^{b-a+1}}}|x|^{2b+2-N}\right|^2}{|x|^{1+b+a}}dx\\
&\geq \frac{C_{3,PI}(a,b,N)}{2}\inf_{c \in \mathbb{R}, \lambda>0}\displaystyle\int_{\mathbb{R}^N}\dfrac{\left|u-ce^{\frac{\lambda|x|^{b+1-a}}{(b+1-a)}}|x|^{2b+2-N}\right|^2}{|x|^{1+b+a}}dx.
\end{align*}
\end{proof}

\begin{theorem}
Let $b+1-a>0$ and $b\geq\dfrac{N-2}{2}$. Then for $u\in C_{0}^{\infty}\left(
\mathbb{R}^{N}\setminus\left\{  0\right\}  \right)  :$
\begin{equation*}
 \delta_1(u)\geq C_{4,PI}(a,b,N)\inf_{c \in \mathbb{R}}\int_{\mathbb{R}^N}\left|u-ce^{-\frac{|x|^{b+1-a}}{b+1-a}}|x|^{2b+2-N}\right|^2\dfrac{dx}{|x|^{1+b+a}}.
\end{equation*}

Also, for $u\in C_{0}^{\infty}\left(  \mathbb{R}^{N}\setminus\left\{
0\right\}  \right)  \setminus\left\{  0\right\} $:
\begin{equation*}
 \delta_2(u)\geq \frac{ C_{4,PI}(a,b,N)}{2}\inf_{c \in \mathbb{R}, \lambda>0}\displaystyle\int_{\mathbb{R}^N}\dfrac{\left|u-ce^{-\frac{\lambda|x|^{b+1-a}}{(b+1-a)}}|x|^{2b+2-N}\right|^2}{|x|^{1+b+a}}dx.
\end{equation*}
\end{theorem}

\begin{proof}
This time we choose $A = \frac{1}{|x|^{2b}}$ and $\mathbf{X} = \left( -|x|^{b-a} - (N-2b-2)\frac{1}{|x|} \right) \frac{x}{|x|}$ in Theorem~\ref{L2CKNI}, and use Theorem \ref{lemma 1.4} to get
\begin{align*}
&  \int_{\mathbb{R}^{N}}\dfrac{|\nabla u|^{2}}{|x|^{2b}}dx+\int_{\mathbb{R}%
^{N}}\dfrac{|u|^{2}}{|x|^{2a}}dx-\left(  3b-a+3-N\right)  \int_{\mathbb{R}%
^{N}}\dfrac{|u|^{2}}{|x|^{a+b+1}}dx\\
&  =\int_{\mathbb{R}^{N}}\dfrac{1}{|x|^{2N-2b-4}}\left\vert \nabla\left(
u|x|^{N-2b-2}e^{\frac{|x|^{b+1-a}}{b+1-a}}\right)  \right\vert ^{2}%
e^{-\frac{2|x|^{b+1-a}}{b+1-a}}dx\\
&\geq  C_{4,PI}(a,b,N)\inf_{c \in \mathbb{R}}\int_{\mathbb{R}^N}\left|u-ce^{-\frac{|x|^{b+1-a}}{b+1-a}}|x|^{2b+2-N}\right|^2\dfrac{dx}{|x|^{1+b+a}}.
\end{align*}
For the second one, by applying a standard scaling argument to Theorem \ref{lemma 1.4}, for any $\lambda>0$, we have
$$\lambda^{1+b-a}\int_{\mathbb{R}^N} |\nabla v|^2 \frac{e^{-\frac{2|x|^{1+b-a}}{(1+b-a)\lambda^{1+b-a}}}}{|x|^{2N-2b-4}} \, dx \geq  C_{4,PI}(a,b,N) \inf_{c \in \mathbb{R}} \int_{\mathbb{R}^N} |v-c|^2 \frac{e^{-\frac{2|x|^{1+b-a}}{(1+b-a)\lambda^{1+b-a}}}}{|x|^{2N+a-3b-3}}  \, dx.$$
Then, with $A = \frac{1}{|x|^{2b}}$ and $\mathbf{X} = \left( -|x|^{b-a} - (N-2b-2)\frac{1}{|x|} \right) \frac{x}{|x|}$ in Theorem~\ref{L2CKNI}, we have
\begin{align*}
&  \left(  \int_{\mathbb{R}^{N}}\dfrac{|\nabla u|^{2}}{|x|^{2b}}dx\right)
^{\frac{1}{2}}\left(  \int_{\mathbb{R}^{N}}\dfrac{|u|^{2}}{|x|^{2a}}dx\right)
^{\frac{1}{2}}-\left(\dfrac{3b-a+3-N}{2}\right) \left(
\int_{\mathbb{R}^{N}}\dfrac{|u|^{2}}{|x|^{a+b+1}}dx\right) \nonumber\\
&  =\frac{1}{2}\lambda^{b-a+1}\int_{\mathbb{R}^{N}}\dfrac{1}{|x|^{2N-2b-4}%
}\left\vert \nabla\left(  u\left\vert x\right\vert ^{N-2b-2}e^{\frac
{|x|^{b+1-a}}{\left(  b+1-a\right)  \lambda^{b-a+1}}}\right)  \right\vert
^{2}e^{-\frac{2|x|^{b+1-a}}{\left(  b+1-a\right)  \lambda^{b-a+1}}}dx\\
&\geq \frac{ C_{4,PI}(a,b,N)}{2}\inf_{c \in \mathbb{R}}\displaystyle\int_{\mathbb{R}^N}\dfrac{\left|u-ce^{-\frac{|x|^{b+1-a}}{(b+1-a)\lambda^{b-a+1}}}|x|^{2b+2-N}\right|^2}{|x|^{1+b+a}}dx\\
&\geq \frac{ C_{4,PI}(a,b,N)}{2}\inf_{c \in \mathbb{R}, \lambda>0}\displaystyle\int_{\mathbb{R}^N}\dfrac{\left|u-ce^{-\frac{\lambda|x|^{b+1-a}}{(b+1-a)}}|x|^{2b+2-N}\right|^2}{|x|^{1+b+a}}dx.
\end{align*}
\end{proof}
\section{Weighted $L^p$-Poincar\'{e} inequalities with Gaussian type measures}
\begin{theorem}\label{LpCase1Thm}
     Assume $1 + b - a > 0$ and $b \leq \frac{N-p}{p}$ with $p >1$. There exists a constant \(C_{PI}(a,b,p,N) > 0\) such that
    \begin{equation}
        \int_{\mathbb{R}^N} |\nabla v|^p \frac{e^{-\frac{p|x|^{1+b-a}}{1+b-a}}}{|x|^{pb}} \, dx \geq C_{PI}(a, b, p, N) \inf_{c \in \mathbb{R}} \int_{\mathbb{R}^N} |v-c|^p \frac{e^{-\frac{p|x|^{1+b-a}}{1+b-a}}}{|x|^{1+b+(p-1)a}}  \, dx.
    \end{equation}
\end{theorem}
\begin{proof}
    {\bf Step 1.} First, we consider the radial case. It suffices to show that
    $$\int_0^\infty |v'|^pe^{-\frac{pr^{1+b-a}}{1+b-a}}r^{N-1-pb}dr \geq C_{\text{radPI}}(a, b, p, N)\inf_{c \in \mathbb{R}}\int_0^\infty |v-c|^p e^{-\frac{pr^{1+b-a}}{1+b-a}} r^{N-1-(p-1)a-b-1}dr.$$ 
    To do this, setting $\dfrac{r^{1+b-a}}{1+b-a}=s^{\frac{p}{p-1}}$, then $r=(1+b-a)^{\frac{1}{1+b-a}}s^{\frac{p}{(p-1)(1+b-a)}}$ and
    $$dr=\frac{p(1+b-a)^{\frac{1}{1+b-a}}}{(p-1)(1+b-a)}s^{\frac{p}{(p-1)(1+b-a)}-1}ds.$$
    Hence,
    \begin{align*}
        &\int_0^\infty |v'|^pe^{-\frac{pr^{1+b-a}}{1+b-a}}r^{N-1-pb}dr\\
        &=\dfrac{p}{p-1}\int_0^\infty \left|v'\left((1+b-a)^{\frac{1}{1+b-a}}s^{\frac{p}{(p-1)(1+b-a)}}\right)\right|^pe^{-ps^{p/(p-1)}}(1+b-a)^{\frac{N+a-(p+1)b-1}{1+b-a}}s^{\frac{pN-p^2b}{(p-1)(1+b-a)}-1}ds
    \end{align*}
    Putting $w(s)=v\left((1+b-a)^{\frac{1}{1+b-a}}s^{\frac{p}{(p-1)(1+b-a)}}\right)$, then
    $$w'(s)=(1+b-a)^{\frac{1}{1+b-a}-1}\frac{p}{(p-1)}s^{\frac{p}{(p-1)(1+b-a)}-1}v'\left((1+b-a)^{\frac{1}{1+b-a}}s^{\frac{p}{(p-1)(1+b-a)}}\right),$$
    which implies that
    \begin{align*}
        &\int_0^\infty |v'|^pe^{-\frac{pr^{1+b-a}}{1+b-a}}r^{N-1-pb}dr\\
        &=\left(\dfrac{p-1}{p}\right)^{p-1}(1+b-a)^{\frac{N-(p-1)a-b-1}{1+b-a}}\int_0^\infty \left|w'(s)\right|^pe^{-ps^{p/(p-1)}}s^{\frac{p(N+a-b-1)-ap^2}{(p-1)(1+b-a)}-1}ds.
    \end{align*}
    Moreover,
    \begin{align*}
        &\int_0^\infty |v-c|^p e^{-\frac{pr^{1+b-a}}{1+b-a}} r^{N-1-(p-1)a-b-1}dr\\
        &=\dfrac{p}{p-1}(1+b-a)^{\frac{N-(p-2)a-2b-2}{1+b-a}}\int_0^\infty |w(s)-c|^p e^{-ps^{p/(p-1)}}s^{\frac{p(N+a-b-1)-p^2a}{(p-1)(1+b-a)}-1}ds.
    \end{align*}
    Let us consider $$f(s)=e^{-ps^{p/(p-1)}}s^{\frac{p(N+a-b-1)-p^2a}{(p-1)(1+b-a)}-1}=\exp{\left((A-1)\ln(s)-ps^{p/(p-1)}\right)},$$ where $A=\frac{p(N+a-b-1)-p^2a}{(p-1)(1+b-a)}$. Since $1+b-a>0$ and $b \leq \frac{N-p}{p}$,  we have $A \geq 1$ and the function 
    $$(A-1)\ln(s)-ps^{p/(p-1)}$$
    is concave. Using the Poincar\'{e} inequality with log-concave measure, we get
    \begin{align*}
        &\int_0^\infty |v'|^pe^{-\frac{pr^{1+b-a}}{1+b-a}}r^{N-1-pb}dr\\
        &=\left(\dfrac{p-1}{p}\right)^{p-1}(1+b-a)^{\frac{N-(p-1)a-b-1}{1+b-a}}\int_0^\infty \left|w'(s)\right|^pe^{-ps^{p/(p-1)}}s^{\frac{p(N+a-b-1)-ap^2}{(p-1)(1+b-a)}-1}ds\\
        &\geq \left(\dfrac{p-1}{p}\right)^{p-1}(1+b-a)^{\frac{N-(p-1)a-b-1}{1+b-a}} C_{lcP}\inf_{c \in \mathbb{R}} \int_0^\infty |w(s)-c|^p e^{-ps^{p/(p-1)}}s^{\frac{p(N+a-b-1)-ap^2}{(p-1)(1+b-a)}-1}ds\\
        &= (1+b-a)\left(\dfrac{p-1}{p}\right)^{p} C_{lcP} \inf_{c \in \mathbb{R}}\int_0^\infty |v-c|^p e^{-\frac{pr^{1+b-a}}{1+b-a}} r^{N-1-(p-1)a-b-1}dr,
    \end{align*}
    as desired.

    {\bf Step 2.} For a general function $v$, WLOG, we can assume that $v$ can be expressed in the form of $v=v_0+w$ where $v_0$ is radial and $w$ is orthogonal to all radial functions and $\int_{\mathbb{S}^{N-1}}wd\sigma=0$. More clearly, we can choose $v_0(r)=\frac{1}{|\mathbb{S}^{N-1}|}\int_{\mathbb{S}^{N-1}}v(r\sigma)d\sigma$. Here, we need to split into two ranges of $p$, $p \geq 2$ and $1<p<2$. 
    
    {\bf Step 3.} For $p \geq 2$, using \cite[Theorem $4.2$]{L90}, we obtain
    \begin{align*}
        |\nabla v|^p=\left(|\partial_r v|^2+\frac{1}{r^2}|\nabla_{\mathbb{S}^{N-1}}v|^2\right)^{p/2}&\geq |\partial_rv|^p+\dfrac{1}{r^p}|\nabla_{\mathbb{S}^{N-1}}v|^p\\
        &=|\partial_r(v_0+w)|^p+\dfrac{1}{r^p}|\nabla_{\mathbb{S}^{N-1}}w|^p\\
        &\geq |\partial_r v_0|^p+p|\partial_r v_0|^{p-2}\partial_rv_0\partial_r w+C_p|\partial_r w|^p+\dfrac{1}{r^p}|\nabla_{\mathbb{S}^{N-1}}w|^p,
    \end{align*}
    where $C_p=\frac{1}{2^{p-1}-1}.$
    Since
    \begin{align*}
        \int_{\mathbb{R}^N}|\partial_r v_0|^{p-2}\partial_rv_0\partial_r w \frac{e^{-\frac{p|x|^{1+b-a}}{1+b-a}}}{|x|^{pb}} dx&=\int_{\mathbb{S}^{N-1}}\int_0^\infty  |\partial_r v_0|^{p-2}\partial_rv_0\partial_r w e^{\frac{-pr^{1+b-a}}{1+b-a}}r^{N-1-pb}drd\sigma\\
        &=-\int_{\mathbb{S}^{N-1}}\int_0^\infty  w \partial_r \left[|\partial_r v_0|^{p-2}\partial_rv_0 e^{\frac{-pr^{1+b-a}}{1+b-a}}r^{N-1-pb}\right]drd\sigma\\
        &=0
    \end{align*}
    and, by the $L^p$-Poincar\'{e} inequality on sphere
    \begin{align*}
        \int_{\mathbb{R}^N} \dfrac{1}{|x|^p}|\nabla_{\mathbb{S}^{N-1}}w|^p \frac{e^{-\frac{p|x|^{1+b-a}}{1+b-a}}}{|x|^{pb}} \, dx&=\int_0^\infty \left(\int_{\mathbb{S}^{N-1}}|\nabla_{\mathbb{S}^{N-1}}w|^pd\sigma\right) r^{N-1-p-pb}e^{-\frac{pr^{1+b-a}}{1+b-a}}dr\\
        &\geq C_{Psph}\int_0^\infty \left(\int_{\mathbb{S}^{N-1}}|w|^pd\sigma\right) r^{N-1-p-pb}e^{-\frac{pr^{1+b-a}}{1+b-a}}dr\\
        &=C_{Psph}\int_{\mathbb{R}^N} \dfrac{1}{|x|^p}|w|^p \frac{e^{-\frac{p|x|^{1+b-a}}{1+b-a}}}{|x|^{pb}} \, dx,
    \end{align*}
    we get
    \begin{align*}
        \int_{\mathbb{R}^N} |\nabla v|^p \frac{e^{-\frac{p|x|^{1+b-a}}{1+b-a}}}{|x|^{pb}} \, dx &\geq \int_{\mathbb{R}^N} \left(|\partial_r v_0|^p+C_p|\partial_r w|^p\right) \frac{e^{-\frac{p|x|^{1+b-a}}{1+b-a}}}{|x|^{pb}} \, dx\\
        &+ C_{Psph}\int_{\mathbb{R}^N} \dfrac{1}{|x|^p}|w|^p \frac{e^{-\frac{p|x|^{1+b-a}}{1+b-a}}}{|x|^{pb}} \, dx\\
        &\geq \int_{\mathbb{R}^N} |\partial_r v_0|^p \frac{e^{-\frac{p|x|^{1+b-a}}{1+b-a}}}{|x|^{pb}} \, dx\\
        &+\min\{C_p, C_{Psph}\}\int_{\mathbb{R}^N}\left(|\partial_r w|^p+\frac{|w|^p}{|x|^p}\right)\frac{e^{-\frac{p|x|^{1+b-a}}{1+b-a}}}{|x|^{pb}} \, dx\\
        & \geq C_{radPI}(a,b,p,N)\inf_{c \in \mathbb{R}} \int_{\mathbb{R}^N} |v_0-c|^p \frac{e^{-\frac{p|x|^{1+b-a}}{1+b-a}}}{|x|^{1+b+(p-1)a}}  \, dx\\
        &+\min\{C_p, C_{Psph}\}\int_{\mathbb{R}^N}\left(|\partial_r w|^p+\frac{|w|^p}{|x|^p}\right)\frac{e^{-\frac{p|x|^{1+b-a}}{1+b-a}}}{|x|^{pb}} \, dx.
    \end{align*}
    
    {\bf Step 4.} Then, using the elementary inequality 
    $$a^p+b^p \geq \frac{(a+b)^p}{2^{p-1}}$$
    for $p\geq 1$, it is enough to prove the following lemma:
    \begin{lemma}
        $$\int_{\mathbb{R}^N}\left(|\partial_r w|^p+\frac{|w|^p}{|x|^p}\right)\frac{e^{-\frac{p|x|^{1+b-a}}{1+b-a}}}{|x|^{pb}} \, dx\geq C_1 \int_{\mathbb{R}^N} |w|^p \frac{e^{-\frac{p|x|^{1+b-a}}{1+b-a}}}{|x|^{1+b+(p-1)a}}  \, dx.$$
    \end{lemma}
    To do this, we also consider two cases: the radial case and general case. First, if $w$ is radial, we will use a contradiction argument. For the sake of contradiction, we can assume that there exists a sequence of radial functions $(w_k)_{k=1}^\infty$ such that
    $$\int_{\mathbb{R}^N} |w_k|^p \frac{e^{-\frac{p|x|^{1+b-a}}{1+b-a}}}{|x|^{1+b+(p-1)a}}  \, dx=1$$
    and 
    $$\lim_{k \rightarrow \infty} \int_{\mathbb{R}^N}\left(|\partial_r w_k|^p+\frac{|w_k|^p}{|x|^p}\right)\frac{e^{-\frac{p|x|^{1+b-a}}{1+b-a}}}{|x|^{pb}} \, dx =0.$$
    Then,
    $$\lim_{k \rightarrow \infty} \int_{\mathbb{R}^N}|\partial_r w_k|^p\frac{e^{-\frac{p|x|^{1+b-a}}{1+b-a}}}{|x|^{pb}} \, dx =0.$$
    and 
    $$\lim_{k \rightarrow \infty} \int_{\mathbb{R}^N}\frac{|w_k|^p}{|x|^p}\frac{e^{-\frac{p|x|^{1+b-a}}{1+b-a}}}{|x|^{pb}} \, dx =0.$$
    
    Moreover, from the radial case of the Poincar\'{e} inequality, we can choose a sequence $(c_k)_{k=1}^\infty$ satisfying
    $$\int_{\mathbb{R}^N}|\partial_r w_k|^p\frac{e^{-\frac{p|x|^{1+b-a}}{1+b-a}}}{|x|^{pb}} \, dx \geq \frac{1}{2}C_{radPI}(a,b,p,N)\int_{\mathbb{R}^N} |w_k-c_k|^p \frac{e^{-\frac{p|x|^{1+b-a}}{1+b-a}}}{|x|^{1+b+(p-1)a}}  \, dx,$$
    which implies that
    $$\lim_{k \rightarrow \infty} \int_{\mathbb{R}^N} |w_k-c_k|^p \frac{e^{-\frac{p|x|^{1+b-a}}{1+b-a}}}{|x|^{1+b+(p-1)a}}  \, dx=0.$$
    Therefore, up to a subsequence, we can derive that $c_k \rightarrow 0$ as $k \rightarrow 0$. From $(p-1)a+b+1<(b+1)p$ for any $p>1$ and $b+1-a>0$, it follows that
    \begin{align*}
        1&=\int_{\mathbb{R}^N} |w_k|^p \frac{e^{-\frac{p|x|^{1+b-a}}{1+b-a}}}{|x|^{1+b+(p-1)a}}  \, dx\\
        &=\int_{B_R} |w_k|^p \frac{e^{-\frac{p|x|^{1+b-a}}{1+b-a}}}{|x|^{1+b+(p-1)a}}  \, dx+\int_{B_R^c} |w_k|^p \frac{e^{-\frac{p|x|^{1+b-a}}{1+b-a}}}{|x|^{1+b+(p-1)a}}  \, dx\\
        &\leq \int_{B_R} \frac{|w_k|^p}{|x|^{(b+1)p}} \frac{e^{-\frac{p|x|^{1+b-a}}{1+b-a}}|x|^{(b+1)p}}{|x|^{1+b+(p-1)a}}  \, dx\\
        &+2^{p-1}\int_{B_R^c} |w_k-c_k|^p \frac{e^{-\frac{p|x|^{1+b-a}}{1+b-a}}}{|x|^{1+b+(p-1)a}}  \, dx+2^{p-1}c_k^p\int_{B_R^c} \frac{e^{-\frac{p|x|^{1+b-a}}{1+b-a}}}{|x|^{1+b+(p-1)a}}  \, dx\\
        &\leq R^{(b+1)p-(p-1)a-b-1}\int_{B_R} \frac{|w_k|^p}{|x|^{(b+1)p}} e^{-\frac{p|x|^{1+b-a}}{1+b-a}}  \, dx\\
        &+2^{p-1}\int_{B_R^c} |w_k-c_k|^p \frac{e^{-\frac{p|x|^{1+b-a}}{1+b-a}}}{|x|^{1+b+(p-1)a}}  \, dx+2^{p-1}c_k^p\int_{B_R^c} \frac{e^{-\frac{p|x|^{1+b-a}}{1+b-a}}}{|x|^{1+b+(p-1)a}}  \, dx,
    \end{align*}
    which tends to $0$ as $k \rightarrow \infty$, where $B_R$ is the Euclidean ball centered at the origin with radius $R>0$. This is a contradiction.
    
    Next, for a general function $w$, fixing $\sigma \in \mathbb{S}^{N-1}$ and we define a radial function $v(x)=w(|x|\sigma)$. Then
    $$|\nabla v(x)|^p=|v'(r)|^p=|\nabla w(|x|\sigma)\cdot \sigma|^p=|\partial_r w(r\sigma)|^p.$$ Applying the radial case for the function $v$, we have
    \begin{align*}
        \int_0^\infty\left(|v'(r)|^p+\frac{|v|^p}{r^p}\right)\frac{e^{-\frac{pr^{1+b-a}}{1+b-a}}}{r^{pb}} \,r^{N-1} dr&=\int_0^\infty\left(|\nabla w(|x|\sigma)\cdot \sigma|^p+\frac{|v|^p}{r^p}\right)\frac{e^{-\frac{pr^{1+b-a}}{1+b-a}}}{r^{pb}} \,r^{N-1} dr\\
        &\geq C_1 \int_0^\infty |v|^p \frac{e^{-\frac{pr^{1+b-a}}{1+b-a}}}{r^{1+b+(p-1)a}} r^{N-1} dr.
    \end{align*}
    Integrating over the sphere $\mathbb{S}^{N-1}$ allows us
    \begin{align*}
        \int_{\mathbb{R}^N}\left(|\partial_r w|^p+\frac{|w|^p}{|x|^p}\right)\frac{e^{-\frac{p|x|^{1+b-a}}{1+b-a}}}{|x|^{pb}} \, dx&=\int_{\mathbb{S}^{N-1}}\int_0^\infty\left(|\nabla w(|x|\sigma)\cdot \sigma|^p+\frac{|v|^p}{r^p}\right)\frac{e^{-\frac{pr^{1+b-a}}{1+b-a}}}{r^{pb}} \,r^{N-1} drd\sigma\\
        &\geq C_1 \int_{\mathbb{S}^{N-1}}\int_0^\infty |v|^p \frac{e^{-\frac{pr^{1+b-a}}{1+b-a}}}{r^{1+b+(p-1)a}} r^{N-1} drd\sigma\\
        &=C_1 \int_{\mathbb{R}^N} |w|^p \frac{e^{-\frac{p|x|^{1+b-a}}{1+b-a}}}{|x|^{1+b+(p-1)a}}  \, dx,
    \end{align*}
    as desired. Thus, we complete the proof of Theorem \ref{LpCase1Thm} for $p \geq 2$ with
    $$\overline{C}_PI(a, b, p, N)=\dfrac{1}{2^{p-1}}\min\{C_{radPI}(a, b, p, N), \min\{C_p, C_{Psph}\}C_1\}.$$

    {\bf Step 5.} For $1<p<2$, similarly,
    \begin{align*}
        \int_{\mathbb{R}^N} \dfrac{1}{|x|^p}|\nabla_{\mathbb{S}^{N-1}}w|^p \frac{e^{-\frac{p|x|^{1+b-a}}{1+b-a}}}{|x|^{pb}} \, dx\geq C_{Psph}\int_{\mathbb{R}^N} \dfrac{1}{|x|^p}|w|^p \frac{e^{-\frac{p|x|^{1+b-a}}{1+b-a}}}{|x|^{pb}} \, dx,
    \end{align*}
    and
    \begin{align*}
        \int_{\mathbb{R}^N}|\partial_r v_0|^{p-2}\partial_rv_0\partial_r w \frac{e^{-\frac{p|x|^{1+b-a}}{1+b-a}}}{|x|^{pb}} dx=0.
    \end{align*}
    Then, from $|a+b|^p \geq |a|^p+p|a|^{p-2}ab$ with $p>1$,
    $$\int_{\mathbb{R}^N}|\partial_r v|^{p}\frac{e^{-\frac{p|x|^{1+b-a}}{1+b-a}}}{|x|^{pb}} dx \geq \int_{\mathbb{R}^N}|\partial_r v_0|^{p}\frac{e^{-\frac{p|x|^{1+b-a}}{1+b-a}}}{|x|^{pb}} dx,$$
    which implies that
    \begin{align*}
        \int_{\mathbb{R}^N}|\partial_r w|^{p}\frac{e^{-\frac{p|x|^{1+b-a}}{1+b-a}}}{|x|^{pb}} dx &\leq 2^{p-1}\left(\int_{\mathbb{R}^N}|\partial_r v|^{p}\frac{e^{-\frac{p|x|^{1+b-a}}{1+b-a}}}{|x|^{pb}} dx+\int_{\mathbb{R}^N}|\partial_r v_0|^{p}\frac{e^{-\frac{p|x|^{1+b-a}}{1+b-a}}}{|x|^{pb}} dx\right)\\
        &\leq 2^p \int_{\mathbb{R}^N}|\partial_r v|^{p}\frac{e^{-\frac{p|x|^{1+b-a}}{1+b-a}}}{|x|^{pb}} dx.
    \end{align*}
    Thus,
    \begin{align*}
        &\int_{\mathbb{R}^N} |\nabla v|^p \frac{e^{-\frac{p|x|^{1+b-a}}{1+b-a}}}{|x|^{pb}} \, dx \\
        &\geq 2^{p/2-1}\int_{\mathbb{R}^N} \left(|\partial_r v|^p+\dfrac{1}{|x|^p}\left|\nabla_{\mathbb{S}^{N-1}w}\right|^p\right) \frac{e^{-\frac{p|x|^{1+b-a}}{1+b-a}}}{|x|^{pb}} \, dx\\ 
        &\geq 2^{p/2-2}\int_{\mathbb{R}^N} |\partial_r v_0|^p \frac{e^{-\frac{p|x|^{1+b-a}}{1+b-a}}}{|x|^{pb}} \, dx\\
        &+\dfrac{1}{2^{p/2+2}}\int_{\mathbb{R}^N} |\partial_r w|^p \frac{e^{-\frac{p|x|^{1+b-a}}{1+b-a}}}{|x|^{pb}} \, dx+2^{p/2-1}C_{Psph}\int_{\mathbb{R}^N} \dfrac{1}{|x|^p}|w|^p \frac{e^{-\frac{p|x|^{1+b-a}}{1+b-a}}}{|x|^{pb}} \, dx\\
        & \geq 2^{p/2-2}C_{radPI}(a,b,p,N)\inf_{c \in \mathbb{R}} \int_{\mathbb{R}^N} |v_0-c|^p \frac{e^{-\frac{p|x|^{1+b-a}}{1+b-a}}}{|x|^{1+b+(p-1)a}}  \, dx\\
        &+\min\left\{\dfrac{1}{2^{p/2+2}}, 2^{p/2-1}C_{Psph}\right\}\int_{\mathbb{R}^N}\left(|\partial_r w|^p+\frac{|w|^p}{|x|^p}\right)\frac{e^{-\frac{p|x|^{1+b-a}}{1+b-a}}}{|x|^{pb}} \, dx.
    \end{align*}

    {\bf Step 6.} Since $$(p-1)(b+1)+a<(p-1)a+b+1<p(b+1)$$ for all $1<p<2$ and $1+b-a>0$, it is enough to prove 
    \begin{equation}\label{eq 3.2}
        \int_{\mathbb{R}^N}\left(|\partial_r w|^p+\frac{|w|^p}{|x|^p}\right)\frac{e^{-\frac{p|x|^{1+b-a}}{1+b-a}}}{|x|^{pb}} \, dx \geq C_2 \int_{\mathbb{R}^N} |w|^pe^{-\frac{p|x|^{1+b-a}}{1+b-a}} |x|^{-(p-1)(b+1)-a}dx. 
    \end{equation}
    If so, then
    \begin{align*}
        &\int_{\mathbb{R}^N}\left(|\partial_r w|^p+\frac{|w|^p}{|x|^p}\right)\frac{e^{-\frac{p|x|^{1+b-a}}{1+b-a}}}{|x|^{pb}} \, dx \\&\geq \frac{\min\{C_2, 1\}}{2} \int_{\mathbb{R}^N} |w|^pe^{-\frac{p|x|^{1+b-a}}{1+b-a}} \left(|x|^{-(p-1)(b+1)-a}+|x|^{-p(b+1)}\right)dx\\
        &\geq C_3\frac{\min\{C_2, 1\}}{2} \int_{\mathbb{R}^N} |w|^pe^{-\frac{p|x|^{1+b-a}}{1+b-a}} |x|^{-(p-1)a-b-1}dx,
    \end{align*}
    where $|x|^{-(p-1)(b+1)-a}+|x|^{-p(b+1)} \geq C_3 |x|^{-(p-1)a-b-1}$ for all $x \neq 0$.
    As a result, 
    \begin{align*}
        &\int_{\mathbb{R}^N} |\nabla v|^p \frac{e^{-\frac{p|x|^{1+b-a}}{1+b-a}}}{|x|^{pb}} \, dx\\
        & \geq 2^{p/2-2}C_{radPI}(a,b,p,N)\inf_{c \in \mathbb{R}} \int_{\mathbb{R}^N} |v_0-c|^p \frac{e^{-\frac{p|x|^{1+b-a}}{1+b-a}}}{|x|^{1+b+(p-1)a}}  \, dx\\
        &+\frac{C_3}{2}\min\left\{\dfrac{1}{2^{p/2+2}}, 2^{p/2-1}C_{Psph}\right\}\min\{C_2, 1\} \int_{\mathbb{R}^N} |w|^pe^{-\frac{p|x|^{1+b-a}}{1+b-a}} |x|^{-(p-1)a-b-1}dx\\
        &\geq C_{PI}(a,b,p,N)\inf_{c \in \mathbb{R}} \int_{\mathbb{R}^N} |v-c|^p \frac{e^{-\frac{p|x|^{1+b-a}}{1+b-a}}}{|x|^{1+b+(p-1)a}}  \, dx.
    \end{align*}
    Here, with $1<p<2$, 
    $$\tilde{C}_PI(a,b,p,N)=\dfrac{1}{2^{p-1}}\min\left\{2^{p/2-2}C_{radPI}(a,b,p,N), \frac{C_3}{2}\min\left\{\dfrac{1}{2^{p/2+2}}, 2^{p/2-1}C_{Psph}\right\}\min\{C_2, 1\}\right\}.$$

    {\bf Step 7.} To prove \eqref{eq 3.2}, by using Young's inequality with $q=\frac{p}{p-1}$ and integrating by parts, we have
    \begin{align*}
    \int_{\mathbb{R}^N}\left(\dfrac{1}{p}|\partial_r w|^p+\dfrac{1}{q}\frac{|w|^p}{|x|^p}\right)\frac{e^{-\frac{p|x|^{1+b-a}}{1+b-a}}}{|x|^{pb}} \, dx&\geq \int_{\mathbb{R}^N}\dfrac{1}{p} \dfrac{\partial_r(|w|^p)}{|x|^{p-1}}\frac{e^{-\frac{p|x|^{1+b-a}}{1+b-a}}}{|x|^{pb}} \, dx\\
    &=\dfrac{1}{p}\int_{\mathbb{S}^{N-1}}\left(\int_0^\infty \partial_r(|w|^p)e^{-p\frac{r^{1+b-a}}{1+b-a}}r^{N-pb-p}dr\right)d\sigma\\
    &=\int_{\mathbb{S}^{N-1}}\int_0^\infty |w|^p e^{-p\frac{r^{1+b-a}}{1+b-a}}r^{N-pb-p+b-a}drd\sigma\\
    &-\dfrac{N-pb-p}{p}\int_{\mathbb{S}^{N-1}}\int_0^\infty |w|^p e^{-p\frac{r^{1+b-a}}{1+b-a}}r^{N-1-pb-p}drd\sigma\\
    &=\int_{\mathbb{R}^N}|w|^pe^{-\frac{p|x|^{1+b-a}}{1+b-a}}|x|^{b-a-pb-p+1}dx\\
    &-\dfrac{N-pb-p}{p}\int_{\mathbb{R}^N}|w|^p e^{-\frac{p|x|^{1+b-a}}{1+b-a}} |x|^{-pb-p}dx.
    \end{align*}
    Hence,
    \begin{align*}
    &\int_{\mathbb{R}^N}|\partial_r w|^p\frac{e^{-\frac{p|x|^{1+b-a}}{1+b-a}}}{|x|^{pb}} \, dx+\left(N-pb-p+\frac{p}{q}\right)\int_{\mathbb{R}^N}|w|^p e^{-\frac{p|x|^{1+b-a}}{1+b-a}} |x|^{-pb-p}dx\\
    &\geq p\int_{\mathbb{R}^N}|w|^pe^{-\frac{p|x|^{1+b-a}}{1+b-a}}|x|^{b-a-pb-p+1}dx,
    \end{align*}
    which proves \eqref{eq 3.2} with 
    $$C_2=\dfrac{p}{\max\{1, N-pb-p+p/q\}}.$$
    To sum up, we complete the proof for all $p>1$ with
    $$C_{PI}(a, b, p, N)=\min\{\overline{C}_PI(a, b, p, N), \tilde{C}_PI(a, b, p, N)\}.$$
\end{proof}

As a consequence of the above theorem, we get the following weighted $L^p$-Poincar\'{e} type inequalities:
\begin{theorem}\label{LpCase2Thm}
    Assume $1 + b - a < 0$ and $b \geq \frac{N-p}{p}$ with $p >1$, we have
    \begin{equation}
        \int_{\mathbb{R}^N} |\nabla v|^p \frac{e^{\frac{p|x|^{1+b-a}}{1+b-a}}}{|x|^{pb}} \, dx \geq C_{2,PI}(a,b,p, N) \inf_{c \in \mathbb{R}} \int_{\mathbb{R}^N} |v-c|^p \frac{e^{\frac{p|x|^{1+b-a}}{1+b-a}}}{|x|^{1+b+(p-1)a}}  \, dx,
    \end{equation}
    where $C_{2, PI}(a,b,p,N)=C_{PI}\left(\frac{2}{p}N-a, \frac{2}{p}N-2-b,p, N\right).$
\end{theorem}
\begin{proof}
    From Theorem \ref{LpCase1Thm}, with $1+n-m>0$ and $n \leq\frac{N-p}{p}$,
    $$\int_{\mathbb{R}^N} |\nabla_x v(x)|^p \frac{e^{-\frac{p|x|^{1+n-m}}{1+n-m}}}{|x|^{pn}} \, dx \geq C_{PI}(m,n,p,N) \inf_{c \in \mathbb{R}} \int_{\mathbb{R}^N} |v(x)-c|^p \frac{e^{-\frac{p|x|^{1+n-m}}{1+n-m}}}{|x|^{1+n+(p-1)m}}  \, dx.$$
    By changing variable $x \mapsto \frac{y}{|y|^2}$, we have $y=\frac{x}{|x|^2}$, $dx=\frac{dy}{|y|^{2N}}$ and
    $$\left|\nabla_x v\left(\frac{y}{|y|^2}\right)\right|^p=\left|(\nabla v)\left(\frac{y}{|y|^2}\right)\right|^p.$$
    Then,
    \begin{align*}
        \int_{\mathbb{R}^N} |\nabla_x v(x)|^p \frac{e^{-\frac{p|x|^{1+n-m}}{1+n-m}}}{|x|^{pn}} \, dx&=\int_{\mathbb{R}^N} \left|(\nabla v)\left(\frac{y}{|y|^2}\right)\right|^p \frac{e^{-\frac{p|y|^{-(1+n-m)}}{1+n-m}}}{|y|^{2N-pn}} \, dy\\
        &\geq C_{PI}(m,n,p,N) \inf_{c \in \mathbb{R}} \int_{\mathbb{R}^N} \left|v\left(\frac{y}{|y|^2}\right)-c\right|^p \frac{e^{-\frac{p|y|^{-(1+n-m)}}{1+n-m}}}{|y|^{2N-(1+n+(p-1)m)}}  \, dy.
    \end{align*}
    Define $w(y)=v\left(\frac{y}{|y|^2}\right)$. Thus, 
    $$|\nabla_y w(y)|^p=\dfrac{1}{|y|^{2p}}\left|(\nabla v)\left(\frac{y}{|y|}\right)\right|^p,$$
    which implies that
    $$\int_{\mathbb{R}^N} |\nabla_y w(y)|^p \frac{e^{-\frac{p|y|^{-(1+n-m)}}{1+n-m}}}{|y|^{2N-pn-2p}} \, dy \geq C_{PI}(m,n,p,N) \inf_{c \in \mathbb{R}} \int_{\mathbb{R}^N} \left|w(y)-c\right|^p \frac{e^{-\frac{p|y|^{-(1+n-m)}}{1+n-m}}}{|y|^{2N-(1+n+(p-1)m)}}  \, dy.$$
    Setting $pb:=2N-pn-2p$, and $a:=\frac{2}{p}N-m$, then $b=\frac{2}{p}N-n-2 \geq \frac{N-p}{p}$ and $1+b-a=m-n-1<0$. As a result,
    $$\int_{\mathbb{R}^N} |\nabla_y w(y)|^p \frac{e^{\frac{p|y|^{1+b-a}}{1+b-a}}}{|y|^{pb}} \, dy \geq C_{PI}\left(\frac{2}{p}N-a, \frac{2}{p}N-2-b, p, N\right) \inf_{c \in \mathbb{R}} \int_{\mathbb{R}^N} \left|w(y)-c\right|^p\frac{e^{\frac{p|y|^{1+b-a}}{1+b-a}}}{|y|^{(p-1)a+b+1}}  \, dy,$$
    for all $b \geq\frac{N-p}{p}$ and $1+b-a<0$, as desired.
\end{proof}

\begin{theorem}\label{LpCase3Thm}
    Assume $1 + b - a < 0$ and $b \leq \frac{N-p}{p}$ with $p >1$, we have
    \begin{equation}
        \int_{\mathbb{R}^N} |\nabla v|^p \frac{e^{\frac{p|x|^{1+b-a}}{1+b-a}}}{|x|^{2N-pb-2p}} \, dx \geq C_{3,PI}(a,b,p,N) \inf_{c \in \mathbb{R}} \int_{\mathbb{R}^N} |v-c|^p \frac{e^{\frac{p|x|^{1+b-a}}{1+b-a}}}{|x|^{2N+(p-1)a+(1-2p)b-2p+1}}  \, dx,
    \end{equation}
    where $C_{3,PI}(a,b,p,N)=C_{PI}(-a+2b+2,b,p,N).$
\end{theorem}
\begin{proof}
    From Theorem \ref{LpCase2Thm}, we get for $1 + n - m < 0$ and $n \geq \frac{N-p}{p}$
    \begin{equation}
        \int_{\mathbb{R}^N} |\nabla v|^p \frac{e^{\frac{p|x|^{1+n-m}}{1+n-m}}}{|x|^{pn}} \, dx \geq C_{2,PI}(m,n,p,N) \inf_{c \in \mathbb{R}} \int_{\mathbb{R}^N} |v-c|^p \frac{e^{\frac{p|x|^{1+n-m}}{1+n-m}}}{|x|^{1+n+(p-1)m}}  \, dx.
    \end{equation}
    By choosing $a, b$ such that $2N-pb-2p=pn$ and $1+b-a=1+n-m$, we have $b=\frac{2}{p}N-2-n \leq\frac{N-p}{p}$ and $1+b-a<0$. Hence,
    \begin{align*}
        &\int_{\mathbb{R}^N} |\nabla v|^p \frac{e^{\frac{p|x|^{1+n-m}}{1+n-m}}}{|x|^{2N-pb-2p}} \, dx\\
        &\geq C_{2,PI}\left(\frac{2}{p}N+a-2b-2, \frac{2}{p}N-b-2,p,N\right) \inf_{c \in \mathbb{R}} \int_{\mathbb{R}^N} |v-c|^p \frac{e^{\frac{p|x|^{1+b-a}}{1+b-a}}}{|x|^{2N+(p-1)a+(1-2p)b+1-2p}}  \, dx,
    \end{align*}
    which implies that $C_{3,PI}(a,b,p,N)=C_{PI}(-a+2b+2,b,p,N)$.
\end{proof}

\begin{theorem}\label{LpCase4Thm}
    Assume $1 + b - a > 0$ and $b \geq \frac{N-p}{p}$ with $p >1$, we have
    \begin{equation}
        \int_{\mathbb{R}^N} |\nabla v|^p \frac{e^{-\frac{p|x|^{1+b-a}}{1+b-a}}}{|x|^{2N-pb-2p}} \, dx \geq C_{4,PI}(a,b,p,N) \inf_{c \in \mathbb{R}} \int_{\mathbb{R}^N} |v-c|^p \frac{e^{-\frac{p|x|^{1+b-a}}{1+b-a}}}{|x|^{2N+(p-1)a+(1-2p)b-2p+1}}  \, dx,
    \end{equation}
    where $C_{4,PI}(a,b,p,N)=C_{PI}\left(\frac{2}{p}N+a-2b-2, \frac{2}{p}N-b-2,p,N\right)$.
\end{theorem}
\begin{proof}
    From Theorem \ref{LpCase1Thm}, we get for $1 + n - m > 0$ and $n \leq \frac{N-p}{p}$
    \begin{equation}
        \int_{\mathbb{R}^N} |\nabla v|^p \frac{e^{-\frac{p|x|^{1+n-m}}{1+n-m}}}{|x|^{pn}} \, dx \geq C_{PI}(m,n,p,N) \inf_{c \in \mathbb{R}} \int_{\mathbb{R}^N} |v-c|^p \frac{e^{-\frac{p|x|^{1+n-m}}{1+n-m}}}{|x|^{1+n+(p-1)m}}  \, dx.
    \end{equation}
    By choosing $a, b$ such that $2N-pb-2p=pn$ and $1+b-a=1+n-m$, we have $b=\frac{2}{p}N-2-n \geq\frac{N-p}{p}$ and $1+b-a>0$. Hence,
    \begin{align*}
        &\int_{\mathbb{R}^N} |\nabla v|^p \frac{e^{-\frac{p|x|^{1+b-a}}{1+b-a}}}{|x|^{2N-pb-2p}} \, dx\\
        &\geq C_{PI}\left(\frac{2}{p}N+a-2b-2, \frac{2}{p}N-b-2,p,N\right) \inf_{c \in \mathbb{R}} \int_{\mathbb{R}^N} |v-c|^p \frac{e^{-\frac{p|x|^{1+b-a}}{1+b-a}}}{|x|^{2N+(p-1)a+(1-2p)b+1-2p}}  \, dx,
    \end{align*}
    which implies that $C_{4,PI}(a,b,p, N)=C_{PI}\left(\frac{2}{p}N+a-2b-2, \frac{2}{p}N-b-2,p,N\right) $.
\end{proof}

\section{Stability estimates of the $L^p$-CKN inequalities}
In this section, we denote
$$\mathcal{R}_p(a,b)=|b|^p+(p-1)|a|^p-p|a|^{p-2}a\cdot b,$$
for all $p>1$, $a, b \in \mathbb{R}^N$ with $N \geq 1$.
From \cite[Theorem $4.2$]{DFLL23}, we have
\begin{theorem}
Let $b+1-a>0$ and $b\leq\dfrac{N-p}{p}$ with $p \geq 2$. Then for $u\in C_{0}^{\infty}\left(
\mathbb{R}^{N}\setminus\left\{  0\right\}  \right)  :$
\begin{align*}
&  \int_{\mathbb{R}^{N}}\dfrac{|\nabla u|^{p}}{|x|^{pb}}dx+(p-1)\int_{\mathbb{R}%
^{N}}\dfrac{|u|^{2}}{|x|^{2a}}dx-\left(  N-1-(p-1)a-b\right)  \int_{\mathbb{R}^{N}%
}\dfrac{|u|^{p}}{|x|^{(p-1)a+b+1}}dx\\
& \geq M_pC_{PI}(a,b,p,N)\inf_{c \in \mathbb{R}}\int_{\mathbb{R}^N}\left|u-ce^{-\frac{|x|^{1+b-a}}{1+b-a}}\right|^p\dfrac{dx}{|x|^{1+b+(p-1)a}}.
\end{align*}

Also, for $u\in C_{0}^{\infty}\left(  \mathbb{R}^{N}\setminus\left\{
0\right\}  \right)  \setminus\left\{  0\right\} $: 
\begin{align*}
&  \left(  \int_{\mathbb{R}^{N}}\dfrac{|\nabla u|^{p}}{|x|^{pb}}dx\right)
^{\frac{1}{p}}\left(  \int_{\mathbb{R}^{N}}\dfrac{|u|^{p}}{|x|^{pa}}dx\right)
^{\frac{p-1}{p}}-\left( \dfrac{N-(p-1)a-b-1}{p}\right) \left(  \int
_{\mathbb{R}^{N}}\dfrac{|u|^{p}}{|x|^{(p-1)a+b+1}}dx\right) \nonumber\\
&\geq \dfrac{M_p}{p}C_{PI}(a,b,p,N)\inf_{c \in \mathbb{R}, \lambda>0}\displaystyle\int_{\mathbb{R}^N}\dfrac{\left|u-ce^{-\frac{\lambda|x|^{b+1-a}}{(b+1-a)}}\right|^p}{|x|^{1+b+(p-1)a}}dx,
\end{align*}
where $M_p\in (0, 1]$ is defined in \cite[Lemma $1.1$]{DFLL23}.
\end{theorem}

\begin{proof}
The first assertion is the direct consequence of \cite[[Theorem $4.2$]{DFLL23} and Theorem \ref{LpCase1Thm}. Indeed,
\begin{align*}
&  \int_{\mathbb{R}^{N}}\dfrac{|\nabla u|^{p}}{|x|^{pb}}dx+(p-1)\int_{\mathbb{R}%
^{N}}\dfrac{|u|^{2}}{|x|^{2a}}dx-\left(  N-1-(p-1)a-b\right)  \int_{\mathbb{R}^{N}%
}\dfrac{|u|^{p}}{|x|^{(p-1)a+b+1}}dx\\
&=\int_{\mathbb{R}^N}\dfrac{1}{|x|^{pb}}\mathcal{R}_p(-u|x|^{b-1-a}x, \nabla u)dx\\
&\geq M_p \int_{\mathbb{R}^N}\dfrac{1}{|x|^{pb}}\left|\nabla u+u|x|^{b-1-a}x\right|^pdx\\
&=M_p \int_{\mathbb{R}^N}\dfrac{1}{|x|^{pb}}\left|\nabla \left(ue^{\frac{|x|^{b+1-a}}{b+1-a}}\right)\right|^pe^{-p\frac{|x|^{b+1-a}}{b+1-a}}dx\\
& \geq M_pC_{PI}(a,b,p,N)\inf_{c \in \mathbb{R}}\int_{\mathbb{R}^N}\left|u-ce^{-\frac{|x|^{1+b-a}}{1+b-a}}\right|^p\dfrac{dx}{|x|^{1+b+(p-1)a}}.
\end{align*}
For the second one, by applying a standard scaling argument to Theorem \ref{LpCase1Thm}, for any $\lambda>0$, we have
$$\lambda^{(p-1)(1+b-a)}\int_{\mathbb{R}^N} |\nabla v|^p \frac{e^{-\frac{p|x|^{1+b-a}}{(1+b-a)\lambda^{1+b-a}}}}{|x|^{pb}} \, dx \geq C_{PI}(a,b,p,N) \inf_{c \in \mathbb{R}} \int_{\mathbb{R}^N} |v-c|^p \frac{e^{-\frac{p|x|^{1+b-a}}{(1+b-a)\lambda^{1+b-a}}}}{|x|^{1+b+(p-1)a}}  \, dx.$$
Using the scaling argument again, with $\lambda=\left(\dfrac{\int_{\mathbb{R}^N}|u|^p/|x|^{pa}dx}{\int_{\mathbb{R}^N}|\nabla u|^p/|x|^{pb}dx}\right)^{1/(p(b+1-a))}$, we get
\begin{align*}
&  \left(  \int_{\mathbb{R}^{N}}\dfrac{|\nabla u|^{p}}{|x|^{pb}}dx\right)
^{\frac{1}{p}}\left(  \int_{\mathbb{R}^{N}}\dfrac{|u|^{p}}{|x|^{pa}}dx\right)
^{\frac{p-1}{p}}-\left( \dfrac{N-(p-1)a-b-1}{p}\right) \left(  \int
_{\mathbb{R}^{N}}\dfrac{|u|^{p}}{|x|^{(p-1)a+b+1}}dx\right) \nonumber\\
&  =\frac{1}{p}\lambda^{(p-1)(b-a+1)}\int_{\mathbb{R}^{N}}\frac{1}{|x|^{pb}%
}\mathcal{R}_p(-\lambda^{a-b-1}u|x|^{b-a-1}x, \nabla u)dx\\
&  \geq\frac{M_p}{p}\lambda^{(p-1)(b-a+1)}\int_{\mathbb{R}^{N}}\frac{1}{|x|^{pb}%
}\left\vert \nabla u+\lambda^{a-b-1}u|x|^{b-a-1}x  \right\vert ^{p}dx\\
&  =\frac{M_p}{p}\lambda^{(p-1)(b-a+1)}\int_{\mathbb{R}^{N}}\frac{1}{|x|^{pb}%
}\left\vert \nabla\left(  ue^{\frac{|x|^{b+1-a}}{\left(  b+1-a\right)
\lambda^{b-a+1}}}\right)  \right\vert ^{p}e^{-\frac{p|x|^{b+1-a}}{\left(
b+1-a\right)  \lambda^{b-a+1}}}dx\\
&\geq \frac{M_p}{p}C_{PI}(a,b,p,N)\inf_{c \in \mathbb{R}}\displaystyle\int_{\mathbb{R}^N}\dfrac{\left|u-ce^{-\frac{|x|^{b+1-a}}{(b+1-a)\lambda^{b-a+1}}}\right|^p}{|x|^{1+b+(p-1)a}}dx\\
&\geq \frac{M_p}{p}C_{PI}(a,b,p,N)\inf_{c \in \mathbb{R}, \lambda>0}\displaystyle\int_{\mathbb{R}^N}\dfrac{\left|u-ce^{-\frac{\lambda|x|^{b+1-a}}{(b+1-a)}}\right|^p}{|x|^{1+b+(p-1)a}}dx.
\end{align*}
\end{proof}

Similarly, from \cite[Theorem $4.3$]{DFLL23}, we have
\begin{theorem}
Let $b+1-a<0$ and $b\geq\dfrac{N-p}{p}$ with $p \geq 2$. Then for $u\in C_{0}^{\infty}\left(
\mathbb{R}^{N}\setminus\left\{  0\right\}  \right)  :$
\begin{align*}
&  \int_{\mathbb{R}^{N}}\dfrac{|\nabla u|^{p}}{|x|^{pb}}dx+(p-1)\int_{\mathbb{R}%
^{N}}\dfrac{|u|^{2}}{|x|^{2a}}dx-\left(  1+(p-1)a+b-N\right)  \int_{\mathbb{R}^{N}%
}\dfrac{|u|^{p}}{|x|^{(p-1)a+b+1}}dx\\
& \geq M_pC_{2,PI}(a,b,p,N)\inf_{c \in \mathbb{R}}\int_{\mathbb{R}^N}\left|u-ce^{\frac{|x|^{1+b-a}}{1+b-a}}\right|^p\dfrac{dx}{|x|^{1+b+(p-1)a}}.
\end{align*}

Also, for $u\in C_{0}^{\infty}\left(  \mathbb{R}^{N}\setminus\left\{
0\right\}  \right)  \setminus\left\{  0\right\} $: 
\begin{align*}
&  \left(  \int_{\mathbb{R}^{N}}\dfrac{|\nabla u|^{p}}{|x|^{pb}}dx\right)
^{\frac{1}{p}}\left(  \int_{\mathbb{R}^{N}}\dfrac{|u|^{p}}{|x|^{pa}}dx\right)
^{\frac{p-1}{p}}-\left( \dfrac{1+(p-1)a+b-N}{p}\right) \left(  \int
_{\mathbb{R}^{N}}\dfrac{|u|^{p}}{|x|^{(p-1)a+b+1}}dx\right) \nonumber\\
&\geq \dfrac{M_p}{p}C_{2,PI}(a,b,p,N)\inf_{c \in \mathbb{R}, \lambda>0}\displaystyle\int_{\mathbb{R}^N}\dfrac{\left|u-ce^{\frac{\lambda|x|^{b+1-a}}{(b+1-a)}}\right|^p}{|x|^{1+b+(p-1)a}}dx,
\end{align*}
where $M_p\in (0, 1]$ is defined in \cite[Lemma $1.1$]{DFLL23}.
\end{theorem}

\begin{proof}
The first assertion is the direct consequence of \cite[Theorem $4.3$]{DFLL23} and Theorem \ref{LpCase2Thm}. Indeed,
\begin{align*}
&  \int_{\mathbb{R}^{N}}\dfrac{|\nabla u|^{p}}{|x|^{pb}}dx+(p-1)\int_{\mathbb{R}%
^{N}}\dfrac{|u|^{2}}{|x|^{2a}}dx-\left(  1+(p-1)a+b-N\right)  \int_{\mathbb{R}^{N}%
}\dfrac{|u|^{p}}{|x|^{(p-1)a+b+1}}dx\\
&=\int_{\mathbb{R}^N}\dfrac{1}{|x|^{pb}}\mathcal{R}_p(u|x|^{b-1-a}x, \nabla u)dx\\
&\geq M_p \int_{\mathbb{R}^N}\dfrac{1}{|x|^{pb}}\left|\nabla u-u|x|^{b-1-a}x\right|^pdx\\
&=M_p \int_{\mathbb{R}^N}\dfrac{1}{|x|^{pb}}\left|\nabla \left(ue^{-\frac{|x|^{b+1-a}}{b+1-a}}\right)\right|^pe^{p\frac{|x|^{b+1-a}}{b+1-a}}dx\\
& \geq M_pC_{2,PI}(a,b,p,N)\inf_{c \in \mathbb{R}}\int_{\mathbb{R}^N}\left|u-ce^{\frac{|x|^{1+b-a}}{1+b-a}}\right|^p\dfrac{dx}{|x|^{1+b+(p-1)a}}.
\end{align*}
For the second one, by applying a standard scaling argument to Theorem \ref{LpCase2Thm}, for any $\lambda>0$, we have
$$\lambda^{(p-1)(1+b-a)}\int_{\mathbb{R}^N} |\nabla v|^p \frac{e^{\frac{p|x|^{1+b-a}}{(1+b-a)\lambda^{1+b-a}}}}{|x|^{pb}} \, dx \geq C_{\text{2,PI}}(a,b,p,N) \inf_{c \in \mathbb{R}} \int_{\mathbb{R}^N} |v-c|^p \frac{e^{\frac{p|x|^{1+b-a}}{(1+b-a)\lambda^{1+b-a}}}}{|x|^{1+b+(p-1)a}}  \, dx.$$
Using the scaling argument again, with $\lambda=\left(\dfrac{\int_{\mathbb{R}^N}|u|^p/|x|^{pa}dx}{\int_{\mathbb{R}^N}|\nabla u|^p/|x|^{pb}dx}\right)^{1/(p(b+1-a))}$, we get
\begin{align*}
&  \left(  \int_{\mathbb{R}^{N}}\dfrac{|\nabla u|^{p}}{|x|^{pb}}dx\right)
^{\frac{1}{p}}\left(  \int_{\mathbb{R}^{N}}\dfrac{|u|^{p}}{|x|^{pa}}dx\right)
^{\frac{p-1}{p}}-\left( \dfrac{1+(p-1)a+b-N}{p}\right) \left(  \int
_{\mathbb{R}^{N}}\dfrac{|u|^{p}}{|x|^{(p-1)a+b+1}}dx\right) \nonumber\\
&  =\frac{1}{p}\lambda^{(p-1)(b-a+1)}\int_{\mathbb{R}^{N}}\frac{1}{|x|^{pb}%
}\mathcal{R}_p(\lambda^{a-b-1}u|x|^{b-a-1}x, \nabla u)dx\\
&  \geq\frac{M_p}{p}\lambda^{(p-1)(b-a+1)}\int_{\mathbb{R}^{N}}\frac{1}{|x|^{pb}%
}\left\vert \nabla u-\lambda^{a-b-1}u|x|^{b-a-1}x  \right\vert ^{p}dx\\
&  =\frac{M_p}{p}\lambda^{(p-1)(b-a+1)}\int_{\mathbb{R}^{N}}\frac{1}{|x|^{pb}%
}\left\vert \nabla\left(  ue^{-\frac{|x|^{b+1-a}}{\left(  b+1-a\right)
\lambda^{b-a+1}}}\right)  \right\vert ^{p}e^{\frac{p|x|^{b+1-a}}{\left(
b+1-a\right)  \lambda^{b-a+1}}}dx\\
&\geq \frac{M_p}{p}C_{2,PI}(a,b,p,N)\inf_{c \in \mathbb{R}}\displaystyle\int_{\mathbb{R}^N}\dfrac{\left|u-ce^{\frac{|x|^{b+1-a}}{(b+1-a)\lambda^{b-a+1}}}\right|^p}{|x|^{1+b+(p-1)a}}dx\\
&\geq \frac{M_p}{p}C_{2,PI}(a,b,p,N)\inf_{c \in \mathbb{R}, \lambda>0}\displaystyle\int_{\mathbb{R}^N}\dfrac{\left|u-ce^{\frac{\lambda|x|^{b+1-a}}{(b+1-a)}}\right|^p}{|x|^{1+b+(p-1)a}}dx.
\end{align*}
\end{proof}

\noindent\textbf{Acknowledgements.}  A. X. Do and G. Lu were partially supported by grants from
the Simons Foundation. N. Lam was partially supported by an NSERC
Discovery Grant.  V. H. Nguyen was supported by Vietnam National Foundation for Science and Technology Development (NAFOSTED)[101.02-2025.33].

\end{document}